\documentclass[12pt]{amsart}

\usepackage{verbatim}
\usepackage{fullpage}
\usepackage{amsfonts}
\usepackage{enumerate}
\usepackage{graphicx}
\usepackage{amsmath, amsthm, amssymb, bbm}
\usepackage{fancybox}
\usepackage{xcolor}

\usepackage{tikz,varwidth}
\usetikzlibrary{shapes.geometric, arrows}
\tikzstyle{stop} = [rectangle, rounded corners, minimum width=3cm, minimum height=1cm,text centered,  text width=4.5cm, draw=black, fill=green!30]
\tikzstyle{no} = [rectangle, rounded corners, minimum width=3cm, minimum height=1cm,text centered,  text width=4.5cm, draw=black, fill=red!30]
\tikzstyle{kernel} = [rectangle, rounded corners, minimum width=3cm, minimum height=1cm,text centered,  text width=4.5cm, draw=black, fill=yellow!30]
\tikzstyle{io} = [trapezium, trapezium left angle=70, trapezium right angle=110, minimum width=1cm, minimum height=1cm, text width=10cm, text centered, draw=black]
\tikzstyle{process} = [rectangle, minimum width=3cm, minimum height=1cm, text centered, text width = 4.5cm, draw=black]
\tikzstyle{decision} = [diamond, draw=black, text width=6em, text badly centered, inner sep=1pt]
\tikzstyle{empty} = []
\tikzstyle{arrow} = [thick,->]

\newcommand\iprec{\mathrel{\ooalign{$\prec$\cr
 \,\raise0.85ex\hbox{\scriptsize$\circ$}\cr}}}

\renewcommand{\tilde}{\widetilde}

\newcommand{\R}{\mathbb{R}}

\newcommand{\N}{\mathbb{N}}

\newcommand{\Hil}{\mathcal{H}}

\newcommand{\eps}{\varepsilon}
\newcommand{\la}{\lambda}

\DeclareMathOperator{\lspan}{span}
\DeclareMathOperator{\tr}{tr}

\DeclareMathOperator{\sgn}{sgn}
\providecommand{\abs}[1]{\lvert#1\rvert}

\newcounter{Theorem}

\numberwithin{equation}{section}
\numberwithin{Theorem}{section}

\theoremstyle{plain} 
\newtheorem{thm}[Theorem]{Theorem}
\newtheorem{cor}[Theorem]{Corollary}
\newtheorem{lem}[Theorem]{Lemma}
\newtheorem{prop}[Theorem]{Proposition}

\theoremstyle{definition}
\newtheorem{defn}[Theorem]{Definition}
\newtheorem{problem}[Theorem]{Problem}

\theoremstyle{remark}
\newtheorem{remark}[Theorem]{Remark}

\begin{document}
\title[Diagonals of self-adjoint operators]{Diagonals of self-adjoint operators I: 
\\
compact operators}

\begin{abstract}
Given a self-adjoint operator $T$ on a separable infinite-dimensional Hilbert space we study the problem of characterizing the set $\mathcal D(T)$ of all possible diagonals of $T$. For compact operators $T$, we give a complete characterization of diagonals modulo the kernel of $T$. That is, we characterize $\mathcal D(T)$ for the class of operators sharing the same nonzero eigenvalues (with multiplicities) as $T$. Moreover, we determine $\mathcal D(T)$ for a fixed compact operator $T$, modulo the kernel problem for positive compact operators with finite-dimensional kernel. 


Our results generalize a characterization of diagonals of trace class positive operators by Arveson and Kadison \cite{ak} and diagonals of compact positive operators by Kaftal and Weiss \cite{kw} and Loreaux and  Weiss \cite{lw}.
The proof uses the technique of diagonal-to-diagonal results, which was pioneered in the earlier joint work of the authors with Siudeja \cite{unbound}.

\end{abstract}

\author{Marcin Bownik}
\author{John Jasper}

\address{Department of Mathematics, University of Oregon, Eugene, OR 97403--1222, USA}

\email{mbownik@uoregon.edu}

\address{Department of Mathematics \& Statistics, Air Force Institute of Technology, Wright Patterson AFB, OH 45433, USA}

\email{john.jasper@afit.edu}

\keywords{diagonals of self-adjoint operators, the Schur-Horn theorem, the Pythagorean theorem, the Carpenter theorem, spectral theory}


\subjclass{Primary:  
47B15, Secondary: 46C05}

\thanks{The first author was partially supported by NSF grant DMS-1956395. The views expressed in this article are those of the authors and do not reflect the official policy or position of the United States Air Force, Department of Defense, or the U.S. Government.}

\maketitle


\section{Introduction}

The classical Schur-Horn theorem \cite{horn, moa, schur} characterizes diagonals of hermitian matrices in terms of their eigenvalues. A sequence $(d_1,\ldots,d_N)$ is a diagonal of a hermitian matrix with eigenvalues $(\lambda_1,\ldots,\lambda_N)$ if and only if 
\begin{equation}\label{horn2}
(d_1,\ldots,d_N) \in \operatorname{conv} \{ (\lambda_{\sigma(1)},\ldots,\lambda_{\sigma(N)}): \sigma \in S_N\}.
\end{equation}
This characterization has attracted  significant interest and has been generalized in many remarkable ways. Some major milestones are the Kostant convexity theorem \cite{ko} and the convexity of moment mappings in symplectic geometry \cite{at, gs1, gs2}. 

An infinite-dimensional extension of the Schur-Horn theorem has been a subject of intensive study in recent years. In particular, we are interested in the following generalization of the Schur-Horn theorem to compact operators. Non-compact operators are studied in our subsequent paper.

\begin{problem}\label{DT}
Given a compact linear operator $T$ on a separable Hilbert space $\mathcal H$,  characterize the set of all diagonals 
\begin{equation}\label{dt}
\mathcal D(T) = \{ (\langle T e_i, e_i \rangle)_{k\in \N}: (e_i)_{i\in\N} \text { is an orthonormal basis of } \mathcal H \} \subset \ell^{\infty}(\N).
\end{equation}
\end{problem}

Arveson and Kadison \cite{ak} extended the Schur-Horn theorem to positive trace class operators. This was preceded by the work of Gohberg and Markus \cite{gm}. The Schur-Horn theorem has been extended to all compact positive operators by Kaftal and Weiss \cite{kw}. These results are stated in terms of majorization inequalities. For a survey on infinite Schur-Horn majorization theorems and their connections to operator ideals we refer to \cite{kw0, weiss}.

Kaftal and Weiss \cite{kw} showed that a nonincreasing sequence of positive numbers $d_1\ge d_2 \ge \ldots>0$ converging to $0$ is a diagonal of a positive compact operator with trivial kernel and positive eigenvalues $\lambda_1 \ge \lambda_2 \ge \ldots>0$, listed with multiplicity, if and only if
\begin{equation}\label{horn1}
\sum_{i=1}^{\infty}d_{i}= \sum_{i=1}^{\infty}\lambda_{i}
\qquad\text{and}\qquad
\sum_{i=1}^{n}d_{i}\leq \sum_{i=1}^{n}\lambda_{i}\quad\text{for all }n\in\N.
\end{equation}
Loreaux and Weiss \cite{lw} extended this result by showing necessary conditions and sufficient conditions on $\mathcal D(T)$ for positive compact operators $T$ with nontrivial kernel. When $\ker(T)$ is infinite dimensional, the necessary and sufficient conditions coincide yielding a complete characterization. However, when $\ker T$ is nontrivial and finite dimensional, the full characterization of $\mathcal D(T)$ remains elusive. This is known as the kernel problem, see Theorem \ref{kp}.

We should also mention that there has been extensive interest beyond positive compact operators. Approximate descriptions of $\mathcal D(T)$ were given by Neumann \cite{neu} and Antezana, Massey, Ruiz, Stojanoff \cite{amrs}.
In his famous work, Kadison \cite{k1, k2} characterized diagonals of orthogonal projections. The authors characterized the set $\mathcal D(T)$ for a class of locally invertible positive operators \cite{mbjj} and for self-adjoint operators with finite spectrum \cite{npt, mbjj5, jj}. M\"uller and Tomilov \cite{mt} showed a broad sufficiency result, in terms of a Blaschke type condition, for the existence of diagonals. Some of these instances were shown to be necessary by Loreaux and Weiss \cite{lw3}. Finally, several authors \cite{am, am2, am3, rara, dfhs, ks, mr, ravi}  studied an extension of the Schur-Horn problem in von Neumann algebras, which was originally proposed by Arveson and Kadison \cite{ak}.

Siudeja jointly with the authors \cite{unbound} proved an infinite-dimensional variant of the Schur-Horn theorem for unbounded self-adjoint operators with discrete spectrum. Previous results dealt only with bounded operators. More importantly for the purposes of this paper, a technique of diagonal-to-diagonal results was developed. A typical diagonal-to-diagonal result is the following extension of the result of Kaftal and Weiss \cite{kw} mentioned above. Given two positive nonincreasing sequences $(d_i)_{i\in \N}$ and $(\lambda_i)_{i\in \N}$, both converging to $0$, and satisfying \eqref{horn1}, if $(\lambda_i)_{i\in \N}$ is a diagonal of some (not necessarily bounded) symmetric operator $T$, then $(d_i)_{i\in \N}$ is also a diagonal of \(T\).

The goal of this paper is to characterize the set of diagonals $\mathcal D(T)$ of an arbitrary compact self-adjoint operator $T$, which extends earlier results for positive compact operators in \cite{kw, lw}. We show two types of characterization results of $\mathcal D(T)$. Theorem \ref{intropv2} shows a characterization of $\mathcal D(T)$ for the class of compact operators sharing the same nonzero eigenvalues (with multiplicities) as $T$. Hence, this result merely neglects the dimension of the kernel of $T$. 

Our second more precise result gives a description of $\mathcal D(T)$ for a fixed compact operator $T$, modulo the above mentioned kernel problem \cite{lw} for positive compact operators. Our description is given in a form of an algorithm that determines whether a numerical sequence is a diagonal of $T$, or not. The algorithm is always conclusive unless the whole problem is reduced to the kernel problem for a pair of positive compact operators into which the original operator $T$ is \textit{decoupled}, see Definition \ref{decouple}. The latter scenario happens precisely when the positive and negative \textit{excesses}, which are measures of tightness of the majorization inequalities, 
are both equal to zero, see Definition \ref{excessdef}.

To describe our results in detail, we need to set some notation.
Let \(I\) be a countably infinite set.
Let $c_0(I)$ be the collection of real-valued sequences converging to $0$ and indexed by the set \(I\). That is, the set \(\{i\in I: |d_{i}|>\eps\}\) is finite for every \(\eps>0\). Let \(c_{0}^{+}(I)\) be the set of nonnegative-valued sequences in \(c_{0}(I)\). We will write \(c_{0}\) and \(c_{0}^{+}\) when the indexing set of the sequence is obvious.

\begin{defn}
Let $\boldsymbol\lambda =(\lambda_i)_{i\in I}$ be a real-valued sequence. Define its {\it positive part} $\boldsymbol\lambda_+ =(\lambda^{+}_i)_{i\in I}$ by $\lambda^+_i=\max(\lambda_i,0)$. The {\it negative part} is defined as $\boldsymbol\lambda_-=(-\boldsymbol \lambda)_+$.
If $\boldsymbol\lambda \in c_0^+$, then define its {\it decreasing rearrangement} $\boldsymbol\lambda^{\downarrow} =(\lambda^{\downarrow}_i)_{i\in \N}$ by taking $\lambda^{\downarrow}_i$ to be the $i$th largest term of $\boldsymbol \lambda$. For the sake of brevity, we will denote the $i$th term of $(\boldsymbol\lambda_{+})^{\downarrow}$ by $\lambda_{i}^{+\downarrow}$, and similarly for $(\boldsymbol\lambda_{-})^{\downarrow}$.
\end{defn}

Our main result characterizing diagonals of compact self-adjoint operators, modulo the dimension of the kernel, takes the following form. The equivalence of (i) and (ii) is a known consequence of the Weyl-von Neumann-Berg theorem \cite[Proposition 39.10]{conway}.

\begin{thm}\label{intropv2}
Let $\boldsymbol\lambda, \boldsymbol d \in c_0$. Set
\[\sigma_{+} = \liminf_{n\to\infty} \sum_{i=1}^n (\lambda_i^{+ \downarrow} - d_i^{+ \downarrow})\quad\text{and}\quad \sigma_{-} = \liminf_{n\to\infty} \sum_{i=1}^n (\lambda_i^{- \downarrow} - d_i^{- \downarrow})\]
Let $T$ be a compact self-adjoint operator with eigenvalue list $\boldsymbol\lambda$.
Then the following are equivalent:
\begin{enumerate}
\item
$\boldsymbol d\in \mathcal D(T')$ for some operator $T'$ in the operator norm closure of unitary orbit of $T$,
\[ T' \in \overline
{\{ \text{$UTU^*$}: U \text{ is unitary}\}}^{||\cdot||},
\]
\item $\boldsymbol d$ is a diagonal of an operator $T'$ such that $T\oplus \boldsymbol 0$ and $T' \oplus \boldsymbol 0$
are unitarily equivalent, where $\boldsymbol 0$ denotes the zero operator on an infinite-dimensional Hilbert space,
\item
$\boldsymbol\lambda$ and $\boldsymbol d$ satisfy the following four conditions 
\begin{equation}\label{p1v2}
\sum_{i=1}^n \lambda_i^{+ \downarrow} \ge 
\sum_{i=1}^n d_i^{+ \downarrow}
\qquad\text{for all }k\in \N,
\end{equation}
\begin{equation}\label{p2v2}
\sum_{i=1}^n \lambda_i^{- \downarrow} \ge 
\sum_{i=1}^n d_i^{- \downarrow}
\qquad\text{for all }k\in \N,
\end{equation}
\begin{equation}\label{p3v2}
\boldsymbol d_+ \in \ell^1
\quad \implies\quad
\sigma_{-}\geq \sigma_{+}
\end{equation}
\begin{equation}\label{p4v2}
\boldsymbol d_- \in \ell^1 
\quad \implies\quad 
\sigma_{+}\geq \sigma_{-}
\end{equation}
\end{enumerate}
\end{thm}

The conditions \eqref{p1v2} and \eqref{p2v2} are well known majorization inequalities as in \eqref{horn1}.  However, \eqref{p3v2} and \eqref{p4v2} are surprising analogues of the trace condition, which have not been discovered previously. Indeed, for trace class operators the diagonal $\boldsymbol d \in \ell^1$, and hence positive and negative excesses are equal $\sigma_+=\sigma_-$. This implies the trace equality $\sum d_i = \sum \lambda_i$. On the other hand, if \(\boldsymbol{d}\notin\ell^{1}\) is a diagonal of a compact positive operator \(T\), then \eqref{p1v2} implies that $\sum \lambda_i \ge \sum d_i=\infty$. Hence, Theorem \ref{intropv2} recovers the characterization of diagonals of positive compact operators modulo the dimension of the kernel by Kaftal and Weiss \cite[Proposition 6.4]{kw}, which itself is a generalization of the trace class result of Arveson and Kadison \cite[Theorem 4.1]{ak}. However, the biggest novelty of conditions \eqref{p3v2} and \eqref{p4v2} is that they hold beyond the class of trace class operators, imposing a nontrivial inequality between excesses if either the positive or negative part of $\boldsymbol d$ is summable.

Theorem \ref{intropv2} is a consequence Theorem \ref{intropv3}, which shows necessary and sufficient conditions on $\mathcal D(T)$ for a compact self-adjoint operator $T$. The proof of the necessity of \eqref{p1v2}--\eqref{p4v2} is relatively straightforward and does not cause many difficulties. In contrast, the sufficiency proof is a complex and challenging collection of incremental results, which culminates only after a long series of results. We use the technique of diagonal-to-diagonal results, which was pioneered in the earlier joint work with Siudeja \cite{unbound}. 

The first step in this process shows an unexpected result about disappearing of a single negative eigenvalue under the dominance assumption $\lambda_i^{+\downarrow} \ge d_i^{+\downarrow}$ for all $i\in \N$, see Theorem \ref{1neg}. It is easy to see that a compact operator with eigenvalues $(-1,1,\frac12,\frac14, \frac18,\ldots)$ can have diagonal $(0,0, \frac12,\frac14, \frac18,\ldots)$. However, it is highly nontrivial that a strictly positive sequence $(\frac12, \frac14, \frac18,\ldots)$ is also a diagonal of this operator. In the second step we show a result about disappearing of infinitely many negative eigenvalues, see Theorem \ref{NoPosDiag}, which at the same time relaxes the dominance assumption to the majorization inequality \eqref{p1v2}. In particular, this implies that a compact operator with eigenvalues $( 1,\frac12,\frac14, \frac18,\ldots, -\frac12,-\frac14, -\frac18, \ldots)$ can have a strictly positive diagonal $(\frac12,\frac14, \frac18,\ldots)$. 

The third step is Theorem \ref{posdiag}, which shows a diagonal-to-diagonal result for nonzero initial diagonal $\boldsymbol \lambda$ and positive target diagonal $\boldsymbol d$ under the equal excess assumption $\sigma_+=\sigma_->0$. Notably, $\boldsymbol \lambda$ might have only finitely many positive terms. As an example, if $\boldsymbol \lambda=(1,1,-\frac12,-\frac14, -\frac18, \ldots)$ is a diagonal of a self-adjoint operator, then so is $\boldsymbol d=(\frac12,\frac14, \frac18,\ldots)$. In the fourth step we relax the inessential assumptions about the initial diagonal $\boldsymbol \lambda$ and the target diagonal $\boldsymbol d$ to establish a complete diagonal-to-diagonal result under the equal excess assumption $\sigma_+=\sigma_->0$, see Theorem \ref{step5}. Here we would like to note that it takes a considerable effort to relax the positivity assumption on the target diagonal $\boldsymbol d$ by using a technique of annihilation of excesses. Since we impose no assumptions on the cardinalities of the negative terms, zero terms, or positive terms of the diagonals $\boldsymbol d$ and $\boldsymbol \lambda$, the proof requires several technical combinatorial arguments covering all possible scenarios. 

The fifth step is Theorem \ref{step6}, which considers the one-sided non-summable case
\[
\sum_{\lambda_i<0} |\lambda_i| = \sum_{d_i<0} |d_i| = \infty \quad\text{and}\quad \sum_{\lambda_i>0} \lambda_i <\infty,
\]
when excesses are not equal $\sigma_->\sigma_+>0$. Note that by \eqref{p3v2} we automatically have $\sigma_- \ge\sigma_+$. The proof relies heavily on elaborate sequential results of Kaftal and Weiss \cite{kw} as well as an extension of their result to the diagonal-to-diagonal setting in the form of Proposition \ref{kwnonneg}. In the sixth step we show a diagonal-to-diagonal result in the two-sided non-summable case
\[
\sum_{\lambda_i<0} |\lambda_i| =  \sum_{\lambda_i>0} \lambda_i =\infty,
\]
see Theorem \ref{step7}. The seventh and final step combines previous sufficiency results and culminates in the proof our main result for compact operators, Theorem \ref{intropv2}. 
In the final section we discuss the kernel problem for positive compact operators and present an algorithm for determining whether a numerical sequence is a diagonal of a given compact self-adjoint operator, or not.

\section{Preliminary results on majorization}

In this section we show preliminary results for numerical sequences such as the equivalence of Riemann and Lebesgue majorization in Proposition \ref{LR}.
We start by introducing notions of majorization, concatenation, and excess of a pair of sequences.

\begin{defn}\label{rmajdef}
Let \(I,J\) be countable infinite sets, and let \(\boldsymbol\lambda = (\lambda_{i})_{i\in I}\) and \(\boldsymbol{d}=(d_{i})_{i\in J}\) be sequences in \(c_{0}^{+}\). We say that \(\boldsymbol\lambda\) {\it majorizes} \(\boldsymbol d\) and write \(\boldsymbol d \prec \boldsymbol \lambda\) if
\begin{equation}\label{rmajdef1}\sum_{i=1}^{n}d_{i}^{\downarrow}\leq \sum_{i=1}^{n}\lambda_{i}^{\downarrow}\quad\text{for all }n\in\N.\end{equation}
If, in addition, we have
\[\liminf_{n\to\infty}\sum_{i=1}^{n}(\lambda_{i}^{\downarrow}-d_{i}^{\downarrow})=0\]
then we say that \(\boldsymbol\lambda\) {\it strongly majorizes} \(\boldsymbol d\), and we write \(\boldsymbol d\preccurlyeq\boldsymbol\lambda\).
\end{defn}

\begin{defn} Given two sequences \(\boldsymbol\lambda = (\lambda_{j})_{j\in J}\) and \(\boldsymbol{d}=(d_{i})_{i\in I}\), the \textit{concatenation} of \(\boldsymbol{\lambda}\) and \(\boldsymbol{d}\), denoted \(\boldsymbol\lambda\oplus\boldsymbol{d}\), is the sequence \((c_{k})_{k\in I\sqcup J}\) where \(I\sqcup J\) is the disjoint union of \(I\) and \(J\), and \[c_{k} = \begin{cases} \lambda_{k} & k\in J,\\ d_{k} & k\in I.\end{cases}\]
Note, if \(I\cap J=\varnothing\), then \(I\sqcup J:=I\cup J\), if \(I\cap J\neq\varnothing\), then \(I\sqcup J= (I\times\{1\})\cup (J\times\{2\})\). In the latter case \(k\in I\) is interpreted to mean \((k,1)\in I\times\{1\}\), and similarly for \(k\in J\).
\end{defn}

\begin{defn}\label{excessdef} Given two sequences \((\lambda_{i})_{i\in I}\) and \((d_{i})_{i\in J}\) in \(c_{0}\) the quantities
\[\sigma_{+} = \liminf_{k\to\infty}\sum_{i=1}^{k}(\lambda^{+\downarrow}_{i}-d^{+\downarrow}_{i})\quad\text{and}\quad \sigma_{-} = \liminf_{k\to\infty}\sum_{i=1}^{k}(\lambda^{-\downarrow}_{i}-d^{-\downarrow}_{i}),\]
are called the \textit{positive excess} and \textit{negative excess}, respectively.
\end{defn}

For sequences in $c_{0}$, which are not necessarily positive, we give an alternative definition of majorization which avoids decreasing rearrangements of sequences. Later, we will show the equivalence of these two definitions for sequences in $c_0^+$, see Proposition \ref{LR}.

\begin{defn}\label{deltadef}
Given two sequences $\boldsymbol{d}\in c_{0}(I)$ and $\boldsymbol\lambda\in c_{0}(J)$, for $\alpha\in\R\setminus\{0\}$ we define the function
\begin{equation}\label{delta}\delta(\alpha,\boldsymbol\lambda,\boldsymbol{d}) = \begin{cases} \displaystyle\sum_{\lambda_{i}\geq\alpha}(\lambda_{i}-\alpha)-\sum_{d_{i}\geq\alpha}(d_{i}-\alpha) & \alpha>0,\\
\displaystyle\sum_{\lambda_{i}\leq\alpha}(\alpha-\lambda_{i})-\sum_{d_{i}\leq\alpha}(\alpha-d_{i}) & \alpha<0.
\end{cases}
\end{equation}
If $\delta(\alpha,\boldsymbol\lambda,\boldsymbol{d})\geq 0$ for all $\alpha\neq 0$ then we say that $\boldsymbol\lambda$ {\it majorizes} $\boldsymbol{d}$.
\end{defn}

We can extend the definition of the function \(\delta\) in \eqref{delta} to sequences for which only the positive or negative part is in $c_0$. That is, if $\boldsymbol{d}_+$, $\boldsymbol\lambda_+\in c_{0}$, then the function $\delta (\alpha,\boldsymbol\lambda,\boldsymbol{d})$ is defined only for $\alpha>0$, similarly, for the negative parts. Indeed, the function \(\delta\) is even well-defined for \(\alpha>0\) in the case that just one of \(\boldsymbol{\lambda}_{+}\in c_{0}\) or \(\boldsymbol{d}_{+}\in c_{0}\). In this case the function may take values \(\pm\infty\).

\begin{lem}\label{dclem}  If $\mathbf c\in c_{0}^{+}(I)$, then
\[g(\alpha) = \sum_{i\colon c_{i}\geq\alpha}(c_{i}-\alpha).\]
is piecewise linear, continuous, and decreasing on $(0,\infty)$, and
\begin{equation}\label{dclem0}\lim_{\alpha\searrow0}g(\alpha) = \sum_{i\in I}c_{i}.\end{equation}
Moreover, if $\boldsymbol{d},\boldsymbol\lambda\in c_{0}(I)$, then $\delta(\alpha,\boldsymbol\lambda,\boldsymbol{d})$ is piecewise linear and continuous on $\R\setminus\{0\}$.
\end{lem}

\begin{proof} Note that $g$ is piecewise linear with knots at $\alpha=c_{i}$ for each $i\in I$. Indeed, if $\alpha\neq c_{i}$ for all $i\in I$, then
\[g'(\alpha) = -\#|\{i\in I\colon c_{i}\geq\alpha\}|.\]
Checking that $g$ is continuous at the knots is straightforward. Thus, we have shown that $g$ is a decreasing continuous function on $(0,\infty)$.

That $g$ is decreasing implies that the limit $\lim_{\alpha\searrow0}g(\alpha)$ exists, though it is possibly infinite. Let $\beta\geq\alpha>0$ and we have
\[g(\alpha) = \sum_{i\colon c_{i}\geq\alpha}(c_{i}-\alpha) \geq \sum_{i\colon c_{i}\geq\beta}(c_{i}-\alpha),\]
which implies
\[\lim_{\alpha\searrow0}g(\alpha) \geq \sum_{i\colon c_{i}\geq\beta}c_{i}.\]
Since this holds for all $\beta$ we have
\begin{equation}\label{dclem1}\lim_{\alpha\searrow0}g(\alpha) \geq \sum_{i\in I}c_{i}.\end{equation}
If $\mathbf c$ is not summable, then \eqref{dclem0} follows from \eqref{dclem1}. Otherwise, we have
\[\lim_{\alpha\searrow0}g(\alpha)=  \lim_{\alpha\searrow0}\sum_{i\colon c_{i}\geq\alpha}(c_{i}-\alpha)\leq \lim_{\alpha\searrow0}\sum_{i\colon c_{i}\geq\alpha}c_{i}=\sum_{i\in I}c_{i},\]
which, together with \eqref{dclem1}, shows \eqref{dclem0}.
\end{proof}

Checking majorization as in Definition \ref{deltadef} seems to require checking an uncountable number of inequalities. This is in contrast to \eqref{rmajdef1}, which requires only a countable number. The next proposition shows that one need only check at most $\#|I|$ inequalities.

\begin{prop}\label{del} Let $\boldsymbol{d}\in c_{0}(I)$ and $\boldsymbol\lambda\in c_{0}(J)$. The following are equivalent:
\begin{enumerate}
\item $\delta(\alpha,\boldsymbol\lambda,\boldsymbol{d})\geq 0$ for all $\alpha\in\R\setminus\{0\}$,
\item $\delta(\lambda_{j},\boldsymbol\lambda,\boldsymbol{d})\geq 0$ for all $j\in J$ such that $\lambda_{j}\neq 0$.
\end{enumerate}
\end{prop}

\begin{proof} It is clear that (i) implies (ii). Assume (ii) holds and let $\delta(\alpha)=\delta(\alpha,\boldsymbol\lambda,\boldsymbol{d})$. Let $(\mu_{i})_{i=1}^{M-1}$ denote the distinct strictly positive values of $\boldsymbol\lambda$ in strictly decreasing order. For $\alpha\in(\mu_{r+1},\mu_{r})$ and $\alpha\neq d_{i}$ we have
\begin{equation}\label{der}\begin{split}\delta'(\alpha) & = -\#|\{j\in J\colon \lambda_{j}\geq\alpha\}| + \#|\{i\in I\colon d_{i}\geq\alpha\}|\\
 & = -\#|\{j\in J\colon \lambda_{j}\geq\mu_{r}\}| + \#|\{i\in I\colon d_{i}\geq\alpha\}|.
\end{split}\end{equation}
The function $\delta$ is piecewise linear and continuous on $(0,\infty)$. From \eqref{der} we see that $\delta$ is concave down on $(\mu_{r+1},\mu_{r})$. This implies that for $\alpha\in[\mu_{r+1},\mu_{r}]$ we have
\[\delta(\alpha)\geq\min\{\delta(\mu_{r}),\delta(\mu_{r+1})\},\]
and by (ii) we have $\delta(\mu_{r})\geq0$ for all $r$. This shows that $\delta(\alpha)\geq 0$ for all $\alpha>0$. A similar argument shows that $\delta(\alpha)\geq0$ for all $\alpha<0$.
\end{proof}

The following result shows that majorization of sequences in \(c_{0}^{+}\) can be expressed in two equivalent ways by Definition \ref{rmajdef} or Definition \ref{deltadef}. The relationship between (i) and (ii) in Proposition \ref{LR} is analogous to that between Riemann and Lebesgue integration. Hence, we often call (ii) {\it Lebesgue majorization}. The quantity that appears on the right-hand side of \eqref{rtelt} is the positive excess given in Definition \ref{excessdef}, and \eqref{rtelt} shows that the excess can be expressed in a Lebesgue form, that is, in terms of the function \(\delta(\alpha,\boldsymbol{\lambda},\boldsymbol{d})\).

\begin{prop}\label{LR}
 Let $\boldsymbol d=(d_{i})_{i=1}^{\infty}, \boldsymbol\lambda=(\lambda_{i})_{i=1}^{\infty}$ be sequences in $c_0^+$. Then, the following two conditions are equivalent:
\begin{enumerate}
\item \(\boldsymbol d\prec \boldsymbol \lambda\),
\item $\delta(\alpha,\boldsymbol\lambda,\boldsymbol d)\geq 0$ for all $\alpha> 0$.
\end{enumerate}
In this case 
\begin{equation}\label{rtelt}\liminf_{\alpha\searrow 0} \delta(\alpha,\boldsymbol\lambda,\boldsymbol d) = \liminf_{k\to\infty}\sum_{i=1}^{k}(\lambda^\downarrow_{i}-d^\downarrow_{i}).\end{equation}
\end{prop}

\begin{proof} Without loss of generality we can assume that $\boldsymbol\lambda$ and $\boldsymbol d$ are nonincreasing. Indeed, (i) and (ii) hold for $\boldsymbol\lambda$ and $\boldsymbol d$ if and only if they hold for their decreasing rearrangements $\boldsymbol\lambda^\downarrow$ and $\boldsymbol d^\downarrow$. 

Set $\delta(\alpha)=\delta(\alpha,\boldsymbol\lambda,\boldsymbol{d})$. For each $\alpha\in(0,\infty)$ set $m_{\alpha} = \#|\{i\colon \lambda_{i}\geq\alpha\}|$ and $n_{\alpha} = \#|\{i\colon d_{i}\geq \alpha\}|$. With this notation we have
\[\delta(\alpha) = \sum_{i=1}^{m_{\alpha}}\lambda_{i} - \sum_{i=1}^{n_{\alpha}}d_{i} + \alpha(n_{\alpha}-m_{\alpha}).\]
Fix $\alpha\in(0,\infty)$. There are two cases to consider.

\textbf{Case 1.} Assume $m_{\alpha}\geq n_{\alpha}$. In this case we have
\begin{align*}
\delta(\alpha) & = \sum_{i=1}^{n_{\alpha}}\lambda_{i} - \sum_{i=1}^{n_{\alpha}}d_{i} + \sum_{i=n_{\alpha}+1}^{m_{\alpha}}\lambda_{i} + \alpha(n_{\alpha}-m_{\alpha})\\
 & \geq \sum_{i=1}^{n_{\alpha}}\lambda_{i} - \sum_{i=1}^{n_{\alpha}}d_{i} + \sum_{i=n_{\alpha}+1}^{m_{\alpha}}\alpha + \alpha(n_{\alpha}-m_{\alpha}) = \sum_{i=1}^{n_{\alpha}}(\lambda_{i}-d_{i})
\end{align*}

\textbf{Case 2.} Assume $m_{\alpha}<n_{\alpha}$. In this case
\begin{align*}
\delta(\alpha) & = \sum_{i=1}^{n_{\alpha}}\lambda_{i} - \sum_{i=1}^{n_{\alpha}}d_{i} - \sum_{i=m_{\alpha}+1}^{n_{\alpha}}\lambda_{i} + \alpha(n_{\alpha}-m_{\alpha})\\
 & > \sum_{i=1}^{n_{\alpha}}\lambda_{i} - \sum_{i=1}^{n_{\alpha}}d_{i} - \sum_{i=n_{\alpha}+1}^{m_{\alpha}}\alpha + \alpha(n_{\alpha}-m_{\alpha}) = \sum_{i=1}^{n_{\alpha}}(\lambda_{i}-d_{i})
\end{align*}

In either case, for every $\alpha\in(0,\infty)$ we have
\begin{equation}\label{ril}\delta(\alpha)\geq \sum_{i=1}^{n_{\alpha}}(\lambda_{i}-d_{i}).\end{equation}
It is clear from \eqref{ril} that (i) implies (ii).

Next, for each $k\in\N$ define
\[\ell_{k}=\max\{i\leq k\colon d_{i}\geq\lambda_{i}\text{ for }i=m,m+1,\ldots,k\}.\] 

We need to consider two possibilities.

\noindent\textbf{Case 1:} $d_k \ge \lambda_k$. Let $k'\in\N$ be the largest index such that $d_{k'} \ge \lambda_k$. We have
\[
\begin{aligned}
\delta(\lambda_{\ell_{k}}) = \delta(\lambda_k)  =  \sum_{i=1}^k (\lambda_i - \lambda_k) 
- \sum_{i=1}^{k'} (d_i - \lambda_k) \le  
\sum_{i=1}^k (\lambda_i - \lambda_k) 
- \sum_{i=1}^{k} (d_i - \lambda_k) =\sum_{i=1}^k (\lambda_i - d_i).
\end{aligned}
\]
{\bf Case 2:} $d_k< \lambda_k$. From the definition of $\ell_{k}$ we have
\[
d_i < \lambda_i \qquad\text{for }i=\ell_{k}+1,\ell_{k}+2,\ldots,k.
\]
Then,
\[
\sum_{i=1}^k (\lambda_i - d_i) = \sum_{i=1}^{\ell_{k}} (\lambda_i-d_i) + \sum_{i=\ell_{k}+1}^k (\lambda_i - d_i) > \sum_{i=1}^{\ell_{k}} (\lambda_i-d_i).
\]
By definition $d_{\ell_{k}}\geq \lambda_{\ell_{k}}$ and $\ell_{\ell_{k}}=\ell_{k}$. By Case 1 we have $\sum_{i=1}^{\ell_{k}}(\lambda_{i}-d_{i})\geq \delta(\lambda_{\ell_{k}})$. Thus, for all $k\in\N$ we have
\begin{equation}\label{lir}\delta(\lambda_{\ell_{k}})\leq \sum_{i=1}^{k}(\lambda_{i}-d_{i}).\end{equation}
From \eqref{lir} we see that (ii) implies (i).

Finally, we wish to prove \eqref{rtelt}. Note that if either $\boldsymbol\lambda$ or $\boldsymbol{d}$ is summable, then from Lemma \ref{dclem} we have
\[\lim_{\alpha\searrow0}\delta(\alpha) = \lim_{\alpha\searrow0}\sum_{i\colon \lambda_{i}\geq\alpha}(\lambda_{i}-\alpha) - \lim_{\alpha\searrow0}\sum_{i\colon d_{i}\geq\alpha}(d_{i}-\alpha) = \sum_{i=1}^{\infty}\lambda_{i}- \sum_{i=1}^{\infty}d_{i} =  \sum_{i=1}^{\infty}(\lambda_{i}-d_{i}) .\]
Thus, we may assume both $\boldsymbol\lambda$ and $\boldsymbol{d}$ are nonsummable.

Since $\boldsymbol{d}$ is nonsummable, we see that $n_{\alpha}\to\infty$ as $\alpha\searrow0$. Using \eqref{ril} yields
\begin{equation}\label{lgr}\liminf_{\alpha\searrow0}\delta(\alpha)\geq\liminf_{\alpha\searrow0}\sum_{i=1}^{n_{\alpha}}(\lambda_{i}-d_{i})\geq \liminf_{k\to\infty}\sum_{i=1}^{k}(\lambda_{i}-d_{i}).\end{equation}

Next, we would like to use \eqref{lir} to demonstrate the reverse inequality to \eqref{lgr}. However, it may not be true that $\ell_{k}\to\infty$ as $k\to\infty$. Thus, we must consider these two cases.

\textbf{Case 1:} There exists $M>0$ such that $\ell_{k}<M$ for all $k\in\N$. In this case $d_{i}<\lambda_{i}$ for all $i\geq M$. For each $j>M$ fix $k=k(j)>j$ such that $\lambda_{k+1}<\lambda_{k}$. Since $d_{k}<\lambda_{k}$, we see that $k':=\#|\{i\colon d_{i}\geq\lambda_{k}\}|<k$ and $d_{i}<\lambda_{k}$ for $i>k'$. Using this we have
\begin{align*}\delta(\lambda_{k}) & = \sum_{\lambda_{i}\geq\lambda_{k}}(\lambda_{i}-\lambda_{k}) - \sum_{d_{i}\geq\lambda_{k}}(d_{i}-\lambda_{k}) = \sum_{i=1}^{k}\lambda_{i} - \sum_{i=1}^{k'}d_{i} - \lambda_{k}(k - k')\\
 & = \sum_{i=1}^{k}(\lambda_{i}-d_{i})+\sum_{i=k'+1}^{k}d_{i} - \lambda_{k}(k - k') = \sum_{i=1}^{k}(\lambda_{i}-d_{i})+\sum_{i=k'+1}^{k}(d_{i} - \lambda_{k}) \leq \sum_{i=1}^{k}(\lambda_{i}-d_{i}).
\end{align*}
Since $\sum_{i=1}^{N}(\lambda_{i}-d_{i})$ is increasing for $N>M$, and $k(j)\to\infty$ as $j\to\infty$ we have
\[\liminf_{\alpha\searrow0}\delta(\alpha)\leq \liminf_{j\to\infty}\delta(\lambda_{k(j)})\leq \liminf_{j\to\infty}\sum_{i=1}^{k(j)}(\lambda_{i}-d_{i}) = \liminf_{k\to\infty}\sum_{i=1}^{k}(\lambda_{i}-d_{i}).\]

\textbf{Case 2:} $\ell_{k}\to\infty$ as $k\to\infty$. From \eqref{lir} we have
\[\liminf_{\alpha\searrow0}\delta(\alpha)\leq \liminf_{k\to\infty}\delta(\lambda_{\ell_{k}})\leq \liminf_{k\to\infty}\sum_{i=1}^{k}(\lambda_{i}-d_{i}).\]
In either case, we have shown the reverse inequality to \eqref{lgr}, and thus we have demonstrated \eqref{rtelt}.
\end{proof}

The following result shows that a stronger identity on the excess holds when the sequence $\boldsymbol \lambda$ is summable.

\begin{prop}\label{LRS} Let \(\boldsymbol d = (d_{i})_{i\in I}\) and \(\boldsymbol\lambda = (\lambda_{i})_{i\in J}\) be nonnegative sequences. If \(\boldsymbol\lambda\) is summable and \(\delta(\alpha,\boldsymbol\lambda,\boldsymbol d)\geq 0\) for all \(\alpha>0\), then \(\boldsymbol{d}\) is summable, and
\[\liminf_{\alpha\searrow 0}\delta(\alpha,\boldsymbol\lambda,\boldsymbol d) = \sum_{j\in J}\lambda_{j} - \sum_{i\in I}d_{i}.\]
\end{prop}

\begin{proof} Without loss of generality, we may assume both \(I\) and \(J\) are countably infinite by appending infinitely many zeros to the sequences. If \(\boldsymbol{d}\notin c_{0}\), then we see from \eqref{delta} that \(\delta(\alpha,\boldsymbol{\lambda},\boldsymbol{d})=-\infty\) for some \(\alpha>0\). Therefore, \(\boldsymbol{d}\in c_{0}\). By Proposition \ref{LR} we see that \(\boldsymbol{d}\prec\boldsymbol\lambda\), and hence \(\boldsymbol d\) is summable. The desired equality follows from \eqref{rtelt}.
\end{proof}

\section{Necessary conditions for one-sided compact operators}

In this section we show necessary conditions on diagonals of operators whose positive (or negative) part is compact. When the excess is zero, we show that the operator must decouple into two simpler operators, see Proposition \ref{p211}. Moreover, we show several estimates on the excess that play the role of trace conditions when the operator is not necessarily trace class. We shall use the following result \cite[Theorem 3.1]{npt}.

\begin{thm}\label{ext} Let $E$ be a self-adjoint operator on a Hilbert space $\Hil$
\[\sigma(E)\subset \{\mu_{1},\ldots,\mu_{m-1}\}\cup[\mu_{m},\infty),\]
where $\mu_{1}<\ldots<\mu_{m}$, $m\geq 2$. Let $(e_{i})_{i\in I}$ be an orthonormal basis for $\Hil$ and set $d_{i} = \langle Ee_{i},e_{i}\rangle$. Then, for any $r=2,\ldots,m$
\[\sum_{d_{i}\leq\mu_{r}}(\mu_{r}-d_{i})\leq \sum_{j=1}^{r-1}(\mu_{r}-\mu_{j})m_{E}(\mu_{j}),\]
where $m_{E}(\mu) = \dim\ker(E-\mu \mathbf I)$.
\end{thm}

Using Proposition \ref{del} and Theorem \ref{ext} we prove a generalization of Schur's Theorem, which shows that majorization is a necessary condition on diagonals of operators whose negative (or positive) part is compact.

\begin{thm}\label{cptschurv2} Let \(E\) be a self-adjoint operator on a Hilbert space \(\Hil\) with compact negative part, and let \((e_{i})_{i\in I}\) be an orthonormal basis for \(\Hil\). Set \(\boldsymbol{d} = (\langle Ee_{i},e_{i}\rangle)_{i\in I}\), and let \(\boldsymbol\lambda\) be the sequence of strictly negative eigenvalues of \(E\), counted with multiplicity. Then \(\delta(\alpha,\boldsymbol\lambda,\boldsymbol{d})\geq 0\) for all \(\alpha<0\).\end{thm}

\begin{proof} Let $(\mu_{i})_{i=1}^{M-1}$ ($M\in\N\cup\{\infty\}$) denote the distinct strictly negative eigenvalues of $E$ in strictly increasing order, and set $\mu_{M}=0$. For \(\alpha\in\R\) let $\delta(\alpha)=\delta(\alpha,\boldsymbol\lambda,\boldsymbol{d})$, and let \(m_{E}(\alpha) = \operatorname{dim}\operatorname{ker}(E-\alpha\mathbf{I})\). By Theorem \ref{ext} for all $r<M+1$ we have
\begin{align*}
\delta(\mu_{r}) & = \sum_{\lambda_{i}\leq\mu_{r}}(\mu_{r}-\lambda_{i})-\sum_{d_{i}\leq\mu_{r}}(\mu_{r}-d_{i}) = \sum_{j=1}^{r-1}(\mu_{r}-\mu_{j})m_{E}(\mu_{j}) - \sum_{d_{i}\leq\mu_{r}}(\mu_{r}-d_{i})\geq0.
\end{align*}
This shows that $\delta(\lambda_{j})\geq0$ for all $j$ such that $\lambda_{j}<0$. Applying Proposition \ref{del} to \(\boldsymbol\lambda\) and \((d_{i})_{d_{i}<0}\) we see that \(\delta(\alpha)\geq 0\) for all \(\alpha<0\).
\end{proof}

\begin{cor}\label{cptschur} Let $E$ be a self-adjoint compact operator for a Hilbert space $\Hil$, and let $(e_{i})_{i\in I}$ be an orthonormal basis for $\Hil$. Set $\boldsymbol{d}=(\langle Ee_{i},e_{i}\rangle)_{i\in I}$ and let $\boldsymbol\lambda\in c_{0}(I)$ be the eigenvalue sequence of $E$. Then, $\delta(\alpha,\boldsymbol\lambda,\boldsymbol d)\geq0$ for all $\alpha\neq 0$.
\end{cor}

\begin{proof} By Theorem \ref{cptschurv2} we see that \(\delta(\alpha,\boldsymbol\lambda,\boldsymbol d)\geq 0\) for all \(\alpha<0\). The same argument applied to $-E$ yields $\delta(\alpha,\boldsymbol\lambda,\boldsymbol d)\geq 0$ for all $\alpha>0$.
\end{proof}

The following is an elementary lemma that will be used to show that the operator decouples when the excess is zero.

\begin{lem}\label{spread} Let $E$ be a self-adjoint operator, and let $\{e_{1},e_{2}\}$ be an orthonormal set. Assume that $\langle Ee_{1},e_{1}\rangle\geq \langle Ee_{2},e_{2}\rangle$. If $\langle Ee_{1},e_{2}\rangle\neq 0$ then there is an orthonormal set $\{f_{1},f_{2}\}$ such that $\lspan\{e_{1},e_{2}\}=\lspan\{f_{1},f_{2}\}$ and 
\[\langle Ef_{1},f_{1}\rangle >\langle Ee_{1},e_{1}\rangle\geq \langle Ee_{2},e_{2}\rangle > \langle Ef_{2},f_{2}\rangle\]
\end{lem}

\begin{proof} Set $\alpha = \langle Ee_{1},e_{2}\rangle/\abs{\langle Ee_{1},e_{2}\rangle}$. For each $\theta\in\R$ set
\[f_{i}(\theta) = \begin{cases} \cos\theta e_{1} - \alpha\sin\theta e_{2} & i=1\\ \sin\theta e_{1} + \alpha\cos\theta e_{2} & i=2,\end{cases}\]
and define the function 
\begin{align*}
g(\theta) &  = \langle Ef_{1}(\theta),f_{1}(\theta)\rangle = \cos^{2}\theta\langle Ee_{1},e_{1}\rangle + \sin^{2}\theta\langle Ee_{2},e_{2}\rangle - 2\sin\theta\cos\theta\operatorname{Re}(\overline{\alpha}\langle Ee_{1},e_{2}\rangle)\\
 & = \cos^{2}\theta\langle Ee_{1},e_{1}\rangle + \sin^{2}\theta\langle Ee_{2},e_{2}\rangle - 2\sin\theta\cos\theta\abs{\langle Ee_{1},e_{2}\rangle}.
\end{align*}
Differentiating we have
\begin{align*}
g'(\theta) & = -2\cos\theta\sin\theta\langle Ee_{1},e_{1}\rangle + 2\cos\theta\sin\theta\langle Ee_{2},e_{2}\rangle - 2(\cos^{2}\theta-\sin^{2}\theta)|\langle Ee_{1},e_{2}\rangle|,
\end{align*}
and $g'(0) = -2|\langle Ee_{1},e_{2}\rangle|$. Thus, for some $\theta_{0}<0$ we have 
\[\langle Ef_{1}(\theta_{0}),f_{1}(\theta_{0})\rangle=g(\theta_{0})>g(0)=\langle Ee_{1},e_{1}\rangle.\]
Setting $f_{1}=f_{1}(\theta_{0})$ and $f_{2}=f_{2}(\theta_{0})$ gives the desired result, since
\[\langle Ef_{1},f_{1}\rangle +\langle Ef_{2},f_{2}\rangle = \langle Ee_{1},e_{1}\rangle + \langle Ee_{2},e_{2}\rangle.
\qedhere
\]
\end{proof}

\begin{prop}\label{p211}
Let $E$ be a self-adjoint operator on $\mathcal H$ with the eigenvalue list (with multiplicity) $\boldsymbol \lambda$, which is possibly an empty list. Let $\boldsymbol d$ be a diagonal of $E$ with respect to some orthonormal basis $(e_i)_{i\in I}$. Assume that either:
\begin{itemize}
\item the positive part $E_{+}$ is compact and 
$\displaystyle
\liminf_{\alpha\searrow 0}\delta(\alpha,\boldsymbol\lambda,\boldsymbol d) =0
$,
or 
\item the negative part $E_-$ is compact and
$\displaystyle
\liminf_{\alpha\nearrow 0}\delta(\alpha,\boldsymbol\lambda,\boldsymbol d) =0$.
\end{itemize}
Then,
the operator $E$ decouples at the point $0$. That is, $E$ is block diagonal with respect to subspaces  
\begin{equation}\label{p210}
\mathcal H_0=\overline{\lspan }\{ e_i: d_i < 0 \}\qquad\text{and}\qquad \mathcal H_1=\overline{\lspan} \{ e_i: d_i \ge 0 \},
\end{equation}
and the spectra of each block satisfy
\begin{equation}\label{p212}
\sigma(E|_{\mathcal H_0}) \subset (-\infty,0]\qquad\text{and}\qquad\sigma(E|_{\mathcal H_1}) \subset [0,\infty).
\end{equation}
\end{prop}

To prove Proposition \ref{p211} we need to show several lemmas.

\begin{lem}\label{decomp} 
Under the same assumptions as in Proposition \ref{p211}, 
if there exist $j,k\in I$ such that
\[\langle Ee_{k},e_{k}\rangle<0\leq\langle Ee_{j},e_{j}\rangle,\]
then
\[\langle Ee_{j},e_{k}\rangle = 0.\]
\end{lem}

\begin{proof} By symmetry it suffices to consider the case when $E_+$ is compact.
Assume toward a contradiction that $\langle Ee_{j},e_{k}\rangle\neq 0$. By Lemma \ref{spread} there is an orthonormal set $\{f_{j},f_{k}\}$ with $\lspan\{f_{j},f_{k}\}=\lspan\{e_{j},e_{k}\}$ such that
\[\langle Ef_{k},f_{k}\rangle< d_{k}<0\leq d_{j}<\langle Ef_{j},f_{j}\rangle.\]
Setting $f_{i}=e_{i}$ for all $i\in I\setminus\{j,k\}$ gives an orthonormal basis $(f_{i})_{i\in I}$. Setting $\mathbf c=(\langle Ef_{i},f_{i}\rangle)_{i\in I}$, for $\alpha\in(d_{k},0)$ we have
\[\delta(\alpha,\boldsymbol\lambda,\mathbf c) = \sum_{\lambda_{i}\leq\alpha}(\alpha-\lambda_{i}) - \sum_{c_{i}\leq \alpha}(\alpha-c_{i}) = \delta(\alpha,\boldsymbol\lambda,\boldsymbol{d}) + c_{k}-d_{k},\]
and for $\alpha\in(0,d_{j})$ we have
\[\delta(\alpha,\boldsymbol\lambda,\mathbf c) = \delta(\alpha,\boldsymbol\lambda,\boldsymbol{d}) + d_{j}-c_{j}.
\]
Since the positive excess is zero, either there is some $\alpha\in(d_{k},0)$ such that $\delta(\alpha,\boldsymbol\lambda,\boldsymbol{d})<d_{k}-c_{k}$ or there is some $\alpha\in(0,d_{j})$ such that $\delta(\alpha,\boldsymbol\lambda,\boldsymbol{d})<c_{j}-d_{j}$. In either case we have $\delta(\alpha,\boldsymbol\lambda,\mathbf c)<0$ for some $\alpha$. From Corollary \ref{cptschur} we see that this is a contradiction and conclude that $\langle Ee_{j},e_{k}\rangle=0$.
\end{proof}

\begin{lem}\label{nel}
Let $E$ be a self-adjoint operator on $\mathcal H$ such that its positive part is a compact non-trace class operator with positive eigenvalues $\lambda_1 \ge \lambda_2\geq \ldots>0$, listed with multiplicity. Let $(f_j)_{j\in\N}$ be the corresponding orthonormal sequence of eigenvectors, that is, $Ef_j = \lambda_j f_j$, $j\in \N$.
Let  $(e_{i})_{i\in \N}$ be an orthonormal sequence in $\Hil$ and let $d_{i} = \langle Ee_{i},e_{i}\rangle$, $i\in\N$. 
Then,
\begin{equation}\label{nel1}
\liminf_{M\to\infty}\sum_{i=1}^{M}(\lambda_{i} - d_{i}) 
\geq \sum_{j=1}^{\infty}
\lambda_{j}\bigg( 1-\sum_{i\in \N} |\langle e_i, f_j \rangle|^2 \bigg) - \sum_{i\in \N} \langle EP e_i,e_i \rangle,
\end{equation}
where $P$ is an orthogonal projection of $\mathcal H$ onto $(\operatorname{span} \{f_j: j\in \N \})^\perp$.
\end{lem}

Note that the last sum in \eqref{nel1} is well-defined, finite or possibly $-\infty$, since the operator $EP$ is self-adjoint and negative semi-definite. Hence, the right hand side of \eqref{nel1} is well-defined as well.

\begin{proof}
We may assume that the left hand side of \eqref{nel1} is finite. Otherwise the conclusion \eqref{nel1} is clear.
For each \(j\in \N \) and \(M\in\N\) we define 
\[a_{j}^{(M)} = \sum_{i=1}^{M}|\langle e_{i},f_{j}\rangle|^{2}\quad\text{and}\quad a_{j} = \lim_{M\to\infty}a_{j}^{(M)} =\sum_{i\in \N}|\langle e_{i},f_{j}\rangle|^{2}.\]
Note that with this notation we have
\begin{equation}\label{cpttrace1} \sum_{i=1}^{M}d_{i} = \sum_{i=1}^{M}\sum_{j\in \N} \lambda_{j}|\langle e_{i},f_{j}\rangle|^{2} 
+  \sum_{i=1}^{M} \langle EP e_i,e_i \rangle
 = \sum_{j\in \N}\lambda_{j}a_{j}^{(M)}
+ \sum_{i=1}^{M} \langle EP e_i,e_i \rangle
.\end{equation}
For \(M\in\N\) we have
\begin{equation}\label{cpttrace2}
0\le a_j^{(M)} \le a_j \le 1,
\end{equation}
and
\begin{equation}\label{cpttrace5}
\sum_{j\in \N}a_{j}^{(M)} = \sum_{i= 1}^{ M} \sum_{j\in \N} |\langle e_{i},f_{j}\rangle|^{2}\leq  \sum_{i=1}^{M}1 = M.
\end{equation}
From \eqref{cpttrace1}, \eqref{cpttrace2}, and the assumption that \((\lambda_{i})_{i\in\N}\) is nonincreasing, we have
\begin{equation}\label{cpttrace3}
\sum_{i=1}^{M}d_{i} \le 
\sum_{j\in \N}\lambda_{j}a_{j}^{(M)}
\le 
\sum_{j\in \N}\lambda_{j}a_{j}
\le 
\lambda_{1} \sum_{j\in \N}a_{j}.
\end{equation}
Since the left hand side of \eqref{nel1} is finite and $\sum_{j\in\N} \lambda_j=\infty$, we have \[
\limsup_{M\to \infty} \sum_{i=1}^{M}d_{i} =\infty.
\]
Hence, $\sum_{j\in\N} a_j = \infty$.
For each $M\in\N$, let $K=K_{M}$ be the largest integer such that
\[\sum_{j=1}^{K}a_{j}\leq M.\]
Note that this implies
\[0\leq M-\sum_{j=1}^{K}a_{j}<a_{K+1}\leq 1.\]
Using the fact that $\lambda_{1}\geq\lambda_{2}\geq \ldots>0$, \eqref{cpttrace2}, and \eqref{cpttrace5} we have
\begin{equation}\begin{split}\label{cpttrace4}
\sum_{j\in \N}\lambda_{j}a_{j}^{(M)} & \leq \sum_{j=1}^{K}\lambda_{j}a_{j} + \sum_{j=1}^{K}\lambda_{j}(a_{j}^{(M)}-a_{j}) + \lambda_{K+1}\sum_{j=K+1}^{\infty}a_{j}^{(M)}
\\
 & \leq \sum_{j=1}^{K}\lambda_{j}a_{j} + \lambda_{K+1}\sum_{j=1}^{K}(a_{j}^{(M)}-a_{j}) + \lambda_{K+1}\left(M-\sum_{j=1}^{K}a_{j}^{(M)}\right)
 \\
 & = \sum_{j=1}^{K}\lambda_{j}a_{j} + \lambda_{K+1}\left(M-\sum_{j=1}^{K}a_{j}\right) \leq \sum_{j=1}^{K}\lambda_{j}a_{j} + \lambda_{K+1}.
 \end{split}
 \end{equation}
 The last term in \eqref{cpttrace4} is bounded by
 \begin{equation}
 \begin{split}\label{cpttrace8}
 &\sum_{j=1}^{M}\lambda_{j}a_{j} + \lambda_{M}\sum_{i=M+1}^{K}a_{j} + \lambda_{K+1} = \sum_{j=1}^{M}\lambda_{j}a_{j} + \lambda_{M}\left(\sum_{i=1}^{K}a_{j} - \sum_{i=1}^{M}a_{j}\right) + \lambda_{K+1}\\
 & \leq \sum_{j=1}^{M}\lambda_{j}a_{j} + \lambda_{M}\left(M - \sum_{j=1}^{M}a_{j}\right) + \lambda_{K+1} = \sum_{j=1}^{M}\lambda_{j}a_{j} + \sum_{j=1}^{M}\lambda_{M}(1-a_{j}) + \lambda_{K+1}.
\end{split}\end{equation}
Combining \eqref{cpttrace1}, \eqref{cpttrace4}, and \eqref{cpttrace8} yields
\begin{equation}\label{cpttrace6}
\sum_{i=1}^{M}(\lambda_{i}-d_i)  \geq 
\sum_{j=1}^{M}(\lambda_{j}-\lambda_{M})(1-a_{j})  - \lambda_{K+1} - \sum_{i=1}^{M} \langle EP e_i,e_i \rangle.
\end{equation}
By the monotone convergence theorem we have
\begin{equation}\label{cpttrace7}
 \liminf_{M\to\infty}\sum_{i=1}^{M}(\lambda_{i} - d_{i}) \geq \sum_{j=1}^{\infty}\lambda_{j}(1-a_{j}) 
 - \sum_{i=1}^{M} \langle EP e_i,e_i \rangle.
 \qedhere\end{equation}
\end{proof}

Next, we have a trace class version of Lemma \ref{nel}. Note that here the sequence \((e_{i})_{i\in I}\) is assumed to be a basis.

\begin{lem}\label{nelt}
Let $E$ be a self-adjoint operator on $\mathcal H$ such that its positive part is a compact trace class operator with positive eigenvalues $(\lambda_{j})_{j=1}^{M}$, where \(M\in\N\cup\{0,\infty\}\). Let $(f_j)_{j=1}^{M}$ be the corresponding orthonormal sequence of eigenvectors, that is, $Ef_j = \lambda_j f_j$.
Let  $(e_{i})_{i\in I}$ be an orthonormal basis in $\Hil$ and let $d_{i} = \langle Ee_{i},e_{i}\rangle$, $i\in I$. 
Then,
\begin{equation}\label{nelt1}
\sum_{i=1}^{M}\lambda_{i} - \sum_{i:d_{i}>0}d_{i} 
=\sum_{j=1}^{M}
\lambda_{j}\bigg( 1-\sum_{i:d_{i}>0} |\langle e_i, f_j \rangle|^2 \bigg) - \sum_{i:d_{i}>0} \langle EP e_i,e_i \rangle,
\end{equation}
where $P$ is an orthogonal projection of $\mathcal H$ onto $(\overline{\lspan}(f_j)_{j=1}^{M})^\perp$.
\end{lem} 

\begin{proof} First, note that
\begin{equation}\label{nelt2}
d_{i} = \sum_{j=1}^{M}\lambda_{j}|\langle e_{i},f_{j}\rangle|^{2} + \langle EPe_{i},e_{i}\rangle.
\end{equation}
Since \(EP\) is a negative semi-definite operator, we have
\[d_{i}\leq \sum_{j=1}^{M}\lambda_{j}|\langle e_{i},f_{j}\rangle|^{2},\]
and hence
\[\sum_{i:d_{i}>0} d_{i} \leq \sum_{i:d_{i}>0}\sum_{j=1}^{M}\lambda_{j}|\langle e_{i},f_{j}\rangle|^{2} = \sum_{j=1}^{M}\lambda_{j}\sum_{i:d_{i}>0}|\langle e_{i},f_{j}\rangle|^{2} \leq \sum_{j=1}^{M}\lambda_{j}\|f_{j}\|^{2} = \sum_{j=1}^{M}\lambda_{j}<\infty.\]
The conclusion follows by summing \eqref{nelt2} over \(\{i:d_{i}>0\}\), and subtracting both sides from \(\sum_{j=1}^{M}\lambda_{j}\).
\end{proof}

We are now ready to prove Proposition \ref{p211}.

\begin{proof}[Proof of Proposition \ref{p211}] 
By symmetry it suffices to consider the case when $E_+$ is compact.
By Lemma \ref{decomp}, we deduce that $\mathcal H_0$ and $\mathcal H_1$ are invariant subspaces of the operator $E$. 

Suppose that $(\lambda_{j})_{j\in J_{0}}$ is the sequence of positive eigenvalues of $E$, where \(J_{0}=\{j\in\N : j<M+1\}\) for some \(M\in\N\cup\{0,\infty\}\), and listed in nonincreasing order.  Let $(f_j)_{j\in J_{0}}$ be the corresponding orthonormal sequence of eigenvectors, that is, $Ef_j = \lambda_j f_j$, $j\in J_{0}$. Let $P$ be an orthogonal projection of $\mathcal H$ onto $(\operatorname{span} \{f_j: j\in J_{0} \})^\perp$.

Let $I_0=\{i\in\N : i<N+1\}$ where $N = \#|\{i\in I: d_i > 0\}|$, and reindex the orthonormal basis $(e_i)_{i\in I}$ such that $I_{0} \subset I$ and $ \langle Ee_{i},e_{i}\rangle=d_{i}^{+\downarrow}$ for all $i\in I_{0}$, where $d_{i}^{+\downarrow}$ is the $i$th largest  positive term of $\boldsymbol d$. 
Thus, $(e_{i})_{i\in I_{0}}$ is an orthonormal sequence in $\Hil_1$.

If $E_{+}$ is not trace class, then \(J_{0}=\N\). Since the excess if zero, we see that \(I_{0}=\N\).
By Proposition \ref{LR} the left-hand side of \eqref{nel1} is zero. By Lemma \ref{nel} we have
\begin{align}
\label{dd1}
\sum_{i\in I_{0}} |\langle e_i,f_j \rangle|^2 & =1 \qquad\text{for all }j\in J_{0},
\\
\label{dd2}
\langle E P e_i, e_i \rangle &=0 \qquad\text{for all }i \in I_{0}.
\end{align}
In the case the $E_{+}$ is trace class, both \eqref{dd1} and \eqref{dd2} follow from Lemma \ref{nelt} and Proposition \ref{LRS}.

By \eqref{dd1}, 
\begin{equation}\label{dd3}
\overline{\lspan}\{f_j: j\in J_{0}\} \subset \overline{\lspan} \{e_i:i\in I_{0}\} \subset \mathcal H_1.
\end{equation}
Hence, $\sigma(E|_{\mathcal H_0}) \subset (-\infty,0]$. 
Take any $i\in I$ such that $d_i=0$. 
By  \eqref{dd3},  we have $Pe_i=e_i$. Hence, 
\begin{equation}\label{dd4}
\langle E P e_i, e_i \rangle =0 \qquad\text{for all }i \in I, \ d_i=0.
\end{equation}
By \eqref{dd2}, \eqref{dd4}, and the fact that $EP \le 0$, we have $EPe_i=0$ for all $i\in I$, $d_i \ge 0$. 
Consequently, $E|_{\mathcal H_1} = E(\mathbf I -P)|_{\mathcal H_1} \ge 0$. This implies that $\sigma(E|_{\mathcal H_1}) \subset [0,\infty)$.
\end{proof}

Using Lemma \ref{nel} we arrive at the following inequality, which is a generalization of the trace condition on diagonals of trace class operators to compact operators where only the negative eigenvalues are summable. In particular, note that when $E$ is trace class, Corollary \ref{cpttracec}, together with its symmetric variant, gives the usual equality between the sum of the diagonal entries and the sum of the eigenvalues.

\begin{thm}\label{cpttrace} Let $E$ be a self-adjoint operator on $\mathcal H$ such that its positive part $E_{+}$ is a compact operator with the eigenvalue list (with multiplicity) $\boldsymbol \lambda$. Let $\boldsymbol d\in c_{0}$ be a diagonal of $E$ such that
\[\sum_{d_{i}<0}|d_{i}|<\infty,
\]
If the positive excess $\sigma_{+}=\liminf_{\alpha\searrow 0}\delta(\alpha,\boldsymbol\lambda,\boldsymbol{d})<\infty$, then the negative part of $E$ is trace class. Moreover, 
\begin{equation}\label{cpttrace0}
\tr(E_-) \le \sum_{d_{i}<0}|d_{i}| + \sigma_{+},
\end{equation}
with the equality when  \(\sum_{\lambda_{i}>0}\lambda_{i} < \infty\).
\end{thm}

\begin{proof} Suppose \(\sum_{\lambda_{i}>0}\lambda_{i} < \infty\). Corollary \ref{cptschur} implies that \(\sum_{d_{i}>0}d_{i}<\infty\), and hence \(\boldsymbol{d}\in\ell^1\). 
Our assumption that \(\sum_{\lambda_{i}>0}\lambda_{i}<\infty\) implies that \(E_{+}\) is trace-class. Let \((\tilde{d}_{i})\) denote the diagonal of \(E_{-}\) with respect to the same basis as \(E\) has diagonal \(\boldsymbol{d}\). Since \(E+E_{-}=E_{+}\), we see that 
\[\sum(d_{i} + \tilde{d}_{i}) = \tr(E_{+}) = \sum_{\lambda_{i}>0}\lambda_{i}<\infty.\]
However, since \(\boldsymbol{d}\) is summable, this implies that \((\tilde{d}_{i})\) is summable. Since \(E_{-}\) is a positive operator, this implies that \(E_{-}\) is trace-class. Therefore, $E$ is trace class and we have $\sum\lambda_{i} = \sum d_{i}$. Proposition \ref{LR} and rearranging gives
\[
\sigma_{+} = \sum_{\lambda_{i}>0}\lambda_{i} - \sum_{d_{i}>0}d_{i}  = \sum_{d_{i}<0}d_{i}-\sum_{\lambda_{i}<0}\lambda_{i} = 
\tr(E_-) - \sum_{d_{i}<0}|d_{i}| ,
\]
which implies the desired conclusion \eqref{cpttrace0} with equality.

Thus, we may assume \(\sum_{\lambda_{i}>0}\lambda_{i}=\infty\). 
Corollary \ref{cptschur} implies  \(\sum_{d_{i}>0}d_{i}=\infty\).
The set \(\{i  : \lambda_{i}>0\}\) is infinite, and hence we can rearrange the positive eigenvalues in nonincreasing order such that \((\boldsymbol\lambda_{+})^{\downarrow} = (\lambda_{i})_{i\in\N}\). Similarly, we reindex the diagonal \(\boldsymbol{d}\) by the set \(\N\cup I_{1}\) such that \((\boldsymbol{d}_{+})^{\downarrow} = (d_{i})_{i\in\N}\). Note that \(d_{i}\leq 0\) for all \(i\in I_{1}.\)

Let $(f_{i})_{i\in\N}$ be an orthonormal sequence of eigenvectors $Ef_{i} = \lambda_{i}f_i$ for all $i\in \N$. Let $(e_{i})_{i\in \N\cup I_{1}}$ be an orthonormal basis such that $d_{i} = \langle E e_{i},e_{i}\rangle$ for all \(i\in\N\cup I_{1}\). 
For each \(j\in \N \) we define 
\[a_{j} =\sum_{i\in \N}|\langle e_{i},f_{j}\rangle|^{2}.\]
Let $P$ be the orthogonal projection of $\mathcal H$ onto $(\operatorname{span} \{f_j: j\in \N \})^\perp  $. By Proposition \ref{LR} and Lemma \ref{nel} we have
\begin{equation}\label{cpttrace9}\sigma_{+} = \liminf_{M\to\infty}\sum_{i=1}^{M}(\lambda_{i} - d_{i}) \geq \sum_{j=1}^{\infty}\lambda_{j}(1-a_{j}) -\sum_{i\in \N} \langle EP e_i,e_i \rangle .\end{equation}
Hence, the assumption that \(\sigma_{+}<\infty\) and the fact that $E_-=-EP\ge 0$ yield
\[
 \sum_{j\in\N}\lambda_{j}(1-a_{j})<\infty\quad\text{and}\quad\sum_{i\in \N} |\langle EP e_i,e_i \rangle|<\infty.\]
Since $(e_{i})_{i\in \N\cup I_{1}}$ is an orthonormal basis, by Fubini's Theorem we have
\[
\sum_{j\in \N}\lambda_{j}(1-a_{j}) = \sum_{j\in \N} \lambda_j \sum_{i\in I_0} |\langle e_{i},f_{j}\rangle|^{2}
 = 
\sum_{i\in I_1}
\sum_{j\in \N}
\lambda_{j}|\langle e_{i},f_{j}\rangle|^2
.\]
By the assumption that \(\sum_{i\in I_{1}}|d_{i}| = \sum_{d_{i}<0}|d_{i}|<\infty\) the following series converge
\[
\sum_{i\in I_{1}}d_{i} = \sum_{i\in I_{1}}\left(\sum_{j\in\N}\lambda_{j}|\langle e_{i},f_{j}\rangle|^2 + \langle EP e_i,e_i \rangle
 \right). \]
This implies that the series $\sum_{i\in I_1} \langle EP e_i,e_i \rangle$ converges as well and
 \begin{equation}\label{cpttrace10}
 \sum_{i\in I_1}
| \langle EP e_i,e_i \rangle | = \sum_{i\in I_{1}} |d_{i}| + \sum_{j\in \N}\lambda_{j}(1-a_{j}).
\end{equation}
Combining \eqref{cpttrace9} and \eqref{cpttrace10} yields
\[
\tr(E_-) = \sum_{i\in \N \cup I_1} |\langle EP e_i,e_i \rangle| \le \sum_{i \in I_1}|d_{i}| + \delta^+.
\]
Therefore, the operator $E_-=-EP$ is trace class and \eqref{cpttrace0} holds.
\end{proof}

\begin{cor}\label{cpttracec} Let $\boldsymbol\lambda,\boldsymbol{d}\in c_{0}$. Assume \(E\) is a compact self-adjoint operator with eigenvalue list \(\boldsymbol\lambda\) and diagonal \(\boldsymbol{d}\). If
\[\sum_{d_{i}<0}|d_{i}|<\infty,
\]
then
\begin{equation}\label{cpttrace00}
\sigma_{+}=\liminf_{\alpha\searrow 0}\delta(\alpha,\boldsymbol\lambda,\boldsymbol{d})\geq \sum_{\lambda_{i}<0}|\lambda_{i}|-\sum_{d_{i}<0}|d_{i}| = \sigma_{-}.
\end{equation}
\end{cor}

\begin{proof}
Since the desired conclusion is obvious in the case that \(\sigma_{+}=\infty\), we may assume \(\sigma_{+}<\infty\). In this case Theorem \ref{cpttrace} yields the required conclusion \eqref{cpttrace00}.
\end{proof}

We finish this section the following cardinality inequalities which are implied  by decoupling.

\begin{lem}\label{card}
Let $E$ be a self-adjoint operator on $\mathcal H$ with the eigenvalue list (with multiplicity) $\boldsymbol \lambda$, which is possibly an empty list. Let $\boldsymbol d$ be a diagonal of $E$ with respect to some orthonormal basis $(e_i)_{i\in I}$. Assume that the conclusion of Proposition \ref{p211} holds, that is, $E$ decouples at $0$. If the positive part $E_{+}$ is compact, then
\begin{equation}\label{card1}
\#|\{i:\lambda_{i}=0\}|\geq \#|\{i:d_{i}=0\}|
\text{\, and \,}
\#|\{i:\lambda_{i} \ge 0\}|\geq \#|\{i:d_{i} \ge0\}|.
\end{equation}
If the negative part $E_-$ is compact, then
\begin{equation}\label{card2}
\#|\{i:\lambda_{i}=0\}|\geq \#|\{i:d_{i}=0\}|
\text{\, and \,}
\#|\{i:\lambda_{i} \le 0\}|\geq \#|\{i:d_{i} \le0\}|.
\end{equation}
\end{lem}

\begin{proof}
Suppose that the positive part $E_{+}$ is compact. Define the spaces $\mathcal H_0$ and $\mathcal H_1$ as in \eqref{p210}. 
If $d_i=0$ for some $i$, then $\langle E|_{\mathcal H_1} e_i, e_i \rangle =0$. Since $ E|_{\mathcal H_1}$ is positive,  the vector $e_i \in \ker E|_{\mathcal H_1} \subset \ker E$. Hence,  the first inequality in \eqref{card1} follows.
By \eqref{p212}, the eigenvalue list of the positive compact operator $E|_{\mathcal H_1}$ includes all positive eigenvalues of $E$ and some of zero eigenvalues of $E$ (with multiplicity).
Hence,
\[
\#|\{i:\lambda_{i} \ge 0\}| \ge \dim \mathcal H_1 = \#|\{i:d_{i} \ge0\}|.
\]
The proof of \eqref{card2} follows by symmetry.
\end{proof}

\section{Preliminary diagonal-to-diagonal results} \label{S4}

The goal of this section is to prove preliminary diagonal-to-diagonal results. We will use several of such results shown in the joint work of the authors with Siudeja \cite{unbound}.

The following result is a symmetric variant of \cite[Theorem 3.6]{unbound} for nonincreasing sequences. It follows by applying \cite[Theorem 3.6]{unbound} to the sequences $(-\lambda_{i})$ and $(-d_{i})$.

\begin{prop}\label{posSH} Let $\boldsymbol\lambda=(\lambda_{i})_{i=1}^{\infty}$ and $\boldsymbol d = (d_{i})_{i=1}^{\infty}$ be nonincreasing sequences and define
\[\delta_{n} = \sum_{i=1}^{n}(\lambda_{i}-d_{i}).\]
If $\boldsymbol\lambda$ is a diagonal of a self-adjoint operator $E$, $\delta_{n}\geq 0$ for all $n\in\N$, and
\[\liminf_{n\to\infty}\delta_{n} = 0,\]
then $\boldsymbol d$ is also a diagonal of $E$.
\end{prop}

The following three lemmas are \cite[Lemma 2.7]{unbound}, \cite[Lemma 3.1]{unbound}, and \cite[Lemma 3.2]{unbound}, respectively.

\begin{lem}\label{offdiag} Let $E$ be a symmetric operator on $\mathcal D \subset \mathcal H$. Assume that real numbers $d_1$, $d_2$, $\tilde d_1$, $\tilde d_2$ satisfy
\begin{equation}\label{oft}
d_{1},d_{2}\in[\tilde{d}_{1},\tilde{d}_{2}],\qquad \tilde{d}_{1}\neq \tilde{d}_{2}, \qquad \text{and}\qquad \tilde{d}_{1}+\tilde{d}_{2}=d_{1}+d_{2}.
\end{equation}
 If there exists an orthonormal set $\{f_{1},f_{2}\} \subset \mathcal D$ such that $\langle Ef_{i},f_{i}\rangle = \tilde{d}_{i}$ for $i=1,2$, then there exists
\[
\frac{\tilde{d}_{2}-d_{1}}{\tilde{d}_{2}-\tilde{d}_{1}}
\le \alpha \le 1
\]
and $\theta\in[0,2\pi)$ such that $\langle Ee_{i},e_{i}\rangle = d_{i}$ for $i=1,2$, where
\begin{equation}\label{offdiag0}e_{1}=\sqrt{\alpha}f_{1} + \sqrt{1-\alpha}\,e^{i\theta}f_{2}\qquad\text{and}\qquad e_{2}=\sqrt{1-\alpha}f_{1}-\sqrt{\alpha}\,e^{i\theta}f_{2}.\end{equation}
Moreover, if $\mathcal H$ is a real Hilbert space, then $e^{i\theta} = \pm 1$. If the inequalities in \eqref{oft} are strict, then $\alpha<1$.
\end{lem}

\begin{lem}\label{loss} Let $(f_{i})_{i\in\N}$ be an orthonormal set, and let $(\alpha_{i})_{i\in\N}$ be a sequence in $[0,1]$. Set $\tilde{e}_{1}=f_{1}$ and inductively define for $i\in\N$,
\begin{equation}\label{los}
e_{i} =\sqrt{\alpha_{i}}\,\tilde{e}_{i} + \sqrt{1-\alpha_{i}}f_{i+1}\qquad\text{and}\qquad\tilde{e}_{i+1} =\sqrt{1-\alpha_{i}}\tilde{e}_{i} - \sqrt{\alpha_{i}}f_{i+1}.
\end{equation}
If for each $n\in \N$
\begin{equation}\label{loss0}
\prod_{i=n}^{\infty}(1-\alpha_{i})=0,\end{equation}
then $(e_{i})_{i\in\N}$ is an orthonormal basis for $\overline{\lspan}\{f_{i}: i\in\N\}$. In particular,
if $\alpha_{i}<1$ for all $i$ and $\sum_{i=1}^{\infty}\frac{\alpha_{i}}{1-\alpha_{i}}=\infty$, then $(e_{i})_{i\in\N}$ is an orthonormal basis for $\overline{\lspan}\{f_{i}: i\in\N\}$.
\end{lem}

\begin{lem}\label{loglem} If $(t_{n})$ is a positive nonincreasing sequence with limit zero, then
\[\sum_{n=1}^{\infty}\frac{t_{n}-t_{n+1}}{t_{n+1}}=\infty.\]
\end{lem}

The following elementary lemma enable us to reduce the proof of diagonal-to-diagonal results by passing to partitions of the diagonal sequences.

\begin{lem}\label{subdiag} Let $(\lambda_{i})_{i\in I}$ and $(d_{i})_{i\in J}$ be two sequences of real numbers. Suppose that: 
\begin{enumerate}
\item there is a set $K$ and partitions $(I_{k})_{k\in K}$ and $(J_{k})_{k\in K}$ of $I$ and $J$, respectively,
\item for every $k\in K$, if $E_k$ is any self-adjoint operator with diagonal $(\lambda_{i})_{i\in I_k}$, then $(d_{i})_{i\in J_k}$ is also a diagonal of $E_k$,
\end{enumerate}
Then, if $E$ is any self-adjoint operator with diagonal $(\lambda_{i})_{i\in I}$, then $(d_{i})_{i\in J}$ is also a diagonal of $E$.
\end{lem}

\begin{proof}
Let $E$ be a self-adjoint operator on $\mathcal H$. Let $(f_{i})_{i\in I}$ be an o.n.~basis of $\mathcal H$ with respect to which $E$ has diagonal $(\lambda_{i})_{i\in I}$. For each $k\in K$ consider the space $\Hil_{k} = \overline{\lspan}\{ f_i : i \in I_k\}$. Let $E_k$ be a compression of $E$ to $\mathcal H_k$. That is, $E_k = (T_k)^* E T_k$, where $T_k: \mathcal H_k \to \mathcal H$ denotes the natural embedding of $\mathcal H_k$ into $\Hil$. By (ii) there exists an o.n.~basis $(e_{i})_{i\in J_k}$ of $\mathcal H_k$ with respect to  which $E_k$ has diagonal $(d_{i})_{i\in J_{k}}$. Then, $E$ has diagonal $(d_{i})_{i\in J}$ with respect to o.n.~basis $(e_{i})_{i\in J}$ of $\mathcal H$.
\end{proof}

\begin{remark} Observe that if $I_{k} = J_{k}$ for some $k\in K$ and $ \lambda_{i}=d_i$ for all $i\in I_{k}$, then condition (ii) in Lemma \ref{subdiag} is automatically satisfied.
\end{remark}

Our next goal is to prove diagonal-to-diagonal result in the case where $\boldsymbol \lambda$ has exactly one negative term, $\boldsymbol d$ is strictly positive, and we have {\it dominant majorization} $\lambda_{i}\geq d_{i}$ for all $i\in\N$. In the following result will apply Lemma \ref{loss} to this setup. However, we need a technical assumption, which will be later removed for Theorem \ref{1neg}.

\begin{lem}\label{tbound} Let $(\lambda_{i})_{i\in\N}$ and $\boldsymbol d=(d_{i})_{i\in\N}$ be positive sequences with $\lambda_{i}\geq d_{i}$ for all $i\in\N$. Assume
\[\lambda_{-1} := \sum_{i=1}^{\infty}(\lambda_{i}-d_{i})\in (0,\infty),\]
and define the sequence
\[t_{n} = \sum_{i=n}^{\infty}(\lambda_{i} - d_{i}).\]
If there is some $c>0$ such that $\lambda_{n}\leq ct_{n+1}$ for all $n\in\N$ and there is a self-adjoint operator $E$ with diagonal $\boldsymbol \lambda:=(-\lambda_{-1},\lambda_{1},\lambda_{2},\ldots)$, then $\boldsymbol d$ is also a diagonal of $E$.
\end{lem}

\begin{proof} First, we will assume that $\lambda_{i}>d_{i}$ for all $i\in\N$. For each $n\in\N$ set
\[\tilde{\alpha}_{n} = \frac{\lambda_{n} - d_{n}}{\lambda_{n} + t_{n}}.\]
Since
\begin{equation}\label{tbound1}-t_{n}<-t_{n+1}<0< d_{n}<\lambda_{n} \qquad\text{for all }n\in\N,\end{equation}
we see that $\tilde{\alpha}_{n}<1$ for all $n\in\N$. Next, we calculate
\[\frac{\tilde{\alpha}_{n}}{1-\tilde{\alpha}_{n}} = \frac{\lambda_{n} - d_{n}}{d_{n}+t_{n}} = \frac{t_{n}-t_{n+1}}{\lambda_{n} + t_{n+1}} \geq \frac{t_{n}-t_{n+1}}{ct_{n+1}+t_{n+1}} = \frac{1}{(1+c)}\frac{t_{n}-t_{n+1}}{t_{n+1}}\]
By Lemma \ref{loglem}, the we have $\sum_{n=1}^{\infty}\frac{t_{n}-t_{n+1}}{t_{n+1}} = \infty$, and hence $\sum_{n=1}^{\infty}\frac{\tilde{\alpha}_{n}}{1-\tilde{\alpha}_{n}}=\infty$.

Let $(f_{i})_{i=1}^{\infty}$ be an orthonormal basis such that 
\[\langle Ef_{i},f_{i}\rangle = \begin{cases} -\lambda_{-1} & i=1\\ \lambda_{i-1} & i\geq 2.\end{cases}\]

From \eqref{tbound1} we see that we can apply Lemma \ref{offdiag} to find $\theta_{2}\in[0,2\pi)$ and $\alpha_{1}\in[\tilde{\alpha}_{1},1)$ so that the vectors
\[e_{1}=\sqrt{\alpha_{1}}f_{1} + \sqrt{1-\alpha_{1}}e^{i\theta_{2}}f_{2}\qquad\text{and}\qquad \tilde{e}_{2}=\sqrt{1-\alpha_{1}}f_{1}-\sqrt{\alpha_{1}}e^{i\theta_{2}}f_{2}\]
form an orthonormal basis for $\lspan\{f_{1},f_{2}\}$ and $\langle Ee_{1},e_{1}\rangle = d_{1}$ and $\langle E\tilde{e}_{2},\tilde{e}_{2}\rangle = -t_{2}$.

Next, we will show that for each $n\in\N$ we have a othonormal basis $\{e_{1},e_{2},\ldots,e_{n-1},\tilde{e}_{n}\}$ for $\lspan\{f_{1},\ldots,f_{n}\}$ such that
\[\langle Ee_{j},e_{j}\rangle = d_{j}\quad \text{for } j\leq n-1\qquad\text{and}\qquad \langle E\tilde{e}_{n},\tilde{e}_{n}\rangle = -t_{n}.\]
Indeed, assume we have such an othonormal basis for some $n\geq 2$.  Again, from \eqref{tbound1} and Lemma \ref{offdiag} we see that there is some $\theta_{n+1}\in[0,2\pi)$ and $\alpha_{n}\in[\tilde{\alpha}_{n},1)$ such that the vectors
\[e_{n}=\sqrt{\alpha_{n}}\tilde{e}_{n} + \sqrt{1-\alpha_{n}}e^{i\theta_{n+1}}f_{n+1}\qquad\text{and}\qquad \tilde{e}_{n+1}=\sqrt{1-\alpha_{n}}\tilde{e}_{n}-\sqrt{\alpha_{n}}e^{i\theta_{n+1}}f_{n+1}\]
are an orthonormal basis for $\lspan\{\tilde{e}_{n},f_{n+1}\}$ and $\langle Ee_{n},e_{n}\rangle = d_{n}$ and $\langle E\tilde{e}_{n+1},\tilde{e}_{n+1}\rangle = -t_{n+1}$. This completes the inductive step, and we have the desired orthonormal basis for each $n\in\N$.

Note that the sequence $(e_{i})_{i=1}^{\infty}$ given by the above procedure is also obtained by applying Lemma \ref{loss} to $(e^{i\theta_{n}}f_{n})_{n=1}^{\infty}$ with $(\alpha_{n})_{n=1}^{\infty}$ and $(\theta_{n})_{n=2}^{\infty}$ as defined above, and $\theta_{1} = 0$. Since $\tilde{\alpha}_{n} \leq \alpha_{n}<1$ for all $n\in\N$, we have
\[\sum_{n=1}^{\infty}\frac{\alpha_{n}}{1-\alpha_{n}}\geq \sum_{n=1}^{\infty}\frac{\tilde{\alpha}_{n}}{1-\tilde{\alpha}_{n}} = \infty.\]
By Lemma \ref{loss}, since $(f_{i})_{i=1}^{\infty}$ is an orthonormal basis, the sequence $(e_{i})_{i=1}^{\infty}$ is also an orthonormal basis. This completes the proof under the assumption that $\lambda_{i}> d_{i}$ for all $i\in\N$.

Since $\lambda_{n}>0$ for all $n\in\N$, we have $t_{n+1}\geq \lambda_{n}/c>0$ for all $n\in\N$, that is, the inequality $\lambda_{n}\geq d_{n}$ is strict for infinitely many $n\in\N$. Let $n_{k}$ denote the $k$th integer $n$ such that $\lambda_{n}>d_{n}$. Set $\tilde{d}_{k} = d_{n_{k}}$ and $\tilde{\lambda}_{k} = \lambda_{n_{k}}$ for each $k\in\N$. Note that $\tilde{\lambda}_{k}>\tilde{d}_{k}$ for all $k\in\N$ and
\[\sum_{k=j+1}^{\infty}(\tilde{\lambda}_{k}-\tilde{d}_{k}) = \sum_{i=n_{j+1}}^{\infty}(\lambda_{i}-d_{i}) = t_{n_{j}+1}\geq \lambda_{n_{j}}/c = \tilde{\lambda}_{j}/c.\]
Set $\tilde{\lambda}_{-1} = \lambda_{-1}$, $J_{1} = \{i\in\N : \lambda_{i}>d_{i}\} = \{n_{1},n_{2},\ldots\}$, $I_{1} = \{-1\}\cup J_{1}$, and $I_{2} = J_{2} = \{i\in\N : \lambda_{i} = d_{i}\}$. By the above argument, if $E_{1}$ is any self-adjoint operator with diagonal $(\tilde{\lambda}_{k})_{k\in\N\cup\{-1\}} = (\sgn(i)\lambda_{i})_{i\in I_{1}}$, then $(\tilde{d}_{k})_{k\in\N} = (d_{i})_{i\in J_{1}}$ is also a diagonal of $E_{1}$.  Since $(\sgn(i)\lambda_{i})_{i\in\N\cup\{-1\}}$ is a diagonal of $E$, by Lemma \ref{subdiag} $(d_{i})_{i\in\N}$ is also a diagonal of $E$.
\end{proof}

The following result enables us to rearrange sequences satisfying dominant majorization into nonincreasing order.

\begin{lem}\label{dra} Let $(\lambda_{i})_{i\in\N}$ and $\boldsymbol d = (d_{i})_{i\in\N}$ be positive sequences in $c_{0}$ such that $\lambda_{i}\geq d_{i}$ for all $i\in\N$. If
\[\sigma:=\sum_{i=1}^{\infty}(\lambda_{i}-d_{i})<\infty,\]
then $\lambda_{i}^{\downarrow}\geq d_{i}^{\downarrow}$ for all $i\in\N$, and
\[\sum_{i=1}^{\infty}(\lambda_{i}^{\downarrow}-d_{i}^{\downarrow}) = \sigma.\]
\end{lem}

\begin{proof} Let $\Pi_{d}$ and $\Pi_{\lambda}$ be permutations of $\N$ so that $(d_{\Pi_{d}(i)})$ and $(\lambda_{\Pi_{\lambda}(i)})$ are nonincreasing.
For each $N\in\N$ and $i\leq N$ we have
\[d_{\Pi_{d}(N)}\leq d_{\Pi_{d}(i)}\leq \lambda_{\Pi_{d}(i)}.\]
That is $d_{\Pi_{d}(N)}\leq \lambda_{i}$ for at least $N$ distinct numbers $i$. Hence $d_{\Pi_{d}(N)}$ is less than the $N$th largest term of $(\lambda_{i})$, that is,
\[d_{N}^{\downarrow} = d_{\Pi_{d}(N)}\leq \lambda_{\Pi_{\lambda}(N)} = \lambda_{N}^{\downarrow}.\]

Since the sum in the definition of $\sigma$ contains only positive terms, we can rearrange the terms without effecting the sum, hence
\[\sigma = \sum_{i=1}^{\infty}(\lambda_{\Pi_{d}(i)}-d_{\Pi_{d}(i)}) = \sum_{i=1}^{\infty}(\lambda_{\Pi_{\lambda}(i)}-d_{\Pi_{\lambda}(i)}).\]
Note that for each $N\in\N$ we have
\[\sum_{i=1}^{N}(\lambda_{\Pi_{d}(i)}-d_{\Pi_{d}(i)})\leq \sum_{i=1}^{N}(\lambda_{\Pi_{\lambda}(i)}-d_{\Pi_{d}(i)})\leq \sum_{i=1}^{N}(\lambda_{\Pi_{\lambda}(i)}-d_{\Pi_{\lambda}(i)}).\]
Letting $N\to\infty$ we see that
\[\sigma = \sum_{i=1}^{\infty}(\lambda_{\Pi_{\lambda}(i)}-d_{\Pi_{d}(i)}) = \sum_{i=1}^{\infty}(\lambda_{i}^{\downarrow}-d_{i}^{\downarrow}) .\]
\end{proof}

The following is the main result of this section. 

\begin{thm}\label{1neg} Let $(\lambda_{i})_{i\in\N}$ and $\boldsymbol d=(d_{i})_{i\in\N}$ be positive sequences in $c_{0}$ such that $\lambda_{i}\geq d_{i}$ for all $i\in\N$ and 
\[\lambda_{-1} := \sum_{i=1}^{\infty}(\lambda_{i}-d_{i})\in(0,\infty).\]
If there is a self-adjoint operator $E$ with diagonal $\boldsymbol \lambda:=(-\lambda_{-1},\lambda_{1},\lambda_{2},\ldots)$, then $\boldsymbol d$ is also a diagonal of $E$.
\end{thm}

\begin{proof} 
By Lemma \ref{dra} we can assume $(\lambda_{i})_{i\in\N}$ and $\boldsymbol d$ are in nonincreasing order.

Since $\lambda_{-1}>0$ we see that there is some $n\in\N$ such that $\lambda_{n}>d_{n}$. Assume that for every $i\in\N$ either $\lambda_{i}=d_{i}$ or $\lambda_{i}=\lambda_{i+1}$. Since $\lambda_{n}>d_{n}$ it must be the case that $\lambda_{n}=\lambda_{n+1}$. Let $k$ be the largest integer so that $\lambda_{n}=\lambda_{n+k}$. Since $\lambda_{n+k}>\lambda_{n+k+1}$, we must have $\lambda_{n+k}=d_{n+k}$. But then $d_{n}<\lambda_{n}=\lambda_{n+k}=d_{n+k}$. This contradicts the assumption that $\boldsymbol d$ is nonincreasing. Thus, there must be some index $n$ such that both $\lambda_{n}>d_{n}$ and $\lambda_{n}>\lambda_{n+1}$. Let $n_{0}$ be the smallest such index. 

Fix $n_{1}>n_{0}$ such that $\lambda_{n_{1}+1}<1/2$. Now for each $k\geq 2$ fix $n_{k}\in\N$ such that $n_{k}>2n_{k-1}-n_{k-2}$ and $\lambda_{n_{k}+1}<2^{-k}$. Define the sequence $(\gamma_{i})_{i=n_{0}+1}^{\infty}$ by setting $\gamma_{i} = 2^{-k}(n_{k}-n_{k-1})^{-1}$ for $i=n_{k-1}+1,\ldots,n_{k}$. The choice of $n_{k}$ so that $n_{k}\geq 2n_{k-1}-n_{k-2}$ implies that $(\gamma_{i})$ nonincreasing sequence, and the sequence is clearly positive. Moreover,
\begin{equation}\label{1neg.2}\sum_{i=n_{0}+1}^{\infty}\gamma_{i} = \sum_{k=1}^{\infty}\sum_{i=n_{k-1}+1}^{n_{k}}\frac{1}{2^{k}(n_{k}-n_{k-1})} = \sum_{k=1}^{\infty}\frac{1}{2^{k}} = 1.\end{equation}

For some $\alpha>0$ define the sequence
\[\tilde{\lambda}_{i} = \begin{cases} \lambda_{i} & i\leq n_{0}-1\\ \lambda_{n_{0}}-\alpha & i=n_{0}\\ \lambda_{i} + \alpha\gamma_{i} & i\geq n_{0}+1.\end{cases}\]
Since $\lambda_{n_{0}}>\lambda_{n_{0}+1}$, we can choose $\alpha>0$ small enough that $\tilde{\lambda}_{n_{0}}\geq \tilde{\lambda}_{n_{0}+1}$. Since $(\gamma_{i})$ is nonincreasing, we see that $(\tilde{\lambda}_{i})_{i=1}^{\infty}$ is also nonincreasing.

Next, we calculate
\[\sum_{i=1}^{N}(\lambda_{i} - \tilde{\lambda}_{i}) = \begin{cases} 0 & N\leq n_{0}-1\\ \alpha - \alpha\displaystyle{\sum_{i=n_{0}+1}^{N}\gamma_{i}} & N\geq n_{0}.\end{cases}\]
From \eqref{1neg.2} and Proposition \ref{posSH} we see that if $E_{1}$ is any self-adjoint operator with diagonal $(\lambda_{i})_{i\in\N}$, then $(\tilde{\lambda}_{i})_{i\in\N}$ is also a diagonal of $E_{1}$. Set $I_{1} = J_{1} = \N$ and $I_{2} = J_{2} = \{-1\}$. Since $(\sgn(i)\lambda_{i})_{i\in I_{1}\cup I_{2}}$ is a diagonal of $E$, by Lemma \ref{subdiag} the sequence $(-\lambda_{-1},\tilde{\lambda}_{1},\tilde{\lambda}_{2},\ldots)$ is also a diagonal of $E$.

 Note that $\tilde{\lambda}_{i}\geq \lambda_{i}\geq d_{i}$ for all $i\in\N\setminus\{n_{0}\}$, and by choosing $\alpha<\lambda_{n_{0}}-d_{n_{0}}$ we also have $\tilde{\lambda}_{n_{0}} >d_{n_{0}}$. For $n\in\N$ such that $n_{k-1}+1\leq n\leq n_{k}$
\[t_{n+1}: = \sum_{i=n+1}^{\infty}(\tilde{\lambda}_{i}-d_{i})  = \sum_{i=n+1}^{\infty}(\lambda_{i} - d_{i}) + \alpha\sum_{i=n+1}^{\infty}\gamma_{i} \geq \alpha\sum_{i=n_{k}+1}^{\infty}\gamma_{i} = \alpha\sum_{j=k}^{\infty}\frac{1}{2^{j+1}} = \frac{\alpha}{2^{k}}.\]
For the same $n$ we have
\[\tilde{\lambda}_{n} = \lambda_{n} + \frac{\alpha}{2^{k}(n_{k}-n_{k-1})}\leq \lambda_{n_{k-1}+1} + \frac{\alpha}{2^{k}(n_{k}-n_{k-1})}< \frac{1}{2^{k-1}} + \frac{\alpha}{2^{k}(n_{k}-n_{k-1})}.\]
From these estimates we deduce
\[\frac{\tilde{\lambda}_{i}}{t_{n+1}}\leq \frac{\frac{1}{2^{k-1}} + \frac{\alpha}{2^{k}(n_{k}-n_{k-1})}}{\frac{\alpha}{2^{k}}} = \frac{2}{\alpha} + \frac{1}{n_{k}-n_{k-1}}\leq \frac{2}{\alpha}+1=:C_{1}\]
This shows that $\tilde{\lambda}_{n}\leq C_{1}t_{n+1}$ for all $n\geq n_{0}+1$, and by possibly making $C_{1}$ larger, we have this inequality for all $n\in\N$. Finally, we can apply Lemma \ref{tbound} to see that $(d_{i})_{i=1}^{\infty}$ is a diagonal of $E$.\end{proof}

\section{Diagonal-to-diagonal results with disappearing negative terms} \label{S5}

The goal of this section is to generalize Theorem \ref{1neg} to the case when $\boldsymbol \lambda$ contains arbitrary number of negative terms and infinite number of positive terms, and the dominant majorization assumption is dropped. That is, we show a diagonal-to-diagonal result where all negative terms in $\boldsymbol \lambda$ disappear and produce a positive diagonal $\boldsymbol d$.

We start with elementary lemma which will be used at a crucial point in the proof of Lemma \ref{infneglem}.

\begin{lem}\label{1convmove} Let $(\lambda_{i})_{i=1}^{N}$ be a nonincreasing sequence and let $0\leq\Delta\leq \lambda_{1}-\lambda_{N}$. Define the sequence $(\tilde{\lambda}_{i})_{i=1}^{N}$ by
\[\tilde{\lambda}_{i} = \begin{cases} \lambda_{1}-\Delta & i=1\\ \lambda_{i} & i=2,\ldots,N-1\\ \lambda_{N}+\Delta & i=N.\end{cases}\]
Then,
\begin{equation}\label{1convmove0}
\sum_{i=1}^{k}(\lambda_{i} - \tilde{\lambda}_{i}^{\downarrow})\geq 0\quad\text{for }k=1,2,\ldots,N.
\end{equation}
\end{lem}

\begin{proof}
Since $(\lambda_1-\Delta,\lambda_N+\Delta) \preccurlyeq (\lambda_1,\lambda_N )$ we have $(\tilde{\lambda}_{i})_{i=1}^{N} \preccurlyeq (\lambda_{i})_{i=1}^{N}$. Hence, \eqref{1convmove0} follows.
\end{proof}

Lemma \ref{infneglem} is a version of Theorem \ref{1neg} where \(\boldsymbol\lambda\) has infinitely many negative terms, but at the cost that we must also assume that the inequalities in dominant majorization are all strict. In Theorem \ref{infneg} we show that this additional assumption can be removed.

\begin{lem}\label{infneglem} Let $(\lambda_{i})_{i\in\N}$ and $\boldsymbol d=(d_{i})_{i\in\N}$ be positive sequences in $c_{0}$ such that $\lambda_{i}> d_{i}$ for all $i\in\N$ and 
\begin{equation}\label{inl0}\sigma := \sum_{i=1}^{\infty}(\lambda_{i}-d_{i})\in(0,\infty).\end{equation}
Let $(\lambda_{i})_{i\in-\N}$ be a positive sequence with
\[\sum_{i\in-\N}\lambda_{i} = \sigma.\]
If there is a self-adjoint operator $E$ with diagonal $\boldsymbol\lambda: = (\sgn(i)\lambda_{i})_{i\in\N\cup-\N}$, then $\boldsymbol d$ is also a diagonal of $E$.
\end{lem}

\begin{proof}
By Lemma \ref{dra} we can assume $(\lambda_{i})_{i\in\N}$ and $\boldsymbol d$ are in nonincreasing order.

For each $k\in\N\cup\{0\}$ we will inductively construct a set $I_{k}\subset \N$ and a sequence $\boldsymbol{\lambda}^{(k)} = (\lambda_{i}^{(k)})_{i\in\N}$ with the following properties:
\begin{equation}
\label{inl1} I_{k+1}\subset I_{k}
\end{equation}
\begin{equation}
\label{inl2} \#|I_{k}| = \#|I_{k}\setminus I_{k+1}| = \infty
\end{equation}
\begin{equation}
\label{inl3}
\{1,2,\ldots,k\}\subset \N\setminus I_{k}
\end{equation}
\begin{equation}
\label{inl4} \lambda_{i}^{(k)} = \lambda_{i}^{(k+1)}\quad\text{for all }i\in \N\setminus I_{k}
\end{equation}
\begin{equation}
\label{inl5}\lambda_{i}^{(k)}\geq d_{i}\quad\text{for all }i\in\N
\end{equation}
\begin{equation}
\label{inl6}\sum_{i\in \N}(\lambda_{i}^{(k)} - d_{i}) = \sigma.
\end{equation}
\begin{equation}
\label{inl7}\sum_{i\in I_{k}\setminus I_{k+1}}(\lambda_{i}^{(k+1)} - d_{i}) = \lambda_{-k-1}.
\end{equation}
\begin{equation}
\label{inl8}\delta(\alpha,\boldsymbol{\lambda}^{(k)},\boldsymbol\lambda^{(k+1)})\geq 0\quad \text{for all }\alpha>0
\end{equation}
First, set $I_{0} = \N$ and $\lambda_{i}^{(0)} = \lambda_{i}$ for all $i\in\N$.
Now, assume that for some $n\in\N\cup\{0\}$ and for all $k\in\{0,\ldots,n\}$ we have constructed the sets $I_{k}$ and the sequences $(\lambda_{i}^{(k)})_{i\in\N}$ satisfying \eqref{inl1}--\eqref{inl8}. By \eqref{inl4} and \eqref{inl7}  we have
\[
\sum_{i\in I_{k}\setminus I_{k+1}}(\lambda_{i}^{(n)} - d_{i}) = \lambda_{-k-1} \qquad\text{for } k=0,\ldots,n-1.
\]
Summing the above equations yields
\[
\sum_{i\in I_0 \setminus I_{n}}(\lambda_{i}^{(n)} - d_{i}) = \sum_{k=0}^{n-1} \lambda_{-k-1}.
\]
Combining this with \eqref{inl6} for $k=n$ yields
\begin{equation}\label{inl9}
\sum_{i\in I_{n}}(\lambda_{i}^{(n)} - d_{i}) =  \sigma - \sum_{k=0}^{n-1} \lambda_{-k-1}
= \sum_{k=n+1}^\infty \lambda_{-k} 
>\lambda_{-n-1}.\end{equation}

Define $m_{n+1} = \min I_{n}$, and note that by \eqref{inl3} we see that $m_{n+1}\geq n+1$. By \eqref{inl9} there is a smallest number $r_{n+1}$ so that
\[\sum_{i\in I_{n},\ i\leq r_{n+1}}(\lambda_{i}^{(n)}-d_{i}) \geq \lambda_{-n-1}.\]
By the minimality of $r_{n+1}$ we have
\[\sum_{i\in I_{n},\ i< r_{n+1}}(\lambda_{i}^{(n)} - d_{i})<\lambda_{-n-1}.\]
Hence we can find a set $J_{n+1}\subset \{i\in I_{n} : i>r_{n+1}\}$ such that $\#|J_{n+1}| = \#|I_{n}\setminus J_{n+1}| = \infty$, and
\[\sum_{i\in J_{n+1}}(\lambda_{i}^{(n)} - d_{i}) + \sum_{i\in I_{n},\ i< r_{n+1}}(\lambda_{i}^{(n)} - d_{i})<\lambda_{-n-1}.\]
Set $I_{n+1} = I_{n}\setminus\big(J_{n+1}\cup \{i\in I_{n} : i\leq r_{n+1}\}\big)$. Hence, we have a set $I_{n+1}$ satisfying \eqref{inl1}--\eqref{inl3} and a number $r_{n+1}$ such that
\begin{equation}\label{inl10}\sum_{i\in I_{n}\setminus I_{n+1}}(\lambda_{i}^{(n)} - d_{i})\geq \lambda_{-n-1}>\sum_{i\in (I_{n}\setminus I_{n+1})\setminus\{r_{n+1}\}}(\lambda_{i}^{(n)} - d_{i}).\end{equation}

Set
\[\Delta = \sum_{i\in I_{n}\setminus I_{n+1}}(\lambda_{i}^{(n)} - d_{i}) - \lambda_{-n-1}.\]
Fix $s_{n+1}\in I_{n+1}$ such that $d_{r_{n+1}}\geq\lambda_{s_{n+1}}^{(n)}$. Define the sequence $\boldsymbol\lambda^{(n+1)}:=(\lambda_{i}^{(n+1)})_{i\in\N}$ by
\[\lambda_{i}^{(n+1)} = \begin{cases} \lambda_{i}^{(n)} & i \neq r_{n+1},s_{n+1},\\ \lambda_{r_{n+1}}^{(n)} - \Delta & i=r_{n+1},\\ \lambda_{s_{n+1}}^{(n)} + \Delta & i=s_{n+1}.\end{cases}.\]
Since $r_{n+1},s_{n+1}\in I_{n}$, we see that \eqref{inl4} holds for $k=n$. From \eqref{inl1} we have
\begin{align*}
\lambda_{r_{n+1}}^{(n)} - \Delta &  = \lambda_{r_{n+1}}^{(n)} - \sum_{i\in I_{n}\setminus I_{n+1}}(\lambda_{i}^{(n)} - d_{i}) + \lambda_{-n-1}\\
 & = d_{r_{n+1}} + \lambda_{-n-1} - \sum_{i\in (I_{n}\setminus I_{n+1})\setminus\{r_{n+1}\}}(\lambda_{i}^{(n)} - d_{i})>d_{r_{n+1}}
\end{align*}
From this we deduce that \eqref{inl5} holds for $k=n+1$. Both \eqref{inl6} and \eqref{inl7} clearly hold. Finally, \eqref{inl8} holds by Lemma \ref{1convmove} and the choice of $s_{n+1}$.

By induction we have a collection of sets $(I_{k})_{k=0}^{\infty}$ and sequences $(\boldsymbol\lambda^{(k)})_{k=0}^{\infty}$ satisfying \eqref{inl1}--\eqref{inl8}. For each $k\in\N\cup\{0\}$ define the set $\Lambda_{k} = I_{k}\setminus I_{k+1}$. By \eqref{inl1} and \eqref{inl3} we see that the sets $(\Lambda_{k})_{k=0}^{\infty}$ are disjoint and 
\[\bigcup_{k=0}^{n}\Lambda_{k} = \N\setminus I_{n+1}\supset\{1,\ldots,n+1\}.\]
Hence $(\Lambda_{k})_{k=0}^{\infty}$ is a partition of $\N$.

Define the sequence $\tilde{\boldsymbol{\lambda}}:=(\tilde{\lambda}_{i})_{i\in\N}$ by setting $\tilde{\lambda}_{i} = \lambda_{i}^{(n+1)}$ for $i\in\Lambda_{n}$. Now, \eqref{inl7} can be rewritten as
\begin{equation}\label{inl70}
\sum_{i\in\Lambda_{k}}(\tilde{\lambda}_{i} - d_{i}) = \lambda_{-k-1}.
\end{equation}
Summing over $k$ we obtain
\begin{equation}\label{inl11}\sum_{i\in\N}(\tilde{\lambda}_{i} - d_{i}) = \sigma.\end{equation}
Since $\boldsymbol d\in c_{0}$ this shows that $(\tilde{\lambda}_{i})$ is also in $c_{0}$.

Fix $\alpha>0$. The set $\{i : \tilde{\lambda}_{i}\geq \alpha\}$ finite, and hence there is some $n\in\N\cup\{0\}$ such that
\[\sum_{\tilde{\lambda}_{i}\geq \alpha}(\tilde{\lambda}_{i} - \alpha) = \sum_{\tilde{\lambda}_{i}\geq \alpha}(\lambda_{i}^{(n)} - \alpha)\leq \sum_{\lambda_{i}^{(n)}\geq \alpha}(\lambda_{i}^{(n)} - \alpha)\]
Hence, using \eqref{inl8} we have
\[\delta(\alpha,\boldsymbol\lambda,\tilde{\boldsymbol{\lambda}})\geq \delta(\alpha,\boldsymbol\lambda,\boldsymbol{\lambda}^{(n)}) \ge 0.\]
By Proposition \ref{LR} we have
\[\delta_{k}: = \sum_{i=1}^{k}(\lambda_{i}-\tilde{\lambda}_{i}^{\downarrow})\geq 0\quad\text{for all } k\in\N.\]
From \eqref{inl0} and \eqref{inl11} we see that $\delta_{k}\to 0$ as $n\to\infty$. By Proposition \ref{posSH}, if $E_{1}$ is any self-adjoint operator with diagonal $(\lambda_{i})_{i\in\N}$, then $\tilde{\boldsymbol\lambda}$ is also a diagonal of $E_{1}$. Since $(\lambda_{i})_{i\in\N}\oplus(-\lambda_{i})_{i\in-\N}$ is a diagonal of $E$, by Lemma \ref{subdiag} the sequence $(\tilde{\lambda}_{i})_{i\in\N}\oplus(-\lambda_{i})_{i\in-\N}$ is also a diagonal of $E$.

For each $k\in\N\cup\{0\}$, set $\tilde{\boldsymbol{\lambda}}^{(k)} = (\tilde{\lambda}_{i})_{i\in\Lambda_{k}}\oplus(-\lambda_{i})_{i\in\{-k-1\}}$. Using \eqref{inl5}, \eqref{inl70}, and Theorem \ref{1neg} we see that if $E_{k}$ is an operator with diagonal $\tilde{\boldsymbol{\lambda}}^{(k)}$, then $(d_{i})_{i\in\Lambda_{k}}$ is also a diagonal of $E_{k}$. We have already deduced that
\[(\tilde{\lambda}_{i})_{i\in\N}\oplus(-\lambda_{i})_{i\in-\N} = \bigoplus_{k=0}^{\infty}\tilde{\boldsymbol{\lambda}}^{(k)}\]
is a diagonal of $E$. By Proposition \ref{subdiag}
\[\boldsymbol d = \bigoplus_{k=0}^{\infty}(d_{i})_{i\in\Lambda_{k}}\]
is also a diagonal of $E$.
\end{proof}

Theorem \ref{infneg} is a generalization of Theorem \ref{1neg} to the case when \(\boldsymbol\lambda\) has infinitely many negative terms. Note that the assumption that $\sigma>0$ implies that positive terms in $\boldsymbol \lambda$ and $\boldsymbol d$ cannot entirely coincide.

\begin{thm}\label{infneg} Let $(\lambda_{i})_{i\in\N}$ and $\boldsymbol d=(d_{i})_{i\in\N}$ be positive sequences in $c_{0}$ such that $\lambda_{i}\geq d_{i}$ for all $i\in\N$ and 
\[\sigma := \sum_{i=1}^{\infty}(\lambda_{i}-d_{i})\in(0,\infty).\]
Let $(\lambda_{i})_{i\in-\N}$ be a positive sequence with
\[\sum_{i\in-\N}\lambda_{i} = \sigma.\]
If there is a self-adjoint operator $E$ with diagonal $\boldsymbol\lambda: = (\sgn(i)\lambda_{i})_{i\in\N\cup-\N}$, then $\boldsymbol d$ is also a diagonal of $E$.
\end{thm}

\begin{proof}
By Lemma \ref{dra} we can assume $(\lambda_{i})_{i\in\N}$ and $\boldsymbol d$ are in nonincreasing order.
Let $n_{0}$ be the smallest number such that
\[\lambda_{n_{0}}>d_{n_{0}} \quad\text{and}\quad \lambda_{n_{0}}>\lambda_{n_{0}+1}.\]
Indeed, if $\lambda_{n}>d_{n}$ and $n_{0}$ is the largest $n$ such that $\lambda_{n_{0}}=\lambda_{n}$, then $\lambda_{n_{0}}>\lambda_{n_{0}+1}$ and $\lambda_{n_{0}} = \lambda_{n}>d_{n}\geq d_{n_{0}}$. Fix $\eps = \min\{\lambda_{n_{0}}-d_{n_{0}},\lambda_{n_{0}}-\lambda_{n_{0}+1}\}$. Define the sequence $(\tilde{\lambda}_{i})_{i\in \N\cup-\N}$ by
\[\tilde{\lambda}_{i} = \begin{cases}
\lambda_{i} & i\leq n_{0}-1,\\
\lambda_{n_{0}}-\frac{\eps}{2} & i=n_{0},\\
\lambda_{i}+ 2^{n_{0}-i-3}\eps & i\geq n_{0}+1.\end{cases}\]
By the choice of $\eps$ the sequence $(\tilde{\lambda}_{i})$ is nonincreasing and $\tilde{\lambda}_{i}\geq d_{i}$ for all $i\in\N$, and $\tilde{\lambda}_{i}>\lambda_{i}\geq d_{i}$ for all $i\geq n_{0}+1$. Hence $\tilde{\lambda}_{i}>d_{i}$ for all but finitely many indices $i$.

Next, we compute
\[\sum_{i=1}^{k}(\lambda_{i} - \tilde{\lambda}_{i}) = \begin{cases} 0 & k\leq n_{0}-1,\\ 2^{n_{0}-k-1}\eps & k\geq n_{0}. \end{cases}\]
By Proposition \ref{posSH}, if $E_{1}$ is any operator with diagonal $(\lambda_{i})_{i\in\N}$, then $(\tilde{\lambda}_{i})_{i\in\N}$ is also a diagonal of $E_{1}$. Since $(\lambda_{i})_{i\in\N}\oplus(-\lambda_{i})_{i\in-\N}$ is a diagonal of $E$, by Lemma \ref{subdiag} $(\tilde{\lambda}_{i})_{i\in\N}\oplus(-\tilde{\lambda}_{i})_{i\in-\N}$ is also a diagonal of $E$.

Set $I_{1} := \{i\in\N : d_{i} > \tilde{\lambda}_{i}\}\cup -\N$ and $I_{2} = J_{2} = \{i\in \N : d_{i}=\tilde{\lambda}_{i}\}$. Since $I_{2}$ is finte, the set $J_{1}:=I_{1}\cap\N$ is infinite. By Lemma \ref{infneglem} if $E_{1}$ is any operator with diagonal $(\tilde{\lambda}_{i})_{i\in I_{1}}$, then $(d_{i})_{i\in J_{1}}$ is also a diagonal of $E_{1}$. Hence, by Lemma \ref{subdiag} $(d_{i})_{i\in J_{1}}\oplus(d_{i})_{i\in J_{2}} = \boldsymbol d$ is also a diagonal of $E$. 
\end{proof}

The main result of this section is Theorem \ref{NoPosDiag}, which generalizes both Theorem \ref{1neg} and Theorem \ref{infneg} by replacing the assumption of dominant majorization with usual majorization, and allows for an arbitrary number of negative terms in \(\boldsymbol\lambda\). The following lemma about sequences, which is a variant \cite[Theorem 3.4]{kw1}, allows us to reduce the proof of Theorem \ref{NoPosDiag} to the case where the inequalities in dominant majorization are eventually strict.

\begin{lem}\label{midseq} Let $(\lambda_{i})_{i\in\N}$ and $(d_{i})_{i\in\N}$ be nonincreasing sequences in $c_{0}$ such that
\[\delta_{k}: = \sum_{i=1}^{k}(\lambda_{i} - d_{i})\geq 0\quad\text{for all }k\in\N,\]
and
\[\sigma :=\liminf_{k\to\infty}\delta_{k}\in(0,\infty).\]
Then, there is a nonincreasing sequence $(\tilde{\lambda}_{i})_{i\in\N}$ such that the following four conclusions hold:
\begin{align}
\label{midseq-2}
\tilde{\lambda}_{i} \geq d_{i} & \qquad\text{for all } i\in\N,
\\
\label{midseq-3}
\tilde{\delta}_{k}: = \sum_{i=1}^{k}(\lambda_{i} - \tilde{\lambda}_{i})\geq 0&\qquad\text{for all }k\in\N,
\\
\label{midseq0}
\liminf_{k\to\infty}\tilde{\delta}_{k} &= 0,
\\
\label{midseq-1}
\sum_{i=1}^{\infty}(\tilde{\lambda}_{i} - d_{i}) &= \sigma,
\end{align}
and
\begin{equation}\label{midseqa}
\tilde{\lambda}_{i} > d_{i}  \qquad\text{for all but finitely many } i\in\N.
\end{equation}
\end{lem}

\begin{proof} \noindent\textbf{Case 1.} $\#|\{k\in\N : \delta_{k}<\sigma\}| <\infty$. Let 
\[M = \max\{k\in\N : \delta_{k}<\sigma\}.\]
Since $\sigma>0$ the set $\{k\in\N : \delta_{k} = 0\}$ is finite. Set
\[Z = \max\big(\{k\in\N : \delta_{k} = 0\}\cup \{0\}\big).\]
(Note, if $Z=0$ then we take $\lambda_{0} = d_{0} = \lambda_{1}+1$.)
Since $\delta_{Z+1}>0$, we have $\lambda_{Z+1}>d_{Z+1}$. Since $0=\delta_{Z} = \delta_{Z-1} + (\lambda_{Z} - d_{Z})\geq \lambda_{Z} - d_{Z}$ we have $d_{Z}\geq \lambda_{Z}$. Putting these together, we have
\[d_{Z}\geq \lambda_{Z}\geq \lambda_{Z+1}>d_{Z+1}.\]
Set
\[\alpha = d_{Z} - d_{Z+1}.\]
Set
\[\beta = \inf\{\delta_{k} : k\geq Z+1\}.\]
Note that $\beta>0$ since $\sigma > 0$.

Choose an integer $N> M$ large enough that both $N\alpha>\sigma$ and $N\beta>M\sigma$. For each $i\geq Z$ set
\begin{equation}\label{midseq5}
\tilde{\lambda}_{i} = \begin{cases} d_{i} & i\leq Z\text{ or }i\geq Z+N+1,\\ d_{i} + \frac{\sigma}{N} & i=Z+1,\ldots,Z+N. \end{cases}
\end{equation}
It is clear that $\tilde{\lambda}_{i}\geq \tilde{\lambda}_{i+1}$ for $i\neq Z$. By the choice of $N$ we have
\[\tilde{\lambda}_{Z} = d_{Z} = d_{Z+1}+\alpha>d_{Z+1} + \frac{\sigma}{N} = \tilde{\lambda}_{Z+1}.\]
Hence, $(\tilde{\lambda}_{i})$ is nonincreasing. By \eqref{midseq5} we quickly deduce \eqref{midseq-2} and \eqref{midseq-1}. To prove \eqref{midseq-3}, we compute
\[\tilde{\delta}_{k} = \begin{cases} \delta_{k} & k\leq Z,\\ \delta_{k} - (k-Z)\frac{\sigma}{N} & k=Z+1,\ldots,Z+N,\\ \delta_{k} - \sigma & k\geq Z+N+1.\end{cases}\]
By the choice of $N$, for $k\in\{Z+1,\ldots,Z+M\}$ we have
\[\tilde{\delta}_{k} = \delta_{k} - (k-Z)\frac{\sigma}{N}\geq \delta_{k} - M\frac{\sigma}{N}\geq \delta_{k} - \beta\geq 0,\]
and for $k\in\{Z+M+1,\ldots,Z+N\}$, since $k\geq M+1$ we have we have
\[\tilde{\delta}_{k} = \delta_{k} - (k-Z)\frac{\sigma}{N}\geq \delta_{k} - \sigma\geq 0.\]
Finally, for $k\geq Z+N+1$, since $k\geq M+1$ we have $\tilde{\delta}_{k} = \delta_{k} - \sigma\geq 0$. This implies that \eqref{midseq0} holds and completes the proof of Case 1.

\noindent\textbf{Case 2.} $\#|\{k\in\N : \delta_{k}<\sigma\}|=\infty$. Set $m_{0} = 0$ and for each $j\in\N$ set
\[m_{j} = \max\big\{m>m_{j-1} : \delta_{m} = \min\{\delta_{k} : k>m_{j-1}\}\big\}.\]
The definition of $m_{j}$ shows that $(\delta_{m_{j}})_{j=1}^{\infty}$ is a strictly increasing sequence with limit $\sigma$. Now, set $N_{0} = 0$, $N_{1}=m_{1}$, and for each $j\geq 2$ choose some $N_{j}\geq \max\{N_{j-1}+1,m_{j}\}$ such that
\[\frac{\delta_{m_{j}} - \delta_{m_{j-1}}}{N_{j} - N_{j-1}}<\frac{\delta_{m_{j-1}} - \delta_{m_{j-2}}}{N_{j-1} - N_{j-2}}\]

 For each $j\in\N$ set
\[\tilde{\lambda}_{i} = d_{i} + \frac{\delta_{m_{j}} - \delta_{m_{j-1}}}{N_{j} - N_{j-1}}\quad\text{for }i\in\{N_{j-1}+1,\ldots,N_{j}\}.\]
Note that for each $j\in\N$ we have
\[\tilde{\lambda}_{N_{j}} = d_{N_{j}} + \frac{\delta_{m_{j}} - \delta_{m_{j-1}}}{N_{j} - N_{j-1}} \leq d_{N_{j}+1} + \frac{\delta_{m_{j}} - \delta_{m_{j-1}}}{N_{j} - N_{j-1}}<d_{N_{j}+1} + \frac{\delta_{m_{j+1}} - \delta_{m_{j}}}{N_{j+1} - N_{j}} = \tilde{\lambda}_{N_{j}+1}.\]
From this and the assumption that $(d_{i})$ is nonincreasing we see that $(\tilde{\lambda}_{i})$ is nonincreasing. It is clear from the definition that $\tilde{\lambda}_{i}\geq d_{i}$ for all $i\in\N$. Finally, for $j\in\N$ and $k\in\{N_{j-1}+1,\ldots,N_{j}\}$ we have
\begin{align*}
\tilde{\delta}_{k} & = \delta_{k} - \sum_{n=1}^{j-1}\sum_{\ell=N_{n-1}+1}^{N_{n}}\frac{\delta_{m_{n}} - \delta_{m_{n-1}}}{N_{n} - N_{n-1}} - \sum_{\ell=N_{j-1}+1}^{k}\frac{\delta_{m_{j}} - \delta_{m_{j-1}}}{N_{j} - N_{j-1}} \\
 & = \delta_{k} - \sum_{n=1}^{j-1}(\delta_{m_{n}} - \delta_{m_{n-1}}) - \left(\frac{\delta_{m_{j}} - \delta_{m_{j-1}}}{N_{j} - N_{j-1}}\right)(k-N_{j-1})\\
 & = \delta_{k} - \delta_{m_{j-1}} - \left(\frac{\delta_{m_{j}} - \delta_{m_{j-1}}}{N_{j} - N_{j-1}}\right)(k-N_{j-1}).
\end{align*}
Since $N_{j-1}+1\leq k\leq N_{j}$, by taking $k=N_{j-1}$ and $k=N_{j}$ in the above equation yields
\[\delta_{k} - \delta_{m_{j-1}}\geq \tilde{\delta}_{k}\geq \delta_{k} - \delta_{m_{j}}.\]
Since $k>N_{j-1}\geq m_{j-1}$, and $\delta_{m_{j}} = \min\{\delta_{k} : k>m_{j-1}\}$, we see that $\delta_{k}\geq\delta_{m_{j}}$, and hence $\tilde{\delta}_{k}\geq 0$. 

For each $k\in\N$ let $j(k)\in\N$ be the unique number so that $k\in\{N_{j(k)-1}+1,\ldots,N_{j(k)}\}$, then we have
\[\delta_{k} - \delta_{m_{j(k)-1}}\geq \tilde{\delta}_{k}\geq \delta_{k} - \delta_{m_{j(k)}}.\]
Since $\sigma<\infty$, taking the $\liminf$ as $k\to\infty$ yields \eqref{midseq0}.

 From \eqref{midseq0} we see that there is a subsequence $(\tilde{\delta}_{k_{j}})_{j=1}^{\infty}$ that converges to $0$. Since
\[\sum_{i=1}^{k_{j}}(\tilde{\lambda}_{i} - d_{i}) + \sum_{i=1}^{k_{j}}(\lambda_{i} - \tilde{\lambda}_{i})= \sum_{i=1}^{k_{j}}(\lambda_{i} - d_{i}),\]
and both sums on the left converge (possibly to $\infty$) as $j\to\infty$, the sum on the right also converges, and hence
\begin{equation}\label{midseq1}\sum_{i=1}^{\infty}(\tilde{\lambda}_{i} - d_{i}) = \lim_{j\to\infty}\delta_{k_{j}} = \liminf_{j\to\infty}\delta_{k_{j}}\geq \liminf_{k\to\infty}\delta_{k} = \sigma.\end{equation}
On the other hand, there is a subsequence $(\delta_{m_{j}})_{j=1}^{\infty}$ so that $\delta_{m_{j}}\to\sigma$ as $j\to\infty$. Since
\[\sum_{i=1}^{m_{j}}(\lambda_{i} - d_{i}) - \sum_{i=1}^{m_{j}}(\tilde{\lambda}_{i} - d_{i}) = \sum_{i=1}^{m_{j}}(\lambda_{i} - \tilde{\lambda}_{i}),\]
and both sums on the left converge as $j\to\infty$, the sum on the right converges as $j\to\infty$, and hence
\begin{equation}\label{midseq2}\sigma - \sum_{i=1}^{\infty}(\tilde{\lambda}_{i} - d_{i}) = \lim_{j\to\infty}\tilde{\delta}_{m_{j}} = \liminf_{j\to\infty}\tilde{\delta}_{m_{j}}\geq \liminf_{k\to\infty}\tilde{\delta}_{k} = 0.\end{equation}
Putting together \eqref{midseq1} and \eqref{midseq2} yields \eqref{midseq-1} and completes the proof of Case 2.

This proves the existence of a nonincreasing sequence $(\tilde \lambda_i)$ fulfilling \eqref{midseq-2}--\eqref{midseq-1}.
Finally, we shall show how to modify it to satisfy the additional conclusion \eqref{midseqa}. Since $\sigma>0$, there exists $i_0\in\N$ such that $\tilde \lambda_{i_0}> d_{i_0}$ and $\tilde\lambda_{i_0} > \tilde \lambda_{i_0+1}$. Let 
\[
\eps= \min(\tilde \lambda_{i_0}- d_{i_0}, (\tilde\lambda_{i_0} - \tilde \lambda_{i_0+1})/2)>0.
\]
Define the sequence $(\check \lambda_i)$ by
\[
\check \lambda_i =
\begin{cases} \tilde \lambda_i &  i\le i_0-1,
\\
\tilde \lambda_i-\eps & i=i_0,
\\
\tilde \lambda_i+ 2^{i_0-i} \eps
& i \ge i_0+1.
\end{cases}
\]
By the choice of $\eps$, we have 
\[
\check\lambda_{i_0} = \tilde \lambda_{i_0} - \eps \ge 
\tilde \lambda_{i_0+1} + \eps >
\check \lambda_{i_0+1}.
\]
Hence, $( \check \lambda_i)$ is nonincreasing sequence. Moreover, $\check \lambda_i \ge d_i$ for all $i\in \N$, and $\check \lambda_i > \tilde \lambda_i \ge d_i$ for all $i\ge i_0+1$. By a direct calculation
\[
\check \delta_k: = \sum_{i=1}^k (\lambda_i - \check \lambda_i) = 
\begin{cases}
\tilde \delta_k & k \le i_0-1,
\\
\tilde \delta_k + 2^{i_0-i}\eps & k \ge i_0.
\end{cases}
\]
Thus, $\check \delta_k \ge 0$ for all $k\in\N$ and 
\[
\liminf_{k\to\infty} \check \delta_k =
\liminf_{k\to\infty} \tilde \delta_k = 0.
\]
Finally, by \eqref{midseq-1}
\[
\sum_{i=1}^{\infty}(\check{\lambda}_{i} - d_{i}) =
\sum_{i=1}^{\infty}(\tilde{\lambda}_{i} - d_{i}) + \sum_{i=1}^{\infty}(\check \lambda_i-\tilde{\lambda}_{i} ) = \sigma.\]
Hence, putting $(\check \lambda_i)$ in place of $(\tilde \lambda_i)$ yields all the conclusions \eqref{midseq-2}--\eqref{midseqa}.
\end{proof}

\begin{thm}\label{NoPosDiag} Let $(\lambda_{i})_{i\in\N}$ and $\boldsymbol d=(d_{i})_{i\in\N}$ be positive nonincreasing sequences in $c_{0}$ such that
\[\sum_{i=1}^{k}(\lambda_{i} - d_{i})\geq 0\quad\text{for all }k\in\N\]
and
\[\sigma := \liminf_{k\to\infty}\sum_{i=1}^{k}(\lambda_{i}-d_{i})\in(0,\infty).\]
Let $J\subset -\N$, and let $(\lambda_{i})_{i\in J}$ be a positive sequence such that
\[\sum_{i\in J}\lambda_{i} = \sigma.\]
If there is a self-adjoint operator $E$ with diagonal $\boldsymbol\lambda: = (\sgn(i)\lambda_{i})_{i\in\N\cup J}$, then $\boldsymbol d$ is also a diagonal of $E$.
\end{thm}

\begin{proof}

By Lemma \ref{midseq} there is a nonincreasing sequence $(\tilde{\lambda}_{i})_{i\in\N}$ such that $\tilde{\lambda}_{i}\geq d_{i}$ for all $i\in\N$,
\[\tilde{\delta}_{k}: = \sum_{i=1}^{k}(\lambda_{i} - \tilde{\lambda}_{i})\geq 0\quad\text{for all }k\in\N,\]
\begin{equation}\label{NoPosDiag1}\liminf_{k\to\infty}\tilde{\delta}_{k} = 0,\end{equation}
and
\begin{equation}\label{NoPosDiag4}\sum_{i=1}^{\infty}(\tilde{\lambda}_{i} - d_{i}) = \sigma = \sum_{i\in J}\lambda_{i}.\end{equation}

By assumption there is a self-adjoint operator $E$ with diagonal $(\sgn(i)\lambda_{i})_{i\in\N\cup J}$. Set $I_{1} = J_{1} = \N$ and $I_{2} = J_{2} = J$. If $E_{1}$ is any self-adjoint operator with diagonal $(\sgn(i)\lambda_{i})_{i\in I_{1}} = (\lambda_{i})_{i\in\N}$, then by Proposition \ref{posSH} the sequence $(\tilde{\lambda}_{i})_{i\in\N}$ is also a diagonal of $E_{1}$. Hence, by Lemma \ref{subdiag} $(\tilde{\lambda}_{i})_{i\in\N}\oplus(-\lambda_{i})_{i\in J}$ is also a diagonal of $E$.

\noindent\textbf{Case 1.} Assume  that $\#|J|=\infty$. Without loss of generality we may assume $J=-\N$. By Theorem \ref{infneg} the sequence $\boldsymbol d$ is a diagonal of $E$.

\noindent\textbf{Case 2.} Assume that $\#|J|=M$. Without loss of generality we may assume that $J=\{-1,\ldots,-M\}$. If $M=1$, then by Theorem \ref{1neg} the sequence $\boldsymbol d$ is also a diagonal of $E$. So, assume that the Theorem is true for $\#|J|=k$ for $k=1,2,\ldots,M-1$. From \eqref{NoPosDiag4} we see that there is a number $N\in\N$ so that
\[\lambda_{-M}<\sum_{i=1}^{N}(\tilde{\lambda}_{i} - d_{i}).\]
Let $N$ be the smallest such number. Define the sequence $(\mu_{i})_{i\in\N}$ as follows
\[\mu_{i} = \begin{cases}d_{i} & i=1,\ldots,N-1,\\ \tilde{\lambda}_{N} - \left(\lambda_{-M} - \sum_{i=1}^{N-1}(\tilde{\lambda}_{i} - d_{i})\right) & i=N,\\ \tilde{\lambda}_{i} & i\geq N+1.\end{cases}\]
By the choice of $N$ we have $\mu_{N}> d_{N}$, and hence $\mu_{i}\geq d_{i}$ for all $i\in\N$. By the minimality of $N$ we have $\tilde{\lambda}_{N}\geq \mu_{N}$, and hence $\tilde{\lambda}_{i}\geq \mu_{i}$ for all $i\in\N$. 

Set $I_{1} = \N\cup\{-M\}$, $J_{2} = \N$, and $I_{2} = J_{2} = \{-1,\ldots,-(M-1)\}$. Observe that
\[\sum_{i=1}^{k}(\tilde{\lambda}_{i} - \mu_{i}) = \begin{cases} \sum_{i=1}^{k}(\tilde{\lambda}_{i} - d_{i}) & k\leq N-1,\\ \lambda_{-M} & i\geq N\end{cases}\]
By Theorem \ref{1neg}, if $E_{1}$ is any operator with diagonal $(\sgn(i)\tilde{\lambda}_{i})_{i\in I_{1}} = (-\lambda_{-M},\lambda_{1},\lambda_{2},\lambda_{3},\ldots)$, then $(\mu_{i})_{i\in J_{2}} = (\mu_{i})_{i\in\N}$ is also a diagonal of $E_{1}$. From Lemma \ref{subdiag} we deduce that 
\[(\mu_{i})_{i\in J_{1}} \oplus (-\lambda_{i})_{i\in J_{2}} = (\mu_{i})_{i\in\N} \oplus (-\lambda_{i})_{i=-1}^{-(M-1)}\]
is a diagonal of $E$. Noting that 
\[\sum_{i=1}^{\infty}(\mu_{i} - d_{i}) = \sum_{i=1}^{\infty}(\tilde{\lambda}_{i} - d_{i}) - \lambda_{-M} = \sum_{i=-1}^{-(M-1)}\lambda_{i},\]
the inductive assumption that Theorem \ref{NoPosDiag} is true when $\#|J|=M-1$ now implies that $\boldsymbol d$ is a diagonal of $E$.

\end{proof}

\section{Equal excess result with positive target diagonal}

The goal of this section is to prove the following diagonal-to-diagonal theorem. Theorem \ref{posdiag} is a culmination of results in Sections \ref{S4} and \ref{S5}, where the target diagonal $\boldsymbol d$ is positive. 
Importantly, in contrast to Theorem \ref{NoPosDiag}, it includes the case where $\boldsymbol \lambda$ has only a finite number of positive terms.

\begin{thm}\label{posdiag} Let $\boldsymbol\lambda = (\lambda_{i})_{i\in\N}$ and $\boldsymbol d = (d_{i})_{i\in\N}$ be sequences in $c_{0}$. Assume $\lambda_{i}\neq 0$ and $d_{i}>0$ for all $i\in\N$. If there is a self-adjoint operator $E$ with diagonal $\boldsymbol\lambda$, 
\[\delta(\alpha,\boldsymbol\lambda, \boldsymbol d)\geq 0\quad\text{for all }\alpha\neq 0,\]
and
\begin{equation}\label{posdiag1}\liminf_{\alpha\searrow 0}\delta(\alpha,\boldsymbol\lambda, \boldsymbol d) = \liminf_{\alpha\nearrow 0}\delta(\alpha,\boldsymbol\lambda, \boldsymbol d)<\infty,\end{equation}
then $\boldsymbol d$ is also a diagonal of $E$.
\end{thm}

The following lemma is a counterpart of Lemma \ref{loss}. Like the proof of Lemma \ref{loss} in \cite{unbound}, this requires a careful application of an infinite sequence of convex moves, also known as $T$-transforms \cite{kw}, to the original basis $(f_{i})_{i\in\N}$. Unlike Lemma \ref{loss}, the limiting orthonormal sequence in Lemma \ref{noloss} has codimension $1$.

\begin{lem}\label{noloss} Let $(f_{i})_{i\in\N}$ be an orthonormal set, and let $(\alpha_{i})_{i\in\N}$ be a sequence in $[0,1]$. Set $\tilde{e}_{1}=f_{1}$ and inductively define for $i\in\N$,
\begin{equation}\label{noloss0}
e_{i} = \sqrt{\alpha_{i}}\,\tilde{e}_{i} + \sqrt{1-\alpha_{i}}\,f_{i+1}  \qquad\text{and}\qquad \tilde{e}_{i+1} = \sqrt{1-\alpha_{i}}\,\tilde{e}_{i} - \sqrt{\alpha_{i}}\,f_{i+1}.
\end{equation}
If
\begin{equation}\label{noloss1}
\lim_{n\to\infty}\prod_{i=n}^{\infty}(1-\alpha_{i})=1,\end{equation}
then there is a vector $e_{\infty}$ such that
\[\lim_{n\to\infty}\tilde{e}_{n} = e_{\infty}\]
$(e_{i})_{i\in\N\cup\{\infty\}}$ is an orthonormal basis for $\overline{\lspan}\{f_{i}: i\in\N\}$. Moreover, if $\alpha_{i}<1$ for all $i$, and $\sum_{i=1}^{\infty}\frac{\alpha_{i}}{1-\alpha_{i}}<\infty$, then \eqref{noloss1} holds.
\end{lem}

\begin{proof} First, we will show that $(\tilde{e}_{i})_{i=1}^{\infty}$ is a Cauchy sequence. We claim that for $i,n\in\N$ we have
\begin{equation}\label{noloss2}\|\tilde{e}_{i} - \tilde{e}_{i+n}\|^{2} = 2\left(1-\left(\prod_{j=i}^{i+n-1}(1-\alpha_{j})\right)^{\frac{1}{2}}\right).\end{equation}
First, note that
\begin{align*}
\|\tilde{e}_{i} - \tilde{e}_{i+1}\|^{2} & = \|\tilde{e}_{i} - \big(\sqrt{1-\alpha_{i}}\tilde{e}_{i} - \sqrt{\alpha_{i}}f_{i+1}\big)\|^{2} = (1-\sqrt{1-\alpha_{i}})^{2} + (\sqrt{\alpha_{i}})^{2}\\
 & = 2(1-\sqrt{1-\alpha_{i}}).
\end{align*}

Now, assume that \eqref{noloss2} is true for $n=k\in\N$, then we have
\[2\left(1-\left(\prod_{j=i}^{i+k-1}(1-\alpha_{j})\right)^{\frac{1}{2}}\right) = \|\tilde{e}_{i} - \tilde{e}_{i+k}\|^{2} = 2 - 2\langle \tilde{e}_{i},\tilde{e}_{i+k}\rangle.\]
In particular
\[\langle \tilde{e}_{i},\tilde{e}_{i+k}\rangle =\left(\prod_{j=i}^{i+k-1}(1-\alpha_{j})\right)^{\frac{1}{2}}.\]
Using this we have
\begin{align*}
\|\tilde{e}_{i} - \sqrt{1-\alpha_{i+k}}\tilde{e}_{i+k}\|^{2} & = 2-\alpha_{i+k}-2\sqrt{1-\alpha_{i+k}}\langle \tilde{e}_{i},\tilde{e}_{i+k}\rangle\\
 & = 2-\alpha_{i+k}-2\left(\prod_{j=i}^{i+k}(1-\alpha_{j})\right)^{\frac{1}{2}}
\end{align*}
To complete the induction we calculate
\begin{align*}
\|\tilde{e}_{i} - \tilde{e}_{i+k+1}\|^{2} & = \|\tilde{e}_{i} - \big(\sqrt{1-\alpha_{i+k}}\,\tilde{e}_{i+k} + \sqrt{\alpha_{i+k}}\,f_{i+k+1}\big)\|^{2}\\
 & = \|\tilde{e}_{i} - \sqrt{1-\alpha_{i+k}}\,\tilde{e}_{i+k}\|^{2} + \|\sqrt{\alpha_{i+k}}\,f_{i+k+1}\|^{2}\\
 & = \|\tilde{e}_{i} - \sqrt{1-\alpha_{i+k}}\,\tilde{e}_{i+k}\|^{2} + \alpha_{i+k} = 2\left(1-\left(\prod_{j=i}^{i+k}(1-\alpha_{j})\right)^{\frac{1}{2}}\right).
\end{align*}

Now, since $\alpha_{j}\in[0,1]$ for all $j\in\N$, we have
\[\prod_{j=i+1}^{i+n}(1-\alpha_{j}) \geq \prod_{j=i+1}^{\infty}(1-\alpha_{j}).\]
Hence \eqref{noloss1} and \eqref{noloss2} show that $(\tilde{e}_{i})_{i\in\N}$ is Cauchy, and thus the limit $e_{\infty}$ exists.

To complete the proof we must show that for each $j\in\N$ the vector $f_{j}$ is in $\overline{\lspan}\{e_{i}: i\in\N\cup\{\infty\}\}$. Note that for each $j\in\N$ we have
\[\lspan\{e_{1},\ldots,e_{j},\tilde{e}_{j+1}\} = \lspan\{f_{1},\ldots,f_{j+1}\}.\]
Hence, for $k\geq j-1$ we have
\[f_{j}\in\lspan\{e_{1},\ldots,e_{k},\tilde{e}_{k+1}\} .\]
Thus, for fixed $j,k\in\N$ with $k\geq j-1$ we have
\begin{align*}
|\langle f_{j},e_{\infty}\rangle|^{2} + \sum_{i=1}^{\infty}|\langle f_{j},e_{i}\rangle|^{2} & = |\langle f_{j},e_{\infty}\rangle|^{2} - |\langle f_{j},\tilde{e}_{k+1}\rangle|^{2} + \left(|\langle f_{j},\tilde{e}_{k+1}\rangle|^{2} + \sum_{i=1}^{k}|\langle f_{j},e_{i}\rangle|^{2}\right)\\
 & = |\langle f_{j},e_{\infty}\rangle|^{2} - |\langle f_{j},\tilde{e}_{k+1}\rangle|^{2} + 1 = \langle f_{j},e_{\infty}+\tilde{e}_{k}\rangle\langle f_{j},e_{\infty}-\tilde{e}_{k}\rangle +1.
\end{align*}
Since $\tilde{e}_{k}\to e_{\infty}$ as $k\to\infty$, letting $k\to\infty$ we have
\[|\langle f_{j},e_{\infty}\rangle|^{2} + \sum_{i=1}^{\infty}|\langle f_{j},e_{i}\rangle|^{2} = 1.\]

For the moreover part, note that
\[\frac{1}{\prod_{j=n}^{n+k}(1-\alpha_{j})} = \prod_{j=n}^{n+k}\left(1+\frac{\alpha_{j}}{1-\alpha_{j}}\right) \leq \exp\left(\sum_{j=n}^{n+k}\frac{\alpha_{j}}{1-\alpha_{j}} \right) .\]
Letting $k\to\infty$ we see that \eqref{noloss1} holds.
\end{proof}

The following lemma is a starting point for showing Theorem \ref{posdiag} in the case when $\boldsymbol \lambda$ has only one positive term. Notice that we cannot prescribe the target diagonal $\boldsymbol{\tilde{\lambda}}$ exactly, but Lemma \ref{infmove} gives some minimal amount of control on $\boldsymbol{\tilde{\lambda}}$, which will be sufficient for our purposes.

\begin{lem}\label{infmove} Let $\lambda_{1}>0$ and let $(\lambda_{i})_{i\in -\N}$ be a nonnegative sequence such that
\[s:=\sum_{i\in-\N}\lambda_{i} < \lambda_{1}.\]
If there is a self-adjoint operator $E$ with diagonal $\boldsymbol\lambda: = (\lambda_{1})\oplus(-\lambda_{i})_{i\in-\N}$, then for every $\eps>0$ there is a  nonincreasing positive sequence $\boldsymbol{\tilde{\lambda}} = (\tilde{\lambda}_{i})_{i\in\N}$ such that
\[\sum_{i=1}^{\infty}\tilde{\lambda}_{i} = \lambda_{1} - \sum_{i=1}^{\infty}\lambda_{-i} \quad\text{and}\quad \quad\tilde{\lambda}_{1}\in[\lambda_{1} - s - \eps,\lambda_{1}-s).\]
and $\boldsymbol{\tilde{\lambda}}$ is a diagonal of $E$.
\end{lem}

\begin{proof} \textbf{Case 1.} Assume $(\lambda_{i})_{i\in-\N}$ either has finite support, or is strictly positive. In either case, we may rearrange the sequence in order to assume
\[\lambda_{-1}\geq\lambda_{-2}\geq\lambda_{-3}\geq\cdots.\]
We may assume without loss of generality that
\[\eps<\frac{2}{3}\left(\lambda_{1} - s\right).\]
It follows that
\begin{equation}\label{infmove1}\tilde{\lambda}_{1}:=\lambda_{1} - s - \eps>\frac{\eps}{2}.\end{equation}
For $i\geq 2$ set
\begin{equation}\label{infmove3}\tilde{\lambda}_{i} = 2^{-i+1}\eps.\end{equation}
From \eqref{infmove1} we see that $(\tilde{\lambda}_{i})_{i\in\N}$ is nonincreasing.
For each $n\in\N$ define
\[\lambda_{1}^{(n)} = \lambda_{1} - \sum_{i=1}^{n-1}\left(\frac{\eps}{2^{i}} +\lambda_{-i}\right).\]
Observe that $\lambda_{1}^{(n)}\searrow\tilde{\lambda}_{1}$ as $n\to\infty$, and $\lambda_{1}^{(n)}> \tilde{\lambda}_{1}$ for all $n\in\N$.

The idea of the proof is  to construct  an orthonormal set $\{\tilde{e}_{k+1},e_{1},e_{2},\ldots,e_{k}\}$ for each $k\in\N$ with respect to which $E$ has diagonal $(\lambda_{1}^{(k+1)},\tilde{\lambda}_{2},\ldots,\tilde{\lambda}_{k+1})$. Using Lemma \ref{noloss} we will show that $(\tilde{e}_{k})_{i\in\N}$ converges to some vector $e_{\infty}$, and the set $(e_{i})_{i\in\N\cup\{\infty\}}$ is a basis.

For each $n\in\N$ set
\[\tilde{\beta}_{n} := \frac{\lambda_{1}^{(n)} - \tilde{\lambda}_{n+1}}{\lambda_{1}^{(n)}+\lambda_{-n}}.\]
Since
\begin{equation}\label{infmove2}-\lambda_{-n}\leq 0<\tilde{\lambda}_{n+1}<\tilde{\lambda}_{1}<\lambda_{1}^{(n+1)}<\lambda_{1}^{(n)}\quad\text{for all }n\in\N,\end{equation}
we see that $\tilde{\beta}_{n}>0$ for all $n\in\N$. 

Let $(f_{i})_{i\in\N}$ be an orthonormal basis such that
\[\langle Ef_{i},f_{i}\rangle = \begin{cases} \lambda_{1} & i=1,\\ -\lambda_{-i+1} & i\geq 2.\end{cases}\]
In order to apply Lemma \ref{offdiag} we set
\[\tilde{d}_{1} = -\lambda_{-1},\quad d_{1} = \tilde{\lambda}_{2},\quad d_{2} = \lambda_{1}^{(2)},\quad\text{and}\quad \tilde{d}_{2} = \lambda_{1}^{(1)}=\lambda_{1}.\]
Hence, we have
\[\tilde{\beta}_{1} = \frac{\tilde{d}_{2} - d_{1}}{\tilde{d}_{2} - \tilde{d}_{1}}.\]
Setting $g_{1} = f_{2}$ and $g_{2} = f_{1}$ we see that $\langle Eg_{i},g_{i}\rangle = \tilde{d}_{i}$ for $i=1,2$. Hence, by Lemma \ref{offdiag} there exists $\theta_{1}\in[0,2\pi)$ and $\beta_{1}\in[\tilde{\beta}_{1},1)$ so that the vectors
\[e_{1}=\sqrt{\beta_{1}}g_{1} + \sqrt{1-\beta_{1}}e^{i\theta_{1}}g_{2} = \sqrt{1-\beta_{1}}\,e^{i\theta_{1}} f_{1} + \sqrt{\beta_{1}}\,f_{2}\]
and
\[\tilde{e}_{2}=\sqrt{1-\beta_{1}}g_{1}-\sqrt{\beta_{1}}e^{i\theta_{1}}g_{2} = -\sqrt{\beta_{1}}\,e^{i\theta_{1}}f_{1} + \sqrt{1-\beta_{1}}\,f_{2}\]
form an orthonormal basis for $\lspan\{g_{1},g_{2}\} = \lspan\{f_{1},f_{2}\}$, $\langle E\tilde{e}_{2},\tilde{e}_{2}\rangle = \lambda_{1}^{(2)}$ and $\langle Ee_{1},e_{1}\rangle = \tilde{\lambda}_{2}$.

Next, we will show that for each $k\in\N$ we have an orthonormal basis $\{e_{1},e_{2},\ldots,e_{k},\tilde{e}_{k+1}\}$ for $\lspan\{f_{1},\ldots,f_{k+1}\}$ such that
\[\langle E\tilde{e}_{k+1},\tilde{e}_{k+1}\rangle = \lambda_{1}^{(k+1)}\quad\text{and}\quad \langle Ee_{j},e_{j}\rangle = \tilde{\lambda}_{j+1}\quad\text{for }j=1,2,\ldots,k.\]
Assume we have such an orthonormal basis for some $k\in\N$. As in the base case, in order to apply Lemma \ref{offdiag} set
\[\tilde{d}_{1} = -\lambda_{-(k+1)},\quad d_{1} = \tilde{\lambda}_{k+2} ,\quad d_{2} = \lambda_{1}^{(k+2)},\quad\text{and}\quad \tilde{d}_{2} = \lambda_{1}^{(k+1)}.\]
By assumption, if we set $g_{1} = f_{k+2}$ and $g_{2} = \tilde{e}_{k+1}$, then we have $\langle Eg_{i},g_{i}\rangle = \tilde{d}_{i}$ for $i=1,2$. By Lemma \ref{offdiag} there exists $\theta_{k+1}\in[0,2\pi)$ and $\beta_{k+1}\in[\tilde{\beta}_{k+1},1)$ so that the vectors
\[e_{k+1}=\sqrt{\beta_{k+1}}g_{1} + \sqrt{1-\beta_{k+1}}e^{i\theta_{k+1}}g_{2} = \sqrt{1-\beta_{k+1}}\,e^{i\theta_{k+1}} \tilde{e}_{k+1} + \sqrt{\beta_{k+1}}\,f_{k+2}\]
and
\[\tilde{e}_{k+2}=\sqrt{1-\beta_{k+1}}g_{1}-\sqrt{\beta_{k+1}}e^{i\theta_{k+1}}g_{2} = -\sqrt{\beta_{k+1}}\,e^{i\theta_{k+1}}\tilde{e}_{k+1} + \sqrt{1-\beta_{k+1}}\,f_{k+2}\]
form an orthonormal basis for $\lspan\{g_{1},g_{2}\} = \lspan\{\tilde{e}_{k+1},f_{k+2}\}$ and 
\begin{equation}\label{infmove5}
\langle E\tilde{e}_{k+2},\tilde{e}_{k+2}\rangle = \lambda_{1}^{(k+2)}
\qquad\text{and }\qquad
\langle Ee_{k+1},e_{k+1}\rangle = \tilde{\lambda}_{k+2}.
\end{equation}

The sequences $(-e^{-i\theta_{n}}\tilde{e}_{n+1})_{n\in\N}$ and $(e^{-i\theta_{n}}e_{n})_{n\in\N}$ are given by Lemma \ref{noloss} applied to $(e^{i\theta_{n}}f_{n})_{n\in\N}$ and $(\alpha_{n})_{n\in\N}$ where $\alpha_{n} = 1-\beta_{n}$ for each $n\in\N$. Since $1-\alpha_{n} = \beta_{n}\geq \tilde{\beta}_{n}$, we have
\[\frac{\alpha_{n}}{1-\alpha_{n}} \leq \frac{1-\tilde{\beta}_{n}}{\tilde{\beta}_{n}} = \frac{\lambda_{-n} + \tilde{\lambda}_{n+1}}{\lambda_{1}^{(n)} - \tilde{\lambda}_{n+1}}\leq \frac{\lambda_{-n} + \tilde{\lambda}_{n+1}}{\lambda_{1} - \tilde{\lambda}_{1}}.\]
Since $(\lambda_{n})_{n\in-\N}$ and $(\tilde{\lambda}_{n+1})_{n\in\N}$ are both summable, we see that $\sum_{i=1}^{\infty}\frac{\alpha_{n}}{1-\alpha_{n}}<\infty$. By Lemma \ref{noloss} the sequence $(-e^{-i\theta_{n}}\tilde{e}_{n+1})_{n\in\N}$ has a limit $e_{\infty}$ and $(e_{i})_{i\in\N\cup\{\infty\}}$ is an orthonormal basis. Hence, by \eqref{infmove5} the sequence $(\tilde{\lambda}_{i})_{i\in\N}$ is a diagonal of $E$.\bigskip

\noindent\textbf{Case 2.} Assume $Z=\{i\in-\N:\lambda_{i}=0\}$ is infinite. If $-\N\setminus Z$ is finite, then we are in Case 1. Thus, we may assume $-\N\setminus Z$ is infinite.

Let $(\mu_{i})_{i\in-\N}$ be the sequence consisting of all of the positive terms of $(\lambda_{i})_{i\in-\N}$ arranged such that $\mu_{-1}\geq\mu_{-2}\geq \cdots$. Let $E_{1}$ be a self-adjoint operator with diagonal $(\lambda_{1})\oplus(-\mu_{i})_{i\in-\N}$. By Case 1 there is a positive, nonincreasing sequence $(\tilde{\mu}_{i})_{i\in\N}$ such that
\[\sum_{i=1}^{\infty}\tilde{\mu}_{i} = \lambda_{1} - \sum_{i=1}^{\infty}\mu_{-i} = \lambda_{1} - s\quad\text{and}\quad \tilde{\mu}_{1}\in[\lambda_{1} - s - \frac{\eps}{2},\lambda_{1} - s)\]
and $(\tilde{\mu}_{i})_{i\in\N}$ is a diagonal of $E_{1}$. By choosing a possibly smaller $\eps$ we may assume $\tilde{\mu}_{1}>\tilde{\mu}_{2}+\eps$, see \eqref{infmove3}. By Lemma \ref{subdiag} the sequence $(\tilde{\mu}_{i})_{i\in\N}\oplus(0)_{i\in Z}$ is a diagonal of $E$.

Let $E_{2}$ be a self-adjoint operator with diagonal $(\tilde{\mu}_{1})\oplus(\zeta_i)_{i\in Z}$, where $\zeta_i=0$. By Case 1 there is a positive nonincreasing sequence $(\tilde{\zeta}_{i})_{i\in\N}$ so that
\[\sum_{i=1}^{\infty}\tilde{\zeta}_{i} = \tilde{\mu}_{1}\quad\text{and}\quad \tilde{\zeta}_{1}\in[\tilde{\mu}_{1}-\frac{\eps}{2},\tilde{\mu}_{1})\]
and $(\tilde{\zeta})_{i\in\N}$ is a diagonal of $E_{2}$. To complete this case we note that by Lemma \ref{subdiag} the strictly positive sequence $(\tilde{\zeta}_{i})_{i\in\N}\oplus(\tilde{\mu}_{i})_{i=2}^{\infty}$ is a diagonal of $E$. Let $(\tilde{\lambda}_{i})_{i\in\N}$ be the decreasing rearrangement of $(\tilde{\zeta}_{i})_{i\in\N}\oplus(\tilde{\mu}_{i})_{i=2}^{\infty}$. By construction we have $\tilde{\lambda}_{1} = \tilde{\zeta}_{1}$, and hence
\[\tilde{\lambda}_{1}\in[\lambda_{1} - s - \eps,\lambda_{1}-s).\]
This completes the proof of Case 2.\bigskip

\noindent\textbf{Case 3.} Assume $Z$ is finite. Fix $i_{0}\in(-\N )\setminus Z$. Let $E_{1}$ be any self-adjoint operator with diagonal $(-\lambda_{i})_{i\in\{i_{0}\}}\oplus(\zeta_{i})_{i\in Z}$ where $\zeta_{i} = 0$. Set $\tilde{\zeta}_{i} = \lambda_{i_{0}}/(\#|Z|+1)$. Then $(-\tilde{\zeta}_{i})_{i\in\{i_{0}\}\cup Z}\preccurlyeq(-\lambda_{i})_{i\in\{i_{0}\}}\oplus(\zeta_{i})_{i\in Z}$. By the Schur-Horn theorem, the strictly negative sequence $(-\tilde{\zeta}_{i})_{i\in\{i_{0}\}\cup Z}$ is a diagonal of $E_{1}$. Set $\tilde{\zeta}_{i} = \lambda_{i}$ for $i\in\N\setminus(\{i_{0}\}\cup Z)$. By Lemma \ref{subdiag} the sequence $(\lambda_{1})\oplus(-\tilde{\zeta}_{i})_{i\in-\N}$ is a diagonal of $E$. Finally, apply Case 1.
\end{proof}

Theorem \ref{posdiagv2} shows the special case of Theorem \ref{posdiag} in the case when $\boldsymbol \lambda$ has a finite number of positive terms.

\begin{thm}\label{posdiagv2} Let $(\lambda_{i})_{i=1}^{N}$ and $\boldsymbol d=(d_{i})_{i\in\N}\in\ell^1$ be positive nonincreasing sequences such that
\[\sum_{i=1}^{k}(\lambda_{i} - d_{i})\geq 0\quad\text{for all }k\leq N\]
and
\[\sigma := \sum_{i=1}^{N}\lambda_{i}-\sum_{i=1}^{\infty}d_{i}>0.\]
Let $(\lambda_{i})_{i\in -\N}$ be a positive sequence such that
\[\sum_{i\in-\N}\lambda_{i} = \sigma.\]
If there is a self-adjoint operator $E$ with diagonal $\boldsymbol\lambda: = (\lambda_{i})_{i=1}^{N}\oplus(-\lambda_{i})_{i\in-\N}$, then $\boldsymbol d$ is also a diagonal of $E$.

\end{thm}

\begin{proof} Since
\[\sum_{i=1}^{N}(\lambda_{i} - d_{i})>\sigma>0,\]
there is some $k\leq N$ such that $\lambda_{k}>d_{k}$. Set $K=\max\{k\leq N : \lambda_{k}>d_{k}\}$. Since it could be the case that $K=N$ we will set $\lambda_{N+1} = \frac{1}{2}\lambda_{N}$. Now, if $K<N$, it must be the case that $\lambda_{K+1}<\lambda_{K}$. Otherwise $\lambda_{K+1}=\lambda_{K}>d_{K}\geq d_{K+1}$, contradicting the maximality of $K$. Hence, in any case we have $\lambda_{K}>\lambda_{K+1}$.

Fix $\eps>0$ such that
\[\eps<\min\{\lambda_{K}-\lambda_{K+1},\lambda_{N},\sigma\}.\]
Choose $M\in\N$ such that
\[s:=\sum_{i=-(M+1)}^{-\infty}\lambda_{i}<\frac{\eps}{2}.\]

Set $I_{1} = \{K,-(M+1),-(M+2),\ldots\}$ and $J_{1} = \{K,N+1,N+2,\ldots\}$. Assume $E_{1}$ is a self-adjoint operator with diagonal $(\lambda_{i})_{i\in I_{1}}$. Define the sequence $\boldsymbol\mu = (\mu_{i})_{i\in\{1\}\cup-\N}$ by
\[\mu_{i} = \begin{cases} \lambda_{K} & i=1,\\ \lambda_{-M+i} & i\in-\N.\end{cases}\]
It is clear that $\boldsymbol\mu$ is a diagonal of $E_{1}$. Note that
\[\sum_{i\in-\N}\mu_{i} = \sum_{i=-(M+1)}^{-\infty}\lambda_{i}<\eps<\lambda_{K} - \lambda_{K-1}<\lambda_{K} = \mu_{1}.\]
Lemma \ref{infmove} there is a positive nonincreasing sequence $\boldsymbol{\tilde{\mu}} = (\tilde{\mu}_{i})_{i\in\N}$ such that
\[\sum_{i=1}^{\infty}\tilde{\mu}_{i} = \mu_{1} - \sum_{i=-1}^{-\infty}\mu_{i}\quad\text{and}\quad \tilde\mu_{1}\in(\mu_{1} -s -\tfrac{\eps}{2},\mu_{1}-s],\]
and $\boldsymbol{\tilde{\mu}}$ is a diagonal of $E_{1}$. Finally, define the sequence $\boldsymbol{\tilde{\lambda}} = (\tilde{\lambda}_{i})_{i\in J_{1}}$ by
\[\tilde{\lambda}_{i} = \begin{cases} \tilde{\mu}_{1} & i=K,\\ \tilde{\mu}_{i-N+1} & i\geq N+1.\end{cases}\]
The sequence $\boldsymbol{\tilde{\lambda}}$ is just a reindexing of $\boldsymbol{\tilde{\mu}}$ and thus $\boldsymbol{\tilde{\lambda}}$ is a diagonal of $E_{1}$. Hence, given any self-adjoint operator $E_{1}$ with diagonal $(\lambda_{i})_{i\in I_{1}}$, there is a sequence $(\tilde{\lambda}_{i})_{i\in J_{1}}$ such that
\[\tilde{\lambda}_{K} + \sum_{i=N+1}^{\infty}\tilde{\lambda}_{i} = \lambda_{K} - \sum_{i=-(M+1)}^{-\infty}\lambda_{i}\quad\text{and}\quad \tilde{\lambda}_{K}\in(\lambda_{K} - s - \tfrac{\eps}{2},\lambda_{K}-s]\]
and $(\tilde{\lambda}_{i})_{i\in J_{1}}$ is a diagonal of $E_{1}$.

Set $I_{2} = J_{2} = \{1,\ldots,K-1,K+1,\ldots,N\}\cup\{-1,-2,\ldots,-M\}$, and set $\tilde{\lambda}_{i} = \lambda_{i}$ for all $i\in I_{2}$. By Lemma \ref{subdiag} the sequence $(\tilde{\lambda}_{i})_{i\in\N\cup\{-1,\ldots,-M\}}$ is a diagonal of $E$. 
If $K<N$, then 
\[\tilde{\lambda}_{K}>\lambda_{K} - s - \frac{\eps}{2}>\lambda_{K} - \eps > \lambda_{K+1} = \tilde{\lambda}_{K+1}.\]
If $K\geq 2$, then
\[\tilde{\lambda}_{K-1} = \lambda_{K-1}\geq \lambda_{K}\geq \tilde{\lambda}_{K}.\]
From these two inequalities we see that $(\tilde{\lambda}_{i})_{i=1}^{N}$ is nonincreasing. By Lemma \ref{infmove}, the sequence $(\tilde{\lambda}_{i})_{i=N+1}^{\infty}$ is nonincreasing. For $i\geq N+1$ we have
\[\tilde{\lambda}_{N+1} < \sum_{i=N+1}^{\infty}\tilde{\lambda}_{i} = \lambda_{K} - \tilde{\lambda}_{K} - \sum_{i=-(M+1)}^{-\infty}\lambda_{i} < \eps - \sum_{i=-(M+1)}^{-\infty}\lambda_{i} < \eps<\lambda_{N} = \tilde{\lambda}_{N}.\]
Putting this all together we see that $(\tilde{\lambda}_{i})_{i\in\N}$ is in nonincreasing order. For $k\leq K-1$ we have
\[\sum_{i=1}^{k}(\tilde{\lambda}_{i} - d_{i}) = \sum_{i=1}^{k}(\lambda_{i} - d_{i}) \geq 0.\]
For $k$ such that $K\leq k\leq N$ we have
\begin{align*}
\sum_{i=1}^{k}(\tilde{\lambda}_{i} - d_{i}) & = \sum_{i=1}^{K-1}(\lambda_{i} - d_{i}) + (\tilde{\lambda}_{K} - d_{K}) + \sum_{i=K+1}^{k}(\lambda_{i} - d_{i}) \geq \sum_{i=1}^{k}(\lambda_{i} - d_{i}) - \eps\\
 & \geq \sum_{i=1}^{N}(\lambda_{i} - d_{i}) - \eps \geq \sigma-\eps>0
\end{align*}
For $k>N$ we have
\begin{align*}
\sum_{i=1}^{k}(\tilde{\lambda}_{i} - d_{i}) & \geq \sum_{i=1}^{N}(\lambda_{i} - d_{i}) - \eps + \sum_{i=N+1}^{k}\tilde{\lambda}_{i} - \sum_{i=N+1}^{k}d_{i}\\
 & \geq \sum_{i=1}^{N}(\lambda_{i} - d_{i}) - \eps - \sum_{i=N+1}^{\infty}d_{i} = \sigma - \eps > 0.
\end{align*}
To complete the proof we note that
\[\sum_{i=1}^{\infty}(\tilde{\lambda}_{i} - d_{i}) = \sum_{i=1}^{N}(\lambda_{i} - d_{i}) - \sum_{i=N+1}^{\infty}d_{i} - \sum_{i=-(M+1)}^{-\infty}\lambda_{i} = \sigma - \sum_{i=-(M+1)}^{-\infty}\lambda_{i} = \sum_{i=-1}^{-M}\lambda_{i}.\]
Finally, Theorem \ref{NoPosDiag} implies that $\boldsymbol d$ is a diagonal of $E$.
\end{proof}

\begin{proof}[Proof of Theorem \ref{posdiag}] Since $\boldsymbol d$ is a positive sequence, without loss of generality we may assume $\boldsymbol d$ is nonincreasing. Moreover, for $\alpha<0$ we have
\[\delta(\alpha,\boldsymbol\lambda,\boldsymbol d) = \sum_{\lambda_{i}\leq\alpha}(\alpha-\lambda_{i}).\]
By Proposition \ref{dclem} we have
\[\sum_{\lambda_{i}<0}\lambda_{i} = \lim_{\alpha\nearrow 0}\delta(\alpha,\boldsymbol\lambda,\boldsymbol d)<\infty.\]

Next, note that for any $\alpha>0$ we have
\[\delta(\alpha,\boldsymbol\lambda,\boldsymbol d) = \sum_{\lambda_{i}\geq\alpha}(\lambda_{i}-\alpha)-\sum_{d_{i}\geq\alpha}(d_{i}-\alpha) \geq 0.\]
From the positivity of $\boldsymbol d$ we see that $\{i\in\N : \lambda_{i}>0\}\neq\varnothing$. \bigskip

\noindent\textbf{Case 1.} Assume $\#|\{i : \lambda_{i}>0\}| = \infty$. If $\#|\{i : \lambda_{i}<0\}| = N\in\N\cup\{\infty\}$, then we set $J=\{i\in-\N : |i|<N+1\}$. There is a bijection $\pi:\N\cup J \to\N$ so that $(\lambda_{\pi(i)})_{i\in\N}$ is a nonincreasing positive sequence and $(\lambda_{\pi(i)})_{i\in J}$ is a negative sequence. For clarity, we define the sequence $\boldsymbol\mu = (\mu_{i})_{i\in\N\cup J}$ by setting $\mu_{i} = \lambda_{\pi(i)}$ for all $i\in\N\cup J$. By Proposition \ref{LR} we have
\[\sum_{i=1}^{k}(\mu_{i} - d_{i})\geq 0\quad\text{for all }k\in\N\]
and
\[\liminf_{\alpha\searrow 0}\delta(\alpha,\boldsymbol\mu,\boldsymbol d) = \liminf_{k\to\infty}\sum_{i=1}^{k}(\mu_{i} - d_{i}).\]

The sequence $\boldsymbol\mu$ contains exactly the same terms as $\boldsymbol\lambda$, and hence $\boldsymbol\mu$ is a diagonal of $E$, and
\[\liminf_{k\to\infty}\sum_{i=1}^{k}(\mu_{i} - d_{i}) = \liminf_{\alpha\searrow 0}\delta(\alpha,\boldsymbol\la,\boldsymbol d) = \lim_{\alpha\nearrow 0}\delta(\alpha,\boldsymbol\lambda,\boldsymbol d) = \sum_{\lambda_{i}<0}\lambda_{i} = \sum_{i\in J}\mu_{i}.\]
By Theorem \ref{NoPosDiag} the sequence $\boldsymbol d$ is a diagonal of $E$. This completes the first case.\bigskip

\noindent\textbf{Case 2.} Assume $\#|\{i : \lambda_{i}>0\}| = N\in\N$. This implies $\#|\{i : \lambda_{i}<0\}|=\infty$. There is a bijection $\pi:-\N\cup\{1,\ldots,N\}\to \N$ so that $(\lambda_{\pi(i)})_{i=1}^{N}$ is a positive, nonincreasing sequence, and $(\lambda_{\pi(i)})_{i\in-\N}$ is a negative sequence. As in the previous case, define the sequence $\boldsymbol\mu = (\mu_{i})_{i\in -\N \cup\{1,\ldots,N\}}$ by setting $\mu_{i} = \lambda_{\pi(i)}$. Define the sequence $\boldsymbol\gamma = (\gamma_{i})_{i\in\N}$ by
\[\gamma_{i} = \begin{cases} \mu_{i} = \lambda_{\pi(i)} & i\leq N,\\ 0 & i\geq N+1.\end{cases}\]
For $\alpha>0$ we have
\[\delta(\alpha,\boldsymbol\gamma,\boldsymbol d) = \delta(\alpha,\boldsymbol\mu,\boldsymbol d)\geq 0.\]
Proposition \ref{LR} implies
\[\sum_{i=1}^{k}(\mu_{i} - d_{i})\geq 0\quad\text{for all }k\leq N.\]
For $\alpha\in(0,\lambda_{N})$ we see that
\[\delta(\alpha,\boldsymbol\lambda,\boldsymbol d) = \sum_{i=1}^{N}\mu_{i} - \sum_{d_{i}\geq\alpha}(d_{i} - \alpha).\]
From Proposition \ref{dclem} we see that
\[\liminf_{\alpha\searrow0}\delta(\alpha,\boldsymbol\lambda,\boldsymbol d) = \sum_{i=1}^{N}\mu_{i} - \sum_{i=1}^{\infty}d_{i}.\]
Using \eqref{posdiag1} we deduce
\[\sum_{i=1}^{N}\mu_{i} - \sum_{i=1}^{\infty}d_{i} = \sum_{\lambda_{i}<0}\lambda_{i} = \sum_{i\in-\N}\mu_{i}.\]
Since $\boldsymbol\mu$ is a diagonal of $E$, by Theorem \ref{posdiagv2} the sequence $\boldsymbol d$ is also a diagonal of $E$. This completes the proof of the second case and the theorem.\end{proof}

\section{Equal excess diagonal-to-diagonal result}\label{S7}

The goal of this section is to show the main diagonal-to-diagonal result for equal positive and negative excesses. To achieve this we shall extend Theorem \ref{posdiag} by relaxing the inessential assumptions about the initial diagonal $\boldsymbol \lambda$ and the target diagonal $\boldsymbol d$.

\begin{thm}\label{step5} Let $\boldsymbol\lambda = (\lambda_{i})_{i\in\N}$ and $\boldsymbol d = (d_{i})_{i\in\N}$ be sequences in $c_{0}$.  If there is a self-adjoint operator $E$ with diagonal $\boldsymbol\lambda$, 
\begin{equation}\label{pos5.0}\delta(\alpha,\boldsymbol\lambda, \boldsymbol d)\geq 0\quad\text{for all }\alpha\neq 0,\end{equation}
and
\begin{equation}\label{pos5}
\liminf_{\alpha\searrow 0}\delta(\alpha,\boldsymbol\lambda, \boldsymbol d) = \liminf_{\alpha\nearrow 0}\delta(\alpha,\boldsymbol\lambda, \boldsymbol d)\in (0,\infty),\end{equation}
then $\boldsymbol d$ is also a diagonal of $E$.
\end{thm}

The following result enables us to deal with zero terms in the initial diagonal $\boldsymbol \lambda$ by reducing the excess on both negative and positive sides by the same amount. Note that Lemma \ref{red} does not require excesses to be the same and hence it can be applied in a later section dealing with a non-equal excess.

\begin{lem}\label{red} Let $\boldsymbol\lambda = (\lambda_{i})_{i\in\N}$ and $\boldsymbol d= (d_{i})_{i\in\N}$ be sequences in $c_{0}$  such that
\begin{equation}\label{red0}\delta(\alpha,\boldsymbol\lambda,\boldsymbol d)\geq0 \quad\text{for all }\alpha\neq 0,\end{equation}
\begin{equation}\label{red1}\sigma_+:=\liminf_{\alpha\searrow 0}\delta(\alpha,\boldsymbol\lambda, \boldsymbol d)>0\quad\text{and}\quad\sigma_-:= \liminf_{\alpha\nearrow 0}\delta(\alpha,\boldsymbol\lambda, \boldsymbol d)>0.\end{equation}
Let $E$ be a self-adjoint operator with diagonal $\boldsymbol\lambda$.

There exists a sequence $\tilde{\boldsymbol\lambda} = (\tilde{\lambda}_{i})_{i\in\N}$ in $c_0$ and a number $s\ge 0$ such that $\tilde{\lambda}_{i}\neq 0$ for all $i\in\N$, both $\tilde{\boldsymbol\lambda}_{+}$ and $\tilde{\boldsymbol\lambda}_{-}$ have infinite support, $\tilde{\boldsymbol\lambda}$ is a diagonal of $E$,
\begin{equation}\label{red2}
\delta(\alpha,\tilde{\boldsymbol\lambda},\boldsymbol d)\geq0 \quad\text{for all }\alpha\neq 0,
\end{equation}
\begin{equation}\label{red3}
\liminf_{\alpha\searrow 0}\delta(\alpha,\tilde{\boldsymbol\lambda}, \boldsymbol d) = \sigma_{+} - s>0\quad\text{and}\quad
\quad\liminf_{\alpha\nearrow 0}\delta(\alpha,\tilde{\boldsymbol\lambda}, \boldsymbol d) = \sigma_{-} - s>0.
\end{equation}
\end{lem}

Note that the positive excess $\sigma_+$ and negative excess $\sigma_-$ defined by \eqref{red1} could be equal to $\infty$. If this is the case, say $\sigma_+=\infty$, then we use the convention that $\sigma_+-s=\infty$ in \eqref{red3}. 

\begin{proof} First, we will prove the lemma under the additional assumption that the set $\{i\in\N:\lambda_{i}\leq 0\}$ is infinite albeit without the conclusion that the support of $\tilde{\boldsymbol\lambda}_{-}$ is infinite. 

Set $N=\{i\in\N : \lambda_{i}<0\}$, $Z:=\{i\in\N : \lambda_{i} = 0\}$, and $\sigma = \min(\sigma_{+},\sigma_{-},1)$. Since $\sigma_{+}>0$ we see that $\boldsymbol\lambda$ has at least one strictly positive term. We may choose $i_{0}\in\N$ such that $\lambda_{i_{0}}>0$ and 
\begin{equation}\label{red5}
\delta(\alpha,\boldsymbol\lambda,\boldsymbol d)\geq \sigma/2
\qquad\text{for all }\alpha\in(0,\lambda_{i_{0}}].
\end{equation}

Let $(i_{k})_{k\in\N}$ be a sequence of distinct indices in $N\cup Z$ so that $Z\subset\{i_{k} : k\in\N\}$ and
\[s:=\sum_{k=1}^{\infty}|\lambda_{i_{k}}|<\eps:=\min(\lambda_{i_{0}}/4,\sigma/4).\]
If $N$ is infinite, then we additionally assume that $N \setminus \{ i_k: k\in \N\}$ is infinite. This is to guarantee that the support of $\tilde {\boldsymbol\lambda}_-$ is infinite if the support of $\boldsymbol\lambda_-$ is infinite.

Fix $\alpha_{0}<0$ such that 
\[\delta(\alpha,\boldsymbol\lambda,\boldsymbol d)\geq \sigma/2\quad\text{for all }\alpha\in[\alpha_{0},0).\]

If $E_1$ is any self-adjoint operator with diagonal $(\lambda_{i_k})_{k\in\N \cup\{0\}}$, then by Lemma \ref{infmove} with \(\eps=\min(\lambda_{i_{0}}/4,\sigma/4)\), as above, there exists a  positive sequence $ (\tilde \lambda_{i_k})_{k\in\N \cup\{0\}}$ such that 
\begin{equation}\label{red6}
\sum_{k=0}^{\infty} \tilde \lambda_{i_{k}} = \lambda_{i_0} - \sum_{k=1}^\infty |\lambda_{i_k}| = \lambda_{i_0}-s \qquad\text{and}\qquad
\tilde \lambda_{i_0} \in [\lambda_{i_0}-s-\eps,\lambda_{i_0}-s)
\end{equation}
and $(\tilde \lambda_{i_k})_{k\in \N \cup\{0\}}$ is a diagonal of $E_1$. Set $\tilde \lambda_i =\lambda_i$ for $i \in \N \setminus \{i_0,i_1,\ldots\}$. By Lemma \ref{subdiag}, the sequence $\tilde {\boldsymbol \lambda}=(\tilde \lambda_i)_{i\in \N}$ is a diagonal of $E$.

Note that
\[\sum_{k=1}^{\infty}\tilde{\lambda}_{i_{k}} \leq \lambda_{{i_0}} - \tilde{\lambda}_{{i_0}}\leq s+\eps<\frac{\lambda_{{i_0}}}{2},\]
and hence $\tilde{\lambda}_{i_{k}}<\lambda_{{i_0}}/2< \tilde \lambda_{i_0}$ for all $k\in\N$. Additionally, $\tilde{\lambda}_{{i_0}}<\lambda_{{i_0}}$, and hence
\[\delta(\alpha,\tilde{\boldsymbol{\lambda}},\boldsymbol{d}) = \delta(\alpha,\boldsymbol{\lambda},\boldsymbol{d})\geq 0\quad\text{for }\alpha\geq \lambda_{{i_0}}.\]
By \eqref{red5} and \eqref{red6}, for $\alpha\in[\tilde{\lambda}_{{i_0}},\lambda_{{i_0}})$ we have
\[\delta(\alpha,\tilde{\boldsymbol{\lambda}},\boldsymbol{d}) = \delta(\alpha,\boldsymbol{\lambda},\boldsymbol{d}) - (\lambda_{{i_0}} - \alpha) \geq \sigma/2 - (\lambda_{{i_0}} - \tilde{\lambda}_{{i_0}})> \sigma/2 - s - \eps>0.\]
For $\alpha\in(0,\tilde{\lambda}_{{i_0}}]$ we have
\begin{align*}
\delta(\alpha,\tilde{\boldsymbol{\lambda}},\boldsymbol{d}) & = \sum_{\tilde{\lambda}_{i}\geq \alpha}(\tilde{\lambda}_{i} - \alpha) - \sum_{d_{i}\geq \alpha}(d_{i} - \alpha)\\
 & = \sum_{\lambda_{i}\geq \alpha}(\lambda_{i} - \alpha) - \sum_{d_{i}\geq \alpha}(d_{i} - \alpha) + \sum_{k\in\N\cup\{0\},\  \tilde{\lambda}_{i_{k}}\geq \alpha}(\tilde{\lambda}_{i_{k}} - \alpha) -(\lambda_{{i_0}} - \alpha)\\
 & = \delta(\alpha,\boldsymbol{\lambda},\boldsymbol{d}) +\sum_{k\in \N,\  \tilde{\lambda}_{i_{k}}\geq \alpha}(\tilde{\lambda}_{i_{k}}-\alpha) - (\lambda_{{i_0}} - \tilde{\lambda}_{{i_0}})>0.
\end{align*}
From this and \eqref{red6} we deduce that \[\liminf_{\alpha\searrow 0}\delta(\alpha,\tilde{\boldsymbol{\lambda}},\boldsymbol{d}) = \sigma_{+} + \sum_{k=1}^{\infty}\tilde{\lambda}_{i_{k}} - (\lambda_{{i_0}} - \tilde{\lambda}_{{i_0}})=\sigma_{+}-s.\]

For $\alpha\leq \alpha_{0}$ we have
\[\delta(\alpha,\tilde{\boldsymbol{\lambda}},\boldsymbol{d}) = \delta(\alpha,\boldsymbol{\lambda},\boldsymbol{d})\geq 0,\]
and for $\alpha\in(\alpha_{0},0)$
\[\delta(\alpha,\tilde{\boldsymbol{\lambda}},\boldsymbol{d}) = \delta(\alpha,\boldsymbol{\lambda},\boldsymbol{d}) - \sum_{k\in\N : \lambda_{i_{k}}\leq\alpha}(\alpha-\lambda_{i_{k}})\geq \frac{\sigma}{2} - \sum_{k=1}^{\infty}|\lambda_{i_{k}}| = \frac{\sigma}{2} - s>0.\]
Moreover,
\[\liminf_{\alpha\nearrow0}\delta(\alpha,\tilde{\boldsymbol{\lambda}},\boldsymbol{d}) = \sigma_{-}-s.\]
This completes the proof of the lemma under the additional assumption that the set $\{i\in\N:\lambda_{i}\leq 0\}$ is infinite; however, we have not concluded yet that the support of $\tilde{\boldsymbol\lambda}_{-}$ is infinite in the case that the support of ${\boldsymbol\lambda}_{-}$ is finite.


To obtain the missing conclusion in the lemma, we set $\boldsymbol \lambda'=-\tilde{\boldsymbol \lambda}$ and $\boldsymbol d'=-\boldsymbol d$. Note that the set $\{i\in \N: \lambda_i' \le 0\}$ is infnite. Hence, we can apply the already shown variant of Lemma \ref{red} to the pair $(\boldsymbol \lambda', \boldsymbol d')$ to obtain $\overline{\boldsymbol \lambda}$. By construction $\overline{\boldsymbol \lambda}_+$ has infinite support. Since $({\boldsymbol \lambda}')_-$ has infinite support $\overline{\boldsymbol \lambda}_-$ also has infinite support. Hence, $-\overline{\boldsymbol \lambda}$ is the desired sequence fulfilling all conclusions of Lemma \ref{red}.
Finally, by replacing $(\boldsymbol \lambda, \boldsymbol d)$ by $(-\boldsymbol \lambda, -\boldsymbol d)$ we can easily deal with the symmetric case when $\{i\in\N:\lambda_{i}\geq 0\}$ is infinite.
\end{proof}

It takes considerably more effort to remove the positivity assumption on the target diagonal $\boldsymbol d$ in Theorem \ref{posdiag} since $\boldsymbol d$ might have both positive and negative parts and possibly infinite number of zero terms. We shall employ a technique of annihilation of excesses, which relies on the following purely sequential lemma.

\begin{lem}\label{exequal} Let $(\lambda_{i})_{i\in\N}$, $(\lambda_{i})_{i\in-\N}$, $(d_{i})_{i\in\N}$, and $(d_{i})_{i\in-\N}$ be nonnegative sequences in $c_{0}$.
Suppose that
\begin{align}
\label{exequal0}
\lambda_{i} & \ge d_{i} \qquad\text{for all }i\in\N \cup - \N,
\\
\label{exequal1}
\lambda_{i} & > d_{i} \qquad \text{for infinitely many }i\in\N\text{ and infinitely many } i\in - \N,
\\\label{exequal3}
\sigma : &= \sum_{i\in\N}(\lambda_{i} - d_{i}) = \sum_{i\in-\N}(\lambda_{i} - d_{i}).
\end{align}
Then, there exist a positive sequence $(\tilde{\lambda}_{i})_{i\in\N}$ in $c_{0}$, a partition $(I_k)_{k\in \N}$ of $\N$ into finite sets, and two increasing sequences of natural numbers $(m_{k})_{k\in\N}$ and $(n_{k})_{k\in\N}$ such that the following three conclusions hold:
\begin{align}\label{exequal2}
(\tilde \lambda_i)_{i\in I_k} \preccurlyeq ( \lambda_i)_{i\in I_k}
&\qquad\text{for all }k\in\N,\\[2ex]
\label{exequal4}
\tilde{\lambda}_{i}\geq d_{i} &\qquad\text{for all }i\in\N,
\\
\label{exequal5}
\sum_{i=1}^{n_{k}}(\tilde{\lambda}_{i} - d_{i}) = \sum_{i=-1}^{-m_{k}}({\lambda}_{i} - d_{i})&\qquad \text{for all }k\in\N.
\end{align}
\end{lem}

By \eqref{exequal0} the sums in \eqref{exequal3} are well-defined and they possibly take the value $\sigma=\infty$.

\begin{proof} For each $k\in\N\cup-\N$, set
\[
\delta_{k} = 
\begin{cases}\sum_{i=1}^{k}({\lambda}_{i} - d_{i}) & k\in\N,\\ \sum_{i=-1}^{k}(\lambda_{i} - d_{i}) & k\in-\N.\end{cases}
\]
By \eqref{exequal0}, \eqref{exequal1}, and \eqref{exequal3}
\begin{equation}\label{exequal7}
\delta_k < \sigma \qquad\text{for all }k \in \N \cup -\N,
\end{equation}
and
\begin{equation}\label{exequal8}
\lim_{k\to \infty} \delta_k = \lim_{k\to \infty} \delta_{-k} = \sigma.
\end{equation}
Suppose that for some $k\in \N$ we have defined disjoint sets $I_1,\ldots,I_{k-1}$ such that 
\[
I_1 \cup \ldots \cup I_{k-1} = \{1,\ldots,M\},
\]
where $M\in \N$, a sequence $(\tilde{\lambda}_{i})_{i=1}^{M}$, and sequences $n_1,\ldots,n_{k-1}$ and $m_1,\ldots,m_{k-1}$ satisfying
\begin{equation*}
\begin{aligned}
(\tilde \lambda_i)_{i\in I_j} \preccurlyeq (\lambda_{i})_{i\in I_j}
&\qquad\text{for }j=1,\ldots,k-1,
\\[2ex]
\tilde{\lambda}_{i}\geq d_{i} & \qquad\text{for }i=1,\ldots, M,
\\
\sum_{i=1}^{n_{j}}(\tilde{\lambda}_{i} - d_{i}) = \sum_{i=-1}^{-m_{j}}({\lambda}_{i} - d_{i})
&\qquad \text{for }j=1,\ldots,k-1.
\end{aligned}
\end{equation*}
In the base case of $k=1$ we let $M=0$ and $m_0=0$.
Let 
\[
m_{k} =\min\{i\ge m_{k-1}+1: \delta_{-i} > \delta_{M+1}\},
\qquad
n_{k} = \min\{i\ge M+1: \delta_i \ge \delta_{-m_{k}}\}.
\]
These numbers are well-defined by \eqref{exequal7} and \eqref{exequal8}.

Let $\eta=\delta_{n_{k}}-\delta_{-m_{k}}$. Choose $r>n_{k}$ and $N$ such that
$\lambda_{r}<d_{n_{k}}$ and $N(d_{n_{k}}-\lambda_{r})>\eta$. Define $I_{k}=\{M+1,\ldots,r+N-1\}$ and $(\tilde{\lambda}_{i})_{i\in I_{k}}$, where
\[
\tilde \lambda_i = \begin{cases} \lambda_{n_{k}}-\eta & i=n_{k}, \\
\lambda_i + \eta/N & i=r,\ldots,r+N-1, \\
\lambda_i & \text{otherwise.}
\end{cases}
\]
By the above definition we have $(\tilde \lambda_i)_{i\in I_{k}} \preccurlyeq (\lambda_i)_{i\in I_{k}}$. 
By the minimality of $n_{k}$ we have $\delta_{n_{k}-1} <\delta_{-m_{k}} \le \delta_{n_{k}}$, which implies that $\eta<\lambda_{n_{k}}-d_{n_{k}}$. Hence, $\tilde \lambda_{n_{k}} \ge d_{n_{k}}$. Finally,
\[
\sum_{i=1}^{n_{k}}(\tilde{\lambda}_{i} - d_{i}) = \delta_{n_{k}}- \eta = \delta_{-m_{k}}=\sum_{i=-1}^{-m_{k}}({\lambda}_{i} - d_{i}).
\]
By induction this yields \eqref{exequal2}, \eqref{exequal4}, and \eqref{exequal5}.
\end{proof}

Using Lemma \ref{exequal} we can prove an initial form of Theorem \ref{step5}  for initial and target diagonal sequences satisfying dominant majorization.

\begin{thm}\label{exel} Let $(\lambda_{i})_{i\in\N}$, $(\lambda_{i})_{i\in-\N}$, $(d_{i})_{i\in\N}$, and $(d_{i})_{i\in-\N}$ be nonnegative sequences in $c_{0}$ satisfying \eqref{exequal0}, \eqref{exequal1}, and \eqref{exequal3}.
Set $\boldsymbol\lambda = (\sgn(i)\lambda_{i})_{i\in\N\cup -\N}$ and $\boldsymbol d = (\sgn(i)d_{i})_{i\in\N\cup -\N}$. If 
 $E$ is a self-adjoint operator with diagonal $\boldsymbol\lambda$, then $\boldsymbol d$ is also a diagonal of $E$.
\end{thm}

\begin{proof} By Lemma \ref{exequal} there exist a positive sequence $(\tilde{\lambda}_{i})_{i\in\N}$ in $c_{0}$, a partition $(I_{k})_{k\in\N}$ of $\N$ into finite sets, and two increasing sequences $(m_{k})_{k\in\N}$ and $(n_{k})_{k\in\N}$ such that \eqref{exequal2}, \eqref{exequal4}, and \eqref{exequal5} hold. Set $\tilde{\lambda}_{i}=\lambda_{i}$ for $i\in-\N$.

By \eqref{exequal2} and the Schur-Horn theorem, if $E_{k}$ is any self-adjoint operator with diagonal $(\lambda_{i})_{i\in I_{k}}$, then $(\tilde{\lambda}_{i})_{i\in I_{k}}$ is also a diagonal of $E_{k}$. Hence, by Lemma \ref{subdiag} the sequence $(\sgn(i)\tilde{\lambda}_{i})_{i\in\N\cup -\N}$ is a diagonal of $E$.

Set $m_{0} = n_{0} = 0$. For each $k\in\N$ set
\[J_{k} = \{-m_{k},\ldots,-m_{k-1}-1\}\cup\{n_{k-1}+1,\ldots,n_{k}\}.\]
From \eqref{exequal4} and \eqref{exequal5} we can deduce that
\[(\sgn(i)d_{i})_{i\in J_{k}}\preccurlyeq (\sgn(i)\tilde{\lambda}_{i})_{i\in J_{k}}\quad\text{for all }k\in\N.\]
Hence, by the Schur-Horn theorem and Lemma \ref{subdiag}, the sequence $\boldsymbol d$ is a diagonal of $E$.
\end{proof}

In the proof of Theorem \ref{step5} we will also employ two elementary lemmas.

\begin{lem}\label{fis} Let $(\lambda_{i})_{i\in\N}$ be a positive nonincreasing sequence in $c_{0}$. Let $(d_{i})_{i=1}^N$ be a positive nonincreasing sequence of length $N\in \N$. Suppose that
\begin{equation}\label{fis0}
\sum_{i=1}^{k}(\lambda_{i} - d_{i})\geq 0\quad\text{for all }k=1,\ldots,N.
\end{equation}

Then, there exists a positive sequence $(\tilde \lambda_{i})_{i\in\N}$ in $c_0$ and $M\in \N$ such that:
\begin{enumerate}
\item $\tilde \lambda_i = d_i$ for $i=1,\ldots,N$,
\item $\tilde \lambda_i = \lambda_i$ for $i\ge M+1$, and
\item $(\tilde \lambda_i)_{i=1}^M \preccurlyeq (\lambda_{i})_{i=1}^M$.
\end{enumerate}
\end{lem}

\begin{proof}
Let $M\in \N$ be the smallest number such that
\[
\sum_{i=1}^M \lambda_i - \sum_{i=1}^N d_i - (M-N) d_N \le 0.
\]
The existence of such $M$ follows from the assumption that $\lambda_i \to 0$ as $i\to \infty$. By \eqref{fis0} we have that $M\ge N$. If $M=N$, then we are done. Hence, we can assume that $M>N$. Define $(\tilde \lambda_{i})_{i\in\N}$ so that (i) and (ii) hold and
\[
\tilde \lambda_i = 
\begin{cases} d_N & \text{for } i=N+1,\ldots,M-1,\\
\sum_{i=1}^M \lambda_i - \sum_{i=1}^N d_i - (M-N-1) d_N & \text{for }i=M.
\end{cases}
\]
By the minimality of $M$ we have $(\tilde \lambda_i)_{i=1}^{M-1} \prec (\lambda_{i})_{i=1}^{M-1}$. Moreover, by the definition of $\tilde \lambda_M$ we have
\[
\sum_{i=1}^M \tilde \lambda_i = \sum_{i=1}^N d_i + (M-N-1) d_N + \tilde \lambda_{M} = \sum_{i=1}^M \lambda_i.
\]
Hence, (iii) holds and the proof is complete.
\end{proof}

\begin{lem}\label{fiz}
Suppose that $(\lambda_{i})_{i\in \N}$ is a sequence in $c_0$ of nonzero numbers that contains infinitely many positive terms and infinitely many negative terms. Then, for any $M\in \N$ there exists a sequence $(\tilde{\lambda}_{i})_{i\in \N}$ and a finite subset $J\subset \N$ such that
\begin{enumerate}
\item $(\tilde \lambda_i)_{i\in J} \preccurlyeq (\lambda_{i})_{i\in J}$,
\item $\tilde \lambda_i = \lambda_i$ for $i\in \N \setminus J$, and
\item $(\tilde \lambda_i)_{i\in \N}$ contains exactly $M$ zero terms. 
\end{enumerate}
\end{lem}

The assumption that $(\lambda_{i})_{i\in \N}$ belongs to $c_0$ is not essential; it is only made for the simplicity of the proof. 

\begin{proof}
Choose a finite subset $J\subset \N$ of size $M+1$ such that $\lambda_i>0$ for exactly one $i\in J$ and $\sum_{i \in J} \lambda_i >0$. Let $i_0\in J$ be the unique element such that $\lambda_{i_0}>0$. Define $(\tilde \lambda_{i})_{i\in\N}$ such that (ii) holds and
\[
\tilde \lambda_i = 
\begin{cases} 0 & \text{for } i\in J \setminus \{i_0\},\\
\lambda_{i_0} -\sum_{i\in J \setminus \{i_0\}} \lambda_i.
\end{cases}
\]
Then, (i) and (iii) hold and the proof is complete.
\end{proof}

We are now ready to give the proof of the main result of this section.

\begin{proof}[Proof of Theorem \ref{step5}]
Let \(\sigma\) be the excess given in \eqref{pos5}. By Lemma \ref{red} there exists a diagonal $\tilde{\boldsymbol \lambda}$ of $E$ such that both $\tilde{\boldsymbol \lambda} _+$ and $\tilde{\boldsymbol \lambda}_-$ have infinite supports, $\tilde \lambda_i\ne 0$ for all $i\in\N$, and the conditions \eqref{pos5.0} and \eqref{pos5} hold with $\boldsymbol \lambda$ replaced by $\tilde{\boldsymbol \lambda}$ and $\sigma$ replaced by $\sigma-s>0$, $s\ge 0$. Hence, without loss of generality we can assume that the sequence $\boldsymbol \lambda$ consists only of nonzero terms and the supports of ${\boldsymbol \lambda} _+$ and ${\boldsymbol \lambda}_-$ are both infinite. 

Note that the assumptions and the conclusion of Theorem \ref{step5} do not depend on order of the terms in sequences $\boldsymbol \lambda$ and $\boldsymbol d$. Hence, we can freely rearrange terms of $(\lambda_{i})_{i\in \N}$ to obtain a sequence $(\lambda_{\pi(i)})_{i\in \N}$ by employing a bijection $\pi:\N \to \N$. The same can be done to $(d_{\pi'(i)})_{i\in \N}$ for another bijection $\pi':\N \to \N$.
Let
\[
I_-=\{i\in \N:  d_i < 0\}, \qquad I_0=\{i\in \N: d_i=0\},
\quad\text{and}\quad I_+=\{i\in \N:  d_i > 0\}.
\]
There are four cases to consider based on the form of the sequence $\boldsymbol d$.

\medskip

\noindent {\bf Case 1.} The sets $I_-$ and $I_+$ are both infinite. 
Let $\eta=\min(1,\sigma/2)$. Choose $\alpha_0>0$ such that
\begin{equation}\label{s5.2}
\delta(\alpha,\boldsymbol \lambda, \boldsymbol d) > \eta\qquad\text{for } 0<\alpha<\alpha_0.
\end{equation}
Since $\boldsymbol \lambda \in c_0$ and $\boldsymbol \lambda_+$ has infinite support, we can rearrange the terms of $\boldsymbol \lambda$ and $\boldsymbol d$ so that:
\begin{enumerate}[(a)]
\item $
0<\lambda_i <\alpha_0 $ for all $i\in I_0,$
\item $\sum_{i\in I_0} \lambda_i < \eta,$
\item $\lambda_i>0$ if and only if $i \in I_+ \cup I_0$,
\item $\lambda_i<0$ if and only if $i \in I_-$,
\item both $(\lambda_{i})_{i\in I_+}$ and $(|\lambda_i|)_{i\in I_-}$ are nonincreasing sequences,
\item both $(d_{i})_{i\in I_+}$ and $(|d_i|)_{i\in I_-}$ are nonincreasing sequences.
\end{enumerate}

Consider a truncated sequence $\boldsymbol \lambda'=(\lambda_{i})_{i\in \N \setminus I_0}$. Note that
\[
\delta(\alpha, \boldsymbol \lambda, \boldsymbol d) - 
\delta(\alpha, \boldsymbol \lambda', \boldsymbol d) 
= \sum_{i\in I_0, \ \lambda_i \ge \alpha} (\lambda_i - \alpha) 
\le \sum_{i\in I_0} \lambda_i< \eta.
\]
Moreover, $\delta(\alpha, \boldsymbol \lambda', \boldsymbol d) = 
\delta(\alpha, \boldsymbol \lambda, \boldsymbol d) \ge 0$
for $\alpha \ge \alpha_0$. Therefore, by \eqref{s5.2} we have
\[
\delta(\alpha, (\lambda_{i})_{i\in I_+},(d_{i})_{i\in I_+})=
\delta(\alpha, \boldsymbol \lambda', \boldsymbol d) \ge 0 \qquad\text{for all }\alpha>0.
\]

Let $\pi:\N \to I_+$ be the order preserving bijection.
Applying Proposition \ref{LR} to positive nonincreasing sequences $(\lambda_{\pi(i)})_{i\in \N}$ and $(d_{\pi(i)})_{i\in \N}$ we deduce that
\[
\delta_k := \sum_{i=1}^k (\lambda_{\pi(i)} -  d_{\pi(i)}) \ge 0 \qquad\text{for }k\in \N\]
and
\[
\liminf_{k\to \infty} \delta_k = \liminf_{\alpha \searrow 0} \delta(\alpha, \boldsymbol \lambda', \boldsymbol d)
= \sigma - \sum_{i\in I_0} \lambda_i>0.
 \]
 By Lemma \ref{midseq}, there exists a nonincreasing positive sequence $(\tilde \lambda_i)_{i\in I_+}$ such  that
 \begin{equation}\label{s5.5}
 \tilde \lambda_{\pi(i)} \ge d_{\pi(i)} \qquad\text{for all }i\in \N,
 \end{equation}
 with strict inequality for all but finitely many $i\in \N$,
 \begin{equation}\label{s5.6}
 \tilde \delta_k := \sum_{i=1}^k (\lambda_{\pi(i)} - \tilde \lambda_{\pi(i)}) \ge 0 \qquad\text{for all }k\in \N,
 \end{equation}
 \begin{equation}\label{s5.7}
 \liminf_{k\to \infty} \tilde \delta_k=0,
 \end{equation}
 \begin{equation}\label{s5.8}
 \sum_{i=1}^\infty (\tilde \lambda_{\pi(i)} - d_{\pi(i)})= \sigma - \sum_{i\in I_0} \lambda_i.
 \end{equation}
 

Let $\pi':\N \to I_-$ be the order preserving bijection. Applying the above procedure to sequences $(|\lambda_i|)_{i\in I_-}$ and $(|d_i|)_{i\in I_-}$ with the help of Proposition \ref{LR} and Lemma \ref{midseq} we deduce the existence of a nondecreasing negative sequence $(\tilde \lambda_i)_{i\in I-}$ 
such  that
 \begin{equation}\label{s5.10}
 |\tilde \lambda_{\pi'(i)}| \ge |d_{\pi'(i)}| \qquad\text{for all }i\in \N,
 \end{equation}
 with strict inequality for all but finitely many $i\in \N$,
 \begin{equation}\label{s5.11}
 \tilde \delta'_k := \sum_{i=1}^k (|\lambda_{\pi'(i)}| - |\tilde \lambda_{\pi'(i)}|) \ge 0 \qquad\text{for all }k\in \N,
 \end{equation}
 \begin{equation}\label{s5.12}
 \liminf_{k\to \infty} \tilde \delta'_k=0,
 \end{equation}
 \begin{equation}\label{s5.13}
 \sum_{i=1}^\infty (|\tilde \lambda_{\pi'(i)}| - |d_{\pi'(i)}|)= \sigma.
 \end{equation}
Finally, we set $\tilde \lambda_i= \lambda_i$ for $i\in I_0$.
Applying Lemma \ref{subdiag} and Proposition \ref{posSH} for the pair $((\lambda_{i})_{i\in I_+}, (\tilde \lambda_i)_{i\in I_+})$ and then for the pair $((\lambda_{i})_{i\in I_-}, (\tilde \lambda_i)_{i\in I_-})$, 
shows that $\tilde{\boldsymbol \lambda}=(\tilde{\lambda}_{i})_{i\in \N}$ is a diagonal of $E$. 

Observe that $(\tilde \lambda_{i})_{i\in I_0 \cup I_+}$, $(|\tilde \lambda_{i}|)_{i\in I_-}$, $(d_{i})_{i\in I_0 \cup I_+}$, and $(|d_{i}|)_{i\in I_-}$ are nonnegative sequences in $c_{0}$ satisfying
\begin{align}
\label{s5.20}
|\tilde \lambda_{i}| \ge |d_{i}| \qquad&\text{for all }i\in\N,
\\
\label{s5.22}
|\tilde \lambda_{i}| > |d_{i}| \qquad&\text{for all but finitely many }i\in\N,
\\
\sum_{i\in I_0 \cup I_+ }(\tilde \lambda_{i} - d_{i}) &=
\sum_{i\in I_+ }(\tilde \lambda_{i} - d_{i})
+\sum_{i\in I_0 }(\lambda_{i} - 0)
=\sigma
 = \sum_{i\in I_-}(|\tilde \lambda_{i}| - |d_{i}|).
\end{align}
Therefore, by Theorem \ref{exel}, $\boldsymbol d$ is a diagonal of $E$.

\medskip

\noindent {\bf Case 2.} The sets $I_-$ and $I_+$ are both finite. This necessarily implies that the set $I_0$ is infinite. Hence, we can split $I_0$ into two infinite sets $I_0^-$ and $I_0^+$. By rearranging the terms of $\boldsymbol d$ and $\boldsymbol \lambda$, we can assume without loss of generality that:
\begin{enumerate}[(a)]
\item $I_-=\{1,\ldots,N\}$ and $I_+=\{N+1,\ldots,N'\}$ for some $N,N'\in \N$,
\item
 both $(d_{i})_{i\in I_+}$ and $(|d_i|)_{i\in I_-}$ are nonincreasing sequences,
\item $\lambda_i>0$ if and only if $i\in I_+ \cup I_0^+$, 
\item $\lambda_i<0$ if and only if $i\in I_- \cup I_0^-$,
\item both sequences $(\lambda_{i})_{i \in I_+ \cup I_0^+}$ and $(|\lambda_i|)_{i \in I_- \cup I_0^-}$ are nonincreasing.
\end{enumerate}
By Proposition \ref{LR}, the assumptions for $(\lambda_{i})_{i\in I_+ \cup I_0^+}$ and $(d_{i})_{i\in I_+}$ in Lemma \ref{fis} are met. Hence, there exists a positive sequence $(\tilde \lambda_i)_{i\in I_+ \cup I_0^+}$ and finite set $I_+'$ such that:
\begin{enumerate}[(i)]
\item $I_+ \subset I_+' \subset I_+ \cup I_0^+$,
\item $(\tilde \lambda_i)_{i\in I_+'} \preccurlyeq (\lambda_{i})_{i\in I_+'}$,
\item $\tilde \lambda_i = \lambda_i$ for $i\in I_0^+ \setminus I_+'$, and
\item $\tilde \lambda_i = d_i$ for $i\in I_+$. 
\end{enumerate}
Likewise, Proposition \ref{LR} and Lemma \ref{fis} applied for $(|\lambda_i|)_{i\in I_- \cup I_0^-}$ and $(|d_i|)_{i\in I_-}$ implies that there exists a negative sequence $(\tilde \lambda_i)_{i\in I_- \cup I_0^-}$ and finite set $I_-'$ such that:
\begin{enumerate}[(i')]
\item $I_- \subset I_-' \subset I_- \cup I_0^-$,
\item $(\tilde \lambda_i)_{i\in I_-'} \preccurlyeq (\lambda_{i})_{i\in I_-'}$,
\item $\tilde \lambda_i = \lambda_i$ for $i\in I_0^- \setminus I_+'$, and
\item $\tilde \lambda_i = d_i$ for $i\in I_-$. 
\end{enumerate}

By Lemma \ref{subdiag} and the Schur-Horn theorem we deduce that $(\tilde \lambda_i)_{i\in \N}$ is a diagonal of $E$. Since $\tilde \lambda_i = d_i$ for all $i\in \N \setminus I_0$, by Lemma \ref{subdiag} it suffices to show that whenever $(\tilde{\lambda}_{i})_{i\in I_0}$ is a diagonal of some self-adjoint operator $E_0$, then $(d_{i})_{i\in I_0}=(0)_{i\in I_0}$ is also a diagonal of $E_0$. This is an easy consequence of Theorem \ref{exel} applied to sequences
$(\tilde \lambda_{i})_{i\in I_0^+}$, $(\tilde \lambda_{i})_{i\in I_0^-}$, $(d_{i})_{i\in I_0^+}$, and $(d_{i})_{i\in I_0^-}$ since $\tilde \lambda_i> d_i=0$ for all $i\in I_0$, and
\[
\sum_{i\in I_0^+} \tilde \lambda_i =
\sum_{i\in I_+ \cup I_0^+} (\lambda_i -d_i) = \sigma
=
\sum_{i\in I_- \cup I_0^-} (|\lambda_i| -|d_i|)
= \sum_{i\in I_0^-} |\tilde \lambda_i|.
\]

\medskip

\noindent {\bf Case 3.} Exactly one of the sets $I_-$ and $I_+$ is infinite and $I_0$ is infinite. By symmetry, we can assume that $I_+$ is infinite and $I_-$ is finite. By rearranging the terms of $\boldsymbol d$ and $\boldsymbol \lambda$, we can assume without loss of generality that: 
\begin{enumerate}[(a)]
\item $I_-=\{1,\ldots,N\}$ for some $N\in \N$,
\item $\lambda_i>0$ if and only if $i\in I_+$, 
\item $\lambda_i<0$ if and only if $i\in I_- \cup I_0$,
\item sequences $(\lambda_{i})_{i \in I_+}$, $(d_{i})_{i\in I_+}$, $(|\lambda_i|)_{i \in I_- \cup I_0}$, and $(|d_i|)_{i \in I_- \cup I_0}$ are nonincreasing.
\end{enumerate}
Applying Proposition \ref{LR} and then Lemma \ref{midseq} for  $(\lambda_{i})_{i \in I_+}$ and $(d_{i})_{i\in I_+}$, we deduce the existence a nonincreasing positive sequence $(\tilde \lambda_i)_{i\in I_+}$ such that \eqref{s5.5}-\eqref{s5.7} hold and
 \begin{equation}\label{s5.9}
 \sum_{i=1}^\infty (\tilde \lambda_{\pi(i)} - d_{\pi(i)})= \sigma.
 \end{equation}
Likewise, applying Proposition \ref{LR} and then Lemma \ref{midseq} for $(|\lambda_i|)_{i \in I_- \cup I_0}$, and $(|d_i|)_{i \in I_- \cup I_0}$, yields a nondecreasing negative sequence $(\tilde \lambda_i)_{i\in I_- \cup I_0}$ such that \eqref{s5.10}-\eqref{s5.13} hold, where $\pi': \N \to I_- \cup I_0$ is the order preserving bijection. 
Applying Lemma \ref{subdiag} and Proposition \ref{posSH} twice
shows that $\tilde{\boldsymbol \lambda}=(\tilde{\lambda}_{i})_{i\in \N}$ is a diagonal of $E$.

Observe that $(\tilde \lambda_{i})_{i\in I_+}$, $(|\tilde \lambda_{i}|)_{i\in I_-\cup I_0}$, $(d_{i})_{i\in I_+}$, and $(|d_{i}|)_{i\in I_- \cup I_0}$ are nonnegative sequences in $c_{0}$ satisfying \eqref{s5.20} and \eqref{s5.22}. Moreover, by \eqref{s5.13} and \eqref{s5.9} we have
\[
\sum_{i\in I_- \cup I_0 }(|\tilde \lambda_{i}| - |d_{i}|) 
=\sigma =
\sum_{i\in I_+ }(\tilde \lambda_{i} - d_{i}).
\]
Therefore, by Theorem \ref{exel}, $\boldsymbol d$ is a diagonal of $E$.

\noindent {\bf Case 4.} Exactly one of the sets $I_-$ and $I_+$ is infinite and $I_0$ is finite. By symmetry, we can assume that $I_-$ is infinite and $I_+$ is finite. We partition $I_-$ into two infinite sets $I_1$ and $I_2$. 
By rearranging the terms of $\boldsymbol d$ and $\boldsymbol \lambda$, we can assume without loss of generality that: 
\begin{enumerate}[(a)]
\item $I_+=\{1,\ldots,N\}$, $I_0=\{N+1,\ldots,N'\}$, $I_-=\{N'+1,N'+2,\ldots\}$ for some $N,N'\in \N$,
\item $\lambda_i>0$ if and only if $i\in I_+\cup I_0 \cup I_1$, 
\item $\lambda_i<0$ if and only if $i\in I_2$,
\item sequences $(\lambda_{i})_{i \in I_+ \cup I_0 \cup I_1}$, $(|\lambda_i|)_{i \in I_2}$, $(d_{i})_{i\in I_+}$, and $(|d_i|)_{i \in I_-}$ are nonincreasing.
\end{enumerate}

By Lemma \ref{fis} applied to sequences $(\lambda_{i})_{i \in I_+ \cup I_0 \cup I_1}$ and $(d_{i})_{i\in I_+}$, we can find a positive sequence $(\tilde \lambda_i)_{i \in I_+ \cup I_0 \cup I_1}$ and a finite subset $J \subset  I_+ \cup I_0 \cup I_1$ such that: 
\begin{enumerate}[(i)]
\item $(\tilde \lambda_i)_{i\in J} \preccurlyeq (\lambda_{i})_{i \in J}$,
\item $\tilde \lambda_i = \lambda_i$ for $i\in (I_+ \cup I_0 \cup I_1) \setminus J$, and
\item $\tilde \lambda_i = d_i$ for $i\in I_+$. 
\end{enumerate}
Set $\tilde \lambda_i = \lambda_i$ for $i\in I_2$. Then, by Lemma \ref{subdiag} and the Schur-Horn theorem, $\tilde{\boldsymbol \lambda}= (\tilde \lambda_i)_{i\in \N}$ is a diagonal of $E$. By (i)--(iii), one can easily show that sequences $\tilde{\boldsymbol \lambda}$ and $\boldsymbol d$ satisfy the assumptions \eqref{pos5.0} and \eqref{pos5} with the same $\sigma$. Hence, by replacing $\boldsymbol \lambda$ by $\tilde{\boldsymbol \lambda}$ we can now assume that $\lambda_i = d_i$ for $i\in I_+$.

Let $\eta=\min(1,\sigma/2)$. Choose $\alpha_0>0$ such that
\begin{equation}\label{s5.21}
\delta(\alpha,\boldsymbol \lambda, \boldsymbol d) > \eta\qquad\text{for } 0<|\alpha|<\alpha_0.
\end{equation}
Then, choose a subset $J_\infty \subset \{N+1,\ldots\}$ so that a subsequence $(\lambda_{i})_{i\in J_\infty}$ contains infinitely many positive and infinitely many negative terms,  and
\begin{equation}\label{s5.23}
|\lambda_i|<\alpha_0 \quad\text{for }i\in J_\infty \quad\text{and}\quad
\sum_{i\in J_\infty} |\lambda_i| < \eta.
\end{equation}

Next, we apply Lemma \ref{fiz} to a sequence $(\lambda_{i})_{i \in J_\infty}$ to find a finite set $J\subset J_\infty$ and a sequence $(\tilde \lambda_i)_{i\in J_\infty}$ such that: 
\begin{enumerate}[(i')]
\item $(\tilde \lambda_i)_{i\in J} \preccurlyeq (\lambda_{i})_{i \in J}$,
\item $\tilde \lambda_i = \lambda_i$ for $i\in J_\infty \setminus J$, and
\item $(\tilde \lambda_i)_{i\in J_\infty}$ has exactly $N'-N=\# |I_0|$ zeros. 
\end{enumerate}

Set $\tilde \lambda_i = \lambda_i$ for $i\in \N \setminus J_\infty$. Then, by Lemma \ref{subdiag} and the Schur-Horn theorem, $\tilde{\boldsymbol \lambda}= (\tilde \lambda_i)_{i\in \N}$ is a diagonal of $E$. By \eqref{s5.21} and \eqref{s5.23}, one can show that sequences $\tilde{\boldsymbol \lambda}$ and $\boldsymbol d$ satisfy \eqref{pos5.0}. Moreover, by (i')--(iii'), one can show that sequences $\tilde{\boldsymbol \lambda}$ and $\boldsymbol d$ satisfy \eqref{pos5}, possibly with a different value for the excesses. Moreover, by rearranging terms $(\tilde \lambda_i)_{i\ge N+1}$, we can assume that $\tilde \lambda_i =0$ if and only if $i\in I_0$ in addition to already proven property that $\tilde \lambda_i = d_i$ for $i\in I_+$. Finally, it remains to apply Theorem \ref{posdiag} to sequences $(\tilde \lambda_i)_{i\in I_-}$ and $(\tilde{d}_{i})_{i\in I_-}$ and Lemma \ref{subdiag} to deduce that $\boldsymbol d$ is a diagonal of $E$.
\end{proof}

\section{Unequal excesses diagonal-to-diagonal result}

The goal of this section is the following diagonal-to-diagonal result in one-sided non-summable case when excesses are not equal. 

\begin{thm}\label{step6} Let $\boldsymbol\lambda = (\lambda_{i})_{i\in\N}$ and $\boldsymbol d = (d_{i})_{i\in\N}$ be sequences in $c_{0}$ such that
\begin{equation}\label{s6.0}
\sum_{\lambda_i<0} |\lambda_i| = \sum_{d_i<0} |d_i| = \infty \quad\text{and}\quad \sum_{\lambda_i>0} \lambda_i <\infty. 
\end{equation}
If there is a self-adjoint operator $E$ with diagonal $\boldsymbol\lambda$, 
\begin{equation}\label{s6.1}\delta(\alpha,\boldsymbol\lambda, \boldsymbol d)\geq 0\quad\text{for all }\alpha\neq 0,\end{equation}
and
\begin{equation}\label{s6.2}
\liminf_{\alpha\nearrow 0}\delta(\alpha,\boldsymbol\lambda, \boldsymbol d) > \liminf_{\alpha\searrow 0}\delta(\alpha,\boldsymbol\lambda, \boldsymbol d)=\sum_{\lambda_i>0} \lambda_i - \sum_{d_i>0} d_i > 0,\end{equation}
then $\boldsymbol d$ is also a diagonal of $E$.
\end{thm}

Observe that \eqref{s6.1} implies that
\[
\sum_{d_i>0} d_i \le \sum_{\lambda_i>0} \lambda_i  <\infty.
\]
Hence, the right hand side of \eqref{s6.2} is well-defined.

The following is a convenient reformulation of a lemma by Kaftal and Weiss \cite[Lemma 5.2]{kw}.

\begin{lem}\label{kwlem} Let \(\boldsymbol\lambda=(\lambda_{i})_{i\in J}\) and \(\boldsymbol d = (d_{i})_{i\in I}\) be positive sequences in \(c_{0}\setminus\ell^{1}\) such that \(\boldsymbol d\prec\boldsymbol\lambda\) and \(\boldsymbol d\not\preccurlyeq\boldsymbol\lambda\). Let \(i_{1}\in I\) and \(j_{1}\in J\) such that \(d_{i_{1}} = \max\boldsymbol d\) and \(\lambda_{j_{1}} = \max\boldsymbol\lambda\). Then, there are partitions into infinite sets \(I=I_{1}\cup I_{2}\) and \(J=J_{1}\cup J_{2}\) with \( i_{1}\in I_{1}\) and \(j_{1}\in J_{1}\) such that \(\boldsymbol d|_{I_{1}}\in \ell^{1}\), \(\boldsymbol d|_{I_{1}}\preccurlyeq \boldsymbol\lambda|_{J_{1}}\), \(\boldsymbol d|_{I_{2}}\prec \boldsymbol\lambda|_{J_{2}}\), and \(\boldsymbol d|_{I_{2}}\not\preccurlyeq \boldsymbol\lambda|_{J_{2}}.\)
\end{lem}

We need the following diagonal-to-diagonal generalization of the Kaftal-Weiss theorem \cite[Corollary 6.1]{kw}.

\begin{prop}\label{kwd2d}
Let $\boldsymbol\lambda=(\lambda_{i})_{i=1}^{\infty}$ and $\boldsymbol d = (d_{i})_{i=1}^{\infty}$ be positive sequences in \(c_{0}\) such that \(\boldsymbol d\prec\boldsymbol\lambda\) and 
\[
\sum_{i=1}^\infty \lambda_i = \sum_{i=1}^\infty d_i.
\]
If $\boldsymbol\lambda$ is a diagonal of a self-adjoint operator $E$,
then $\boldsymbol d$ is also a diagonal of $E$.
\end{prop}

\begin{proof}
If \(\boldsymbol d\preccurlyeq\boldsymbol\lambda\), then we can apply Proposition \ref{posSH} to obtain the desired conclusion. Thus we may assume \(\boldsymbol d\not \preccurlyeq\boldsymbol\lambda.\)

Applying Lemma \ref{kwlem} to the sequences \(\boldsymbol\lambda\) and \(\boldsymbol d\) we obtain two partitions into infinite sets \(\N=I_{1}^{1}\cup I_{2}^{1}\) and \(\N=J_{1}^{1}\cup J_{2}^{1}\).  Since \(\boldsymbol d|_{I_{2}^{1}}\prec \boldsymbol\lambda|_{J_{2}^{1}}\), and \(\boldsymbol d|_{I_{2}^{1}}\not\preccurlyeq \boldsymbol\lambda|_{J_{2}^{1}}\), we may apply Lemma \ref{kwlem} to the sequences \(\boldsymbol\lambda|_{J_{2}^{1}}\) and \(\boldsymbol d|_{I_{2}^{1}}\) to obtain partitions \(I_{2}^{1} = I_{1}^{2}\cup I_{2}^{2}\) and \(J_{2}^{1} = J_{1}^{2}\cup J_{2}^{2}\). Carrying on in this way we obtain sequences of disjoint sets \((I_{1}^{j})_{j=1}^{\infty}\) and \((J_{1}^{j})_{j=1}^{\infty}\). By Lemma \ref{kwlem}, the set \(I_{1}^{j}\) contains the index of the largest term of \(\boldsymbol d|_{I_{2}^{j-1}}\), and similarly for \(J_{1}^{j}\). Since \(\boldsymbol d\) and \(\boldsymbol\lambda\) are positive an in \(c_{0}\) this implies that
\[\bigcap_{j=1}^{\infty}I_{2}^{j} = \bigcap_{j=1}^{\infty}J_{2}^{j} = \varnothing.\]
Therefore, \((I_{1}^{j})_{j=1}^{\infty}\) and \((J_{1}^{j})_{j=1}^{\infty}\) are partitions of \(\N\). Again, by Lemma \ref{kwlem}, for each \(j\in\N\) we have \(\boldsymbol d|_{I_{1}^{j}}\preccurlyeq\boldsymbol\lambda|_{I_{1}^{j}}\).
Finally, by Proposition \ref{posSH} and Lemma \ref{subdiag} we can conclude that \(\boldsymbol d\) is a diagonal of \(E\).
 \end{proof}

In addition, we will need the following diagonal-to-diagonal result for non-negative sequences.

\begin{prop}\label{kwnonneg}
Let $\boldsymbol\lambda=(\lambda_{i})_{i=1}^{\infty}$ and $\boldsymbol d = (d_{i})_{i=1}^{\infty}$ be non-negative sequences in $c_0$ such that $\boldsymbol d \prec\boldsymbol\lambda$, $\boldsymbol d \not\preccurlyeq \boldsymbol\lambda$, and
\[
\sum_{i=1}^\infty \lambda_i = \sum_{i=1}^\infty d_i = \infty.
\]
Moreover, assume 
\begin{equation}\label{kwnonneg0}\#|\{i : \lambda_{i}=0\}|\geq \#|\{i : d_{i} = 0\}|.\end{equation}
If $\boldsymbol\lambda$ is a diagonal of a self-adjoint operator $E$,
then $\boldsymbol d$ is also a diagonal of $E$.
\end{prop}

\begin{proof} By Lemma \ref{subdiag} and \eqref{kwnonneg0} we may reduce to the case that \(\boldsymbol d\) is has no terms equal to zero. 

First, consider the case that \(J = \{i : \lambda_{i}=0\}\) is an infinite set. 
By the assumption that $\boldsymbol d_- \not\preccurlyeq \boldsymbol\lambda_-$ and Proposition \ref{LR} we have $\liminf_{\alpha\searrow 0}\delta(\alpha,\boldsymbol\lambda,\boldsymbol d)>0$. Choose $\sigma$ such that
\[0<\sigma<\liminf_{\alpha\searrow 0}\delta(\alpha,\boldsymbol\lambda,\boldsymbol d).\]
Let \(\alpha_{0}>0\) such that \(\delta(\alpha,\boldsymbol\lambda,\boldsymbol d)>\frac{\sigma}{2}\) for all \(\alpha\in(0,\alpha_{0}]\). Fix \(i_{1}\in\N\) such that \(0<\lambda_{i_{1}}<\min\{\alpha_{0},\frac{\sigma}{2}\}.\)
Let \(J = \{i_{2},i_{3},\ldots\}\). Define the sequence \(\tilde{\boldsymbol\lambda} = (\tilde{\lambda}_{i})_{i=1}^{\infty}\) by
\[\tilde{\lambda}_{i} = \begin{cases}  2^{-j}\lambda_{i_{1}} & i=i_{j},\ j\in\N,\\ \lambda_{i} & \text{otherwise}.\end{cases}\]

The sequences \( (\lambda_{i_j})_{j\in \N} \) and \((\tilde \lambda_{i_j})_{j\in \N}\) satisfy the assumptions of Proposition \ref{posSH}. Thus, if \(E_{1}\) is any operator with diagonal \( (\lambda_{i_j})_{j\in \N} \), then \((\tilde \lambda_{i_j})_{j\in \N}\) is also a diagonal of \(E_{1}\). Hence, by Lemma \ref{subdiag} the sequence \(\tilde{\boldsymbol\lambda}\) is a diagonal of \(E\).

If \(\alpha>\alpha_{0}\) then \(\delta(\alpha,\tilde{\boldsymbol\lambda},\boldsymbol d) = \delta(\alpha,\boldsymbol\lambda,\boldsymbol d)\geq 0\). On the other hand, if \(\alpha\in(0,\alpha_{0}]\), then
\begin{align*}
\delta(\alpha,\tilde{\boldsymbol\lambda},\boldsymbol d) & = \sum_{i:\tilde{\lambda}_{i}\geq\alpha} (\tilde{\lambda}_{i} - \alpha) - \sum_{i:d_{i}\geq\alpha}(d_{i}-\alpha) \geq \sum_{\genfrac{}{}{0pt}{2}{i:\tilde{\lambda}_{i}\geq\alpha}{i\notin\{i_{1},i_{2},\ldots\}}} (\tilde{\lambda}_{i} - \alpha) - \sum_{i:d_{i}\geq\alpha}(d_{i}-\alpha)\\
 & = \sum_{i:\lambda_{i}\geq\alpha} (\lambda_{i} - \alpha) - \sum_{i:d_{i}\geq\alpha}(d_{i}-\alpha) - (\lambda_{i_{1}} - \alpha) = \delta(\alpha,\boldsymbol\lambda,\boldsymbol d)- (\lambda_{i_{1}} - \alpha)\\
 & \geq \frac{\sigma}{2} - \lambda_{i_{1}} > 0.
\end{align*}
Thus, for these positive sequences we have \(\boldsymbol d\prec \tilde{\boldsymbol\lambda}\). By Proposition \ref{kwd2d} we conclude that \(\boldsymbol d\) is a diagonal of \(E\). 

Now, if the set \(J\) is finite, then the argument is similar. The difference being that in this case we will define a sequence \(\tilde{\boldsymbol\lambda}\) by modifying only a finite number of terms of \(\boldsymbol\lambda\). Thus, we will need to use the finite Schur-Horn theorem 
in the place of Proposition \ref{posSH} above.
\end{proof}

\begin{proof}[Proof of Theorem \ref{step6}]
Observe that by Proposition \ref{LR}
\[\sigma_{-} = \liminf_{\alpha\nearrow 0}\delta(\alpha,\boldsymbol\lambda,\boldsymbol d)\quad\text{and}\quad \sigma_{+} = \liminf_{\alpha\searrow 0}\delta(\alpha,\boldsymbol\lambda,\boldsymbol d) = \sum_{\lambda_{i}>0}\lambda_{i} - \sum_{d_{i}>0}d_{i}.\]
By the assumption \eqref{s6.2} we have \(\sigma_{-}>\sigma_{+}>0\). If $\sigma_-=\infty$, then choose $\sigma_-'$ such that $\infty >\sigma_-'>\sigma_+$. Otherwise, we let $\sigma'_-=\sigma_-$. Fix \(\alpha_{0}<0\) such that \
\[\delta(\alpha,\boldsymbol\lambda,\boldsymbol d)>\frac{\sigma'_{-}+\sigma_{+}}{2}\quad\text{for all }\alpha\in[\alpha_{0},0).\]
Fix \(i_{1}\in\N\) such that \(\lambda_{i_{1}}<0\) and
\[|\lambda_{i_{1}}|<\min\left\{\frac{\sigma'_{-}-\sigma_{+}}{2},|\alpha_{0}|,\frac{\sigma_{+}}{2}\right\}\]
Set \(I = \{i\in\N : \lambda_{i_{1}}<d_{i}<0\}\). Let \(I_{0}\subset I\) be an infinite set such that
\begin{equation}\label{step6.1}\sum_{i\in I_{0}}|d_{i}|<|\lambda_{i_{1}}|.\end{equation}
Set \(J = \{i\in\N\setminus\{i_{1}\} : \lambda_{i_{1}}<\lambda_{i}<0\}.\) Note that \((|\lambda_{i}|)_{i\in J}\) is a nonsummable sequence in \(c_{0}\). Hence, for each \(x\in(0,\infty)\) there is an infinite subset \(J_{0}\subset J\) such that
\[\sum_{i\in J_{0}}|\lambda_{i}| = x.\]
Hence, we can let \(J_{0}\subset J\) be an infinite subset such that
\begin{equation}\label{step6.2}\sum_{i\in J_{0}}|\lambda_{i}|=\sigma_{+}- |\lambda_{i_{1}}|+\sum_{i\in I_{0}}|d_{i}|.\end{equation}

Set \(J_{1} = \{i_{1}\}\cup J_{0}\cup \{i : \lambda_{i}\geq 0\}\) and \(J_{2} = \N\setminus J_{1}\). Set \(I_{1} = I_{0}\cup \{i : d_{i}\geq 0\}\) and \(I_{2} = \N\setminus I_{1}\). Moreover, define the sequences \(\boldsymbol\lambda_{1} = (\lambda_{i})_{i\in J_{1}}\), \(\boldsymbol\lambda_{2} = (\lambda_{i})_{i\in J_{2}}\), \(\boldsymbol d_{1} = (d_{i})_{i\in I_{1}}\), and \(\boldsymbol d_{2} = (d_{i})_{i\in I_{2}}\).

Define the sequence \(\tilde{\boldsymbol\lambda}_{1} = (\lambda_{i_{1}},0,0,\ldots)\) and \(\boldsymbol{\tilde{d}}_{1} = (d_{i})_{i\in I_{0}}\). For \(\alpha<0\), using Proposition \ref{LR} and \eqref{step6.1} we see that 
\[ \delta(\alpha,\boldsymbol\lambda_{1},\boldsymbol d_{1})\geq \delta(\alpha,\tilde{\boldsymbol\lambda}_{1},\boldsymbol{\tilde{d}}_{1})\geq 0.\]
For \(\alpha>0\) we have
\[\delta(\alpha,\boldsymbol\lambda_{1},\boldsymbol d_{1}) = \delta(\alpha,\boldsymbol\lambda,\boldsymbol d) \geq 0.\]
By Proposition \ref{LR} and \eqref{step6.2} we have
\[\liminf_{\alpha\nearrow 0}\delta(\alpha,\boldsymbol\lambda_{1},\boldsymbol d_{1}) = \sum_{\genfrac{}{}{0pt}{1}{\lambda_{i}<0}{i\in J_{1}}}\lambda_{i} - \sum_{\genfrac{}{}{0pt}{1}{d_{i}<0}{i\in I_{1}}}d_{i} = \sigma_{+} = \sum_{\genfrac{}{}{0pt}{1}{\lambda_{i}\geq 0}{i\in J_{1}}}\lambda_{i} - \sum_{\genfrac{}{}{0pt}{1}{d_{i}\geq 0}{i\in I_{1}}}d_{i} = \liminf_{\alpha\searrow 0}\delta(\alpha,\boldsymbol\lambda_{1},\boldsymbol d_{1})\]
Hence, by Theorem \ref{step5}, if \(E_{1}\) is any self-adjoint operator with diagonal \(\boldsymbol\lambda_{1}\), then \(\boldsymbol d_{1}\) is also a diagonal of \(E_{1}\).

For \(\alpha\leq \alpha_{0}\) we have
\[\delta(\alpha,\boldsymbol\lambda_{2},\boldsymbol d_{2}) =\delta(\alpha,\boldsymbol\lambda,\boldsymbol d)\geq 0.\]
For \(\alpha\in(\alpha_{0},0)\) we have
\begin{align*}
\delta(\alpha,\boldsymbol\lambda_{2},\boldsymbol d_{2}) & \geq \delta(\alpha,\boldsymbol\lambda,\boldsymbol d) - \sum_{\genfrac{}{}{0pt}{1}{\lambda_{i}\leq\alpha}{i\in J_{0}\cup\{i_{1}\}}}(\alpha-\lambda_{i}) \geq \delta(\alpha,\boldsymbol\lambda,\boldsymbol d) -\left( \sum_{i\in J_{0}}|\lambda_{i}| + |\lambda_{i_{1}}| \right) \\
 & > \frac{\sigma'_{-}+\sigma_{+}}{2} - \left(\sigma_{+}+\sum_{i\in I_{0}}|d_{i}|\right)=  \frac{\sigma'_{-}-\sigma_{+}}{2} - \sum_{i\in I_{0}}|d_{i}|>0.
 \end{align*}
 
Note that \(-\boldsymbol\lambda_{2}\) and \(-\boldsymbol d_{2}\) are positive sequences. Proposition \ref{LR} yields \(-\boldsymbol d_{2}\prec -\boldsymbol\lambda_{2}\). Hence, if \(E_{2}\) is any self-adjoint operator with diagonal \(\boldsymbol\lambda_{2}\), then by Proposition \ref{kwd2d} the sequence \(-\boldsymbol d_{2}\) is a diagonal of \(-E_{2}\), and hence \(\boldsymbol d_{2}\) is a diagonal of \(E_{2}\). Finally, by Lemma \ref{subdiag} the sequence \(\boldsymbol d\) is a diagonal of \(E\).
\end{proof}

\section{Two sided non-summable case}

The goal of this section to show the diagonal-to-diagonal result in a two-sided non-summable case.

\begin{thm}\label{step7} Let $\boldsymbol\lambda = (\lambda_{i})_{i\in\N}$ and $\boldsymbol d = (d_{i})_{i\in\N}$ be sequences in $c_{0}$ such that
\begin{equation}\label{s7.0}
\sum_{\lambda_i<0} |\lambda_i| =  \sum_{\lambda_i>0} \lambda_i =\infty
\end{equation}
and
\begin{equation}\label{s7.1}\delta(\alpha,\boldsymbol\lambda, \boldsymbol d)\geq 0\quad\text{for all }\alpha\neq 0.\end{equation}
Let
\[
\sigma_-=\liminf_{\alpha\nearrow 0}\delta(\alpha,\boldsymbol\lambda, \boldsymbol d), \qquad \sigma_+= \liminf_{\alpha\searrow 0}\delta(\alpha,\boldsymbol\lambda, \boldsymbol d).
\]
Suppose that the following four conditions hold:
\begin{equation}\label{s7.2}
\sigma_- + \sigma_+>0,
\end{equation}
\begin{equation}\label{s7.5}
\sigma_- =0 \text{ or } \sigma_+=0 \implies \#|\{i : \lambda_{i}=0\}|\geq \#|\{i : d_{i} = 0\}|,
\end{equation}
\begin{equation}\label{s7.3}
\sum_{d_i>0} d_i =\infty \qquad\text{or}\qquad \sigma_- =\infty,\end{equation}
and
\begin{equation}\label{s7.4}
\sum_{d_i<0} |d_i| =\infty \qquad\text{or}\qquad  \sigma_+ =\infty.\end{equation}
If there is a self-adjoint operator $E$ with diagonal $\boldsymbol\lambda$, 
then $\boldsymbol d$ is also a diagonal of $E$.
\end{thm}

Conditions \eqref{s7.2}--\eqref{s7.4} can be conveniently reformulated depending on the summability of negative and positive parts of $\boldsymbol d$ in the following way.

\begin{thm}\label{step7a} Let $\boldsymbol\lambda = (\lambda_{i})_{i\in\N}$ and $\boldsymbol d = (d_{i})_{i\in\N}$ be sequences in $c_{0}$. Let
\[
\sigma_-=\liminf_{\alpha\nearrow 0}\delta(\alpha,\boldsymbol\lambda, \boldsymbol d), \qquad \sigma_+= \liminf_{\alpha\searrow 0}\delta(\alpha,\boldsymbol\lambda, \boldsymbol d).
\]
Suppose that
\begin{equation}\label{s7.0a}
\sum_{\lambda_i<0} |\lambda_i| =  \sum_{\lambda_i>0} \lambda_i =\infty,
\end{equation}
\begin{equation}\label{s7.1a}\delta(\alpha,\boldsymbol\lambda, \boldsymbol d)\geq 0\quad\text{for all }\alpha\neq 0,\end{equation}
\begin{equation}\label{s7.e}
\sigma_- =0 \text{ or } \sigma_+=0 \implies \#|\{i : \lambda_{i}=0\}|\geq \#|\{i : d_{i} = 0\}|.
\end{equation}
In addition, suppose that one the following four conditions hold:
\begin{enumerate}[(i)]
\item
$\sum_{d_i>0} d_i = \sum_{d_i<0} |d_i|=\infty$ and $\sigma_- + \sigma_+ >0$,
\item
$\sum_{d_i<0} |d_i| <\infty $, $\sum_{d_i>0} d_i =\infty$, and $ \sigma_+ =\infty$,
\item
$\sum_{d_i>0} d_i <\infty$, $\sum_{d_i<0} |d_i| =\infty $, and $ \sigma_- =\infty$,
\item
$\sum_{i\in \N} |d_i| <\infty$.
\end{enumerate}
If there is a self-adjoint operator $E$ with diagonal $\boldsymbol\lambda$, 
then $\boldsymbol d$ is also a diagonal of $E$.
\end{thm}

\begin{proof}
We claim that, without loss of generality, we may assume 
\begin{equation}\label{s7.f}\#|\{i : \lambda_{i}=0\}|\geq \#|\{i : d_{i} = 0\}|.\end{equation}
Indeed, if \(\sigma_{+}=0\), or \(\sigma_{-}=0\), then this follows by our assumption \eqref{s7.e}. Now, assume \(\sigma_{+}> 0\) and \(\sigma_{-}> 0\). Let \(0<s<\min\{\sigma_{+},\sigma_{-}\}\). Fix \(\alpha_{0}>0\) such that
\[\delta(\alpha,\boldsymbol\lambda,\boldsymbol d) >s\quad\text{for }0<|\alpha|<\alpha_{0}.\]
Let \(I\subset \{i : -\alpha_{0}<\lambda_{i}<0\}\) and \(J\subset\{i : 0<\lambda_{i}<\alpha_{0}\}\) be infinite sets such that
\[-\sum_{i\in I}\lambda_{i} = \sum_{i\in J}\lambda_{i} = s.\]
By Theorem \ref{step5}, if \(E_{1}\) is any self-adjoint operator with diagonal \((\lambda_{i})_{i\in I\cup J}\), then a sequence consisting of a countably infinite number of zeros is also a diagonal if \(E_{1}\). Hence, by Lemma \ref{subdiag}, the sequence \(\tilde{\boldsymbol\lambda} = (\tilde{\lambda}_{i})_{i\in\N}\) given by
\[\tilde{\lambda}_{i} = \begin{cases} 0 & i\in I\cup J,\\ \lambda_{i} & i\in\N\setminus(I\cup J),\end{cases}\]
is a diagonal of \(E\).

For \(|\alpha|\geq\alpha_{0}\) we have
\[\delta(\alpha,\tilde{\boldsymbol\lambda},\boldsymbol d) = \delta(\alpha,\boldsymbol\lambda,\boldsymbol d)\geq 0.\]
For \(0<\alpha<\alpha_{0}\) we have
\[\delta(\alpha,\tilde{\boldsymbol\lambda},\boldsymbol d) = \delta(\alpha,\boldsymbol\lambda,\boldsymbol d) - \sum_{\genfrac{}{}{0pt}{1}{\lambda_{i}\geq\alpha}{i\in J}}(\lambda_{i} - \alpha) \geq \delta(\alpha,\boldsymbol\lambda,\boldsymbol d) - \sum_{i\in J}\lambda_{i}>0.\]
A similar calculation shows that \(\delta(\alpha,\tilde{\boldsymbol\lambda},\boldsymbol d)\geq 0\) for \(-\alpha_{0}<\alpha<0\). Moreover, we see that
\[\liminf_{\alpha\searrow0}\delta(\alpha,\tilde{\boldsymbol\lambda},\boldsymbol d)=\sigma_{+}-s>0\quad\text{and}\quad \liminf_{\alpha\nearrow0}\delta(\alpha,\tilde{\boldsymbol\lambda},\boldsymbol d)=\sigma_{-}-s>0.\]
Hence, by replacing \(\boldsymbol\lambda\) with \(\tilde{\boldsymbol\lambda}\) we see that in addition to the assumptions of the theorem, \eqref{s7.f} also holds.

\noindent\textbf{Case 1.} Assume (i) holds. Suppose \(\sigma_{+}>0\). By Proposition \ref{kwnonneg}, if \(E_{1}\) is any self-adjoint operator with diagonal \((\lambda_{i})_{\lambda_{i}\geq 0}\), then \((d_{i})_{d_{i}\geq 0}\) is also a diagonal of \(E_{1}\). By Proposition \ref{kwd2d}, if \(E_{2}\) is any self-adjoint operator with diagonal \((\lambda_{i})_{\lambda_{i}<0}\), then \((d_{i})_{d_{i}<0}\) is also a diagonal of \(E_{2}\). From Lemma \ref{subdiag} we conclude that \(\boldsymbol d\) is a diagonal of \(E\). If $\sigma_->0$ is handled by a symmetric argument.

\noindent\textbf{Case 2.} Assume (iv) holds. Since we are in the case that \(\sigma_{+}>0\) and \(\sigma_{-}>0\), we may assume by the argument at the beginning of the proof that \(\#|\{i : \lambda_{i} = 0\}| = \infty\). Choose \(\alpha_{0}>0\) such that
\[\delta(\alpha_{0},\boldsymbol\lambda,\boldsymbol d),\delta(-\alpha_{0},\boldsymbol\lambda,\boldsymbol d)> \max\left\{\sum_{d_{i}>0}d_{i},\sum_{d_{i}<0}|d_{i}|\right\}.\]
 Let \(\tilde{\boldsymbol\lambda} = \boldsymbol d \oplus (\lambda_{i})_{i\in J}\), where $J=\{i:|\lambda_{i}|<\alpha_{0}\}$. That is, \(\tilde{\boldsymbol\lambda}\) is the sequence consisting of all the terms of \(\boldsymbol d\) together with all terms of \(\boldsymbol\lambda\) in the interval \((-\alpha_{0},\alpha_{0})\).

For \(|\alpha|\geq \alpha_{0}\) we have \(\delta(\alpha,\boldsymbol\lambda,\tilde{\boldsymbol\lambda}) = \delta(\alpha,\boldsymbol\lambda,\boldsymbol d)\geq 0\). Meanwhile, for \(0<\alpha<\alpha_{0}\) we have
\[\delta(\alpha,\boldsymbol\lambda,\tilde{\boldsymbol\lambda}) = \sum_{\lambda_{i}\geq\alpha_{0}}(\lambda_{i} - \alpha) - \sum_{d_{i}\geq \alpha}(d_{i} - \alpha) = \delta(\alpha_{0},\boldsymbol\lambda,\boldsymbol d) - \sum_{d_i\in[\alpha,\alpha_{0})}d_i\geq \delta(\alpha_{0},\boldsymbol\lambda,\boldsymbol d) - \sum_{d_{i}>0}d_{i}>0. \]
On the other side, if \(-\alpha_{0}<\alpha<0\), then
\[\delta(\alpha,\boldsymbol\lambda,\boldsymbol d)\geq \delta(\alpha_{0},\boldsymbol\lambda,\boldsymbol d) - \sum_{d_{i}>0}|d_{i}|>0.\]
We conclude that \(\tilde{\boldsymbol\lambda}_{+}\prec\boldsymbol\lambda_{+}\) and \(\tilde{\boldsymbol\lambda}_{-}\prec\boldsymbol\lambda_{-}\). Note that each of these sequences is positive. However, both \(\boldsymbol\lambda\) and \(\tilde{\boldsymbol\lambda}\) contain a countably infinite number of zeros. Hence, combining Proposition \ref{kwd2d} and Lemma \ref{subdiag} applied to the sequences \(\boldsymbol\lambda_{+},\boldsymbol\lambda_{-},\boldsymbol 0\), where \(\boldsymbol 0\) is a sequence of a countably infinite number of zeros, we see that \(\tilde{\boldsymbol\lambda}\) is a diagonal of \(E\).

Let \(K\) and \(L\) be infinite sets such that 
\[\{i:\lambda_{i}=0\}\subset K\subset \{i : 0\leq\lambda_{i}<\alpha_{0}\},\qquad L\subset \{i : -\alpha_{0}<\lambda_{i}<0\}\] 
and
\begin{equation}\label{s7.h}
\sum_{i\in K}\lambda_{i} = \sum_{i\in L}|\lambda_{i}|<\infty.
\end{equation}

Let \(\tilde{\boldsymbol d} = \boldsymbol d\oplus (\lambda_{i})_{i\in K\cup L}\). Note that \(\tilde{\boldsymbol d}\) is a subsequence of \(\tilde{\boldsymbol\lambda}\), hence we will have
\begin{equation}\label{s7.g}\tilde{\lambda}_{i}^{+\downarrow}\geq \tilde{d}_{i}^{+\downarrow}\quad\text{and}\quad \tilde{\lambda}_{i}^{-\downarrow}\geq \tilde{d}_{i}^{-\downarrow}\quad\text{for all }i\in\N.\end{equation}
Since \(\tilde{\boldsymbol\lambda}_{+},\tilde{\boldsymbol\lambda}_{-}\notin\ell^{1}\) and \(\tilde{\boldsymbol d}_{+},\tilde{\boldsymbol d}_{-}\in\ell^{1}\), we must have strict inequalities in \eqref{s7.g} for infinitely many \(i\in\N\). By Theorem \ref{exel}, if \(E_{1}\) is any operator with diagonal \((\tilde{\lambda}_{i}^{+\downarrow})_{i\in\N}\oplus(-\tilde{\lambda}_{i}^{-\downarrow})_{i\in\N}\), then \((\tilde{d}_{i}^{+\downarrow})_{i\in\N}\oplus(-\tilde{d}_{i}^{-\downarrow})_{i\in\N}\) is also a diagonal of \(E_{1}\). By the choice of \(K\) we see that \(\tilde{\boldsymbol d}\) contains infinitely many zeros. Hence, by Lemma \ref{subdiag} applied to the sets \(I_{1} = \{i : \tilde{\lambda}_{i}\neq 0\}\), \(J_{1} = \{i : \tilde{d}_{i}\neq 0\}\), \(I_{2} = \{i : \tilde{\lambda}_{i}=0\}\), and \(J_{2} = \{i : \tilde{d}_{i}=0\}\) the sequence \(\tilde{\boldsymbol d}\) is a diagonal of \(E\).
By \eqref{s7.h} we have 
\[
\liminf_{\alpha\searrow 0}\delta(\alpha,\tilde {\boldsymbol d}, \boldsymbol d) 
=
\sum_{i\in K}\lambda_{i} = \sum_{i\in L}|\lambda_{i}|
= \liminf_{\alpha\nearrow 0}\delta(\alpha,\tilde {\boldsymbol d}, \boldsymbol d) \in (0,\infty). 
\]
Since $\tilde{\boldsymbol d}$ is a subsequence of $\boldsymbol d$, majorization holds, and by Theorem \ref{step5}, $\boldsymbol d$ is a diagonal of $E$.

\noindent\textbf{Case 3.} Assume (ii) holds. Fix \(\alpha_{0}>0\) such that
\[\delta(\alpha,\boldsymbol\lambda,\boldsymbol d)\geq 2\quad\text{for all }\alpha\in (0,\alpha_{0}).\]
Let \(I_{0}\subset\{i : \lambda_{i}\in (0,\alpha_{0})\}\) such that
\[\sum_{i\in I_{0}} \lambda_{i} = 1.\]
Set \(J_{0} = \N\setminus I_{0}\). Note that for \(\alpha\geq \alpha_{0}\) we have
\[\delta(\alpha,\boldsymbol\lambda|_{J_{0}},\boldsymbol d) = \delta(\alpha,\boldsymbol\lambda,\boldsymbol d)\geq 0.\]
For \(0<\alpha<\alpha_{0}\) we have
\[\delta(\alpha,\boldsymbol\lambda|_{J_{0}},\boldsymbol d) = \delta(\alpha,\boldsymbol\lambda,\boldsymbol d) - \sum_{\genfrac{}{}{0pt}{1}{\lambda_{i}\geq\alpha}{i\in I_{0}}}(\lambda_{i} - \alpha)\geq 2-\sum_{i\in I_{0}}\lambda_{i} = 1.\]

Fix \(\alpha_{1}\in(0,\frac{1}{2}\alpha_{0})\) such that
\[\delta(\alpha,\boldsymbol\lambda|_{J_{0}},\boldsymbol d)\geq 2\quad\text{for all }\alpha\in (0,\alpha_{1}).\]
Let \(I_{1}\subset \{i\in J_{0}:\lambda_{i}\in (0,\alpha_{1})\}\) such that
\[\sum_{i\in I_{k}}\lambda_{i} = 1.\]
Set \(J_{1} = J_{0}\setminus I_{1}\). By a similar argument as above, we have \(\delta(\alpha,\boldsymbol\lambda|_{J_{1}},\boldsymbol d)\geq 0\) for all \(\alpha>0\). 

Carrying on in this manner, we obtain a positive decreasing sequence \((\alpha_{i})_{i=0}^{\infty}\), such that \(\alpha_{k}<\frac{1}{2}\alpha_{k-1}\), and nested sets \(J_{0}\supset J_{1}\supset\cdots\). Set \(J = \bigcap_{k=0}^{\infty}J_{k}\) and \(I=\N\setminus J\). We see that
\[\delta(\alpha,\boldsymbol\lambda|_{J},\boldsymbol d)\geq 0\quad\text{for all }\alpha>0.\]
Since \(\boldsymbol d_{+}\notin\ell^1\), we conclude that \(\boldsymbol\lambda|_{J}\notin\ell^1\). Moreover,
\[\sum_{i\in I}\lambda_{i} = \sum_{k=0}^{\infty}\sum_{i\in I_{k}}d_{i} = \sum_{i=0}^{\infty}1=\infty.\]
Let \(i_{0}\in I\) such that \(\lambda_{i_{0}} = \max(\lambda_{i})_{i\in I}\). Let \(K\subset \{i : d_{i}>0\}\) be an infinite subset such that
\[\sum_{i\in K}d_{i}<\lambda_{i_{0}}.\]
By the choice of \(K\) we see that \(\delta(\alpha,\boldsymbol\lambda|_{I},\boldsymbol d|_{K})\geq 0\) for all \(\alpha>0\). Setting \(L = K\cup \{i : d_{i}\leq 0\}\) we see that
\[\delta(\alpha,\boldsymbol\lambda|_{I},\boldsymbol d|_{L})\geq 0\quad \text{for all }\alpha\neq 0.\]
Since \(\boldsymbol d|_{L}\in\ell^1\), by Case 2, if \(E_{1}\) is any self-adjoint operator with diagonal \(\boldsymbol\lambda|_{I}\), then \(\boldsymbol d|_{L}\) is also a diagonal of \(E_{1}\).

Set \(M = \{i : d_{i}>0\}\setminus K\). By Proposition \ref{kwd2d}, if \(E_{2}\) is any operator with diagonal \(\boldsymbol\lambda|_{J}\), then \(\boldsymbol d|_{M}\) is also a diagonal of \(E_{2}\). Finally, by Lemma \ref{subdiag} \(\boldsymbol d\) is a diagonal of \(E\).

\noindent\textbf{Case 4.} Assume (iii) holds. This case follow by considering the operator \(-E\), and noting that \(-\boldsymbol\lambda\) and \(-\boldsymbol d\) satisfy the assumptions of Case 3.
\end{proof}

\section{Main result for compact operators}

In this section we prove our main result, Theorem \ref{intropv2}, on diagonals of compact self-adjoint operators by combining necessity and sufficiency results from earlier sections. The precise description of the set of diagonals splits into two cases. Theorem \ref{intropv3} characterizes eigenvalue and diagonal sequences for which positive and negative excesses are not both equal to zero. In the case when the excesses \(\sigma_{+}\) and \(\sigma_{-}\) are both zero, Theorem \ref{intropv4} reduces characterization of diagonals to the problem of characterizing diagonals of compact positive operators. This is discussed in the next section.


\begin{thm}\label{intropv3}
Let $E$ be a compact operator on $\mathcal H$ with the eigenvalue list $\boldsymbol\lambda \in c_0$. Let 
$\boldsymbol d \in c_0$. Define
\[\sigma_{+} = \liminf_{n\to\infty} \sum_{i=1}^n (\lambda_i^{+ \downarrow} - d_i^{+ \downarrow})\quad\text{and}\quad \sigma_{-} = \liminf_{n\to\infty} \sum_{i=1}^n (\lambda_i^{- \downarrow} - d_i^{- \downarrow}).\]
\noindent {\rm (Necessity)} If $\boldsymbol d$ is a diagonal of $E$, then
\begin{equation}\label{p5v3}
\sigma_{+}=0 \implies \#|\{i:\lambda_{i}=0\}|\geq \#|\{i:d_{i}=0\}|
\text{\, and \,}
\#|\{i:\lambda_{i} \ge 0\}|\geq \#|\{i:d_{i} \ge0\}|,
\end{equation}
\begin{equation}\label{p6v3}
\sigma_{-}=0 \implies \#|\{i:\lambda_{i}=0\}|\geq \#|\{i:d_{i}=0\}|
\text{\, and \,}
\#|\{i:\lambda_{i} \le 0\}|\geq \#|\{i:d_{i} \le0\}|,
\end{equation}
\begin{equation}\label{p1v3}
\sum_{i=1}^n \lambda_i^{+ \downarrow} \ge 
\sum_{i=1}^n d_i^{+ \downarrow}
\qquad\text{for all }n\in \N,
\end{equation}
\begin{equation}\label{p2v3}
\sum_{i=1}^n \lambda_i^{- \downarrow} \ge 
\sum_{i=1}^n d_i^{- \downarrow}
\qquad\text{for all }n\in \N,
\end{equation}
\begin{equation}\label{p3v3}
\boldsymbol d_+ \in \ell^1
\quad \implies\quad
\sigma_{-}\geq \sigma_{+},
\end{equation}
\begin{equation}\label{p4v3}
\boldsymbol d_- \in \ell^1 
\quad \implies\quad 
\sigma_{+}\geq \sigma_{-}.
\end{equation}
\noindent {\rm (Sufficiency)} Conversely, if \eqref{p5v3}--\eqref{p4v3} hold and
\begin{equation}\label{p0v3}\sigma_{+}+\sigma_{-}>0,\end{equation}
then $\boldsymbol d$ is a diagonal of $E$. 
\end{thm}

\begin{proof}
(Necessity) 
Suppose that $E$ is a compact self-adjoint operator with eigenvalue list $\boldsymbol \lambda$ and diagonal $\boldsymbol d$ such that \eqref{p0v3} holds. Our goal is to show that 
conditions \eqref{p5v3}--\eqref{p4v3} hold. Conclusions \eqref{p5v3} and \eqref{p6v3} follow by Lemma \ref{card}. By Corollary \ref{cptschur} we have
\[
\delta(\alpha,\boldsymbol \lambda,\boldsymbol d) \ge 0 \qquad\text{for all }\alpha \ne 0.
\]
By Proposition \ref{LR} applied to sequences $(\boldsymbol \lambda_+)^\downarrow$ and $(\boldsymbol d_+)^\downarrow$ we have $(\boldsymbol d_+)^\downarrow \prec (\boldsymbol \lambda_+)^\downarrow$. Hence, \eqref{p1v3} holds. The conclusion \eqref{p2v3} follows by symmetry. 

Suppose that $\boldsymbol d_-\in \ell^1$. By Corollary \ref{cpttracec} we have
\begin{equation}\label{ee}
\liminf_{\alpha\searrow 0}\delta(\alpha,\boldsymbol\lambda,\boldsymbol{d})\geq \sum_{\lambda_{i}<0}|\lambda_{i}|-\sum_{d_{i}<0}|d_{i}|.
\end{equation}
By Proposition \ref{LR} applied to sequences $(\boldsymbol \lambda_+)^\downarrow$ and $(\boldsymbol d_+)^\downarrow$, the left-hand side of \eqref{ee} equals $\sigma_+$. If $\boldsymbol \lambda_- \in \ell^1$, then the right-hand side of \eqref{ee} equals $\sigma_-$ by Propositions \ref{LR} and \ref{LRS} applied to sequences $(\boldsymbol \lambda_-)^\downarrow$ and $(\boldsymbol d_-)^\downarrow$. If $\boldsymbol \lambda_-$ is not summable, then the right-hand side of \eqref{ee} is infinity. By \eqref{ee}, $\sigma_+=\infty$. Either way, \eqref{p4v3} holds. The conclusion \eqref{p3v3} follows by symmetry.

(Sufficiency) The sufficiency part of Theorem \ref{intropv3} is an immediate consequence of the following diagonal-to-diagonal result.
\end{proof}

\begin{thm}
Let $\boldsymbol\lambda = (\lambda_{i})_{i\in\N}$ and $\boldsymbol d = (d_{i})_{i\in\N}$ be sequences in $c_{0}$.
Suppose that \eqref{p5v3}--\eqref{p0v3} hold. If there is a self-adjoint operator $E$ with diagonal $\boldsymbol\lambda$, 
then $\boldsymbol d$ is also a diagonal of $E$.
\end{thm}

\begin{proof}
Let $\boldsymbol \lambda, \boldsymbol d\in c_0$ be such that all of the conditions \eqref{p5v3}--\eqref{p0v3} hold. Suppose that a self-adjoint operator $E$ has diagonal $\boldsymbol\lambda$.
Our goal is to show that $\boldsymbol d$ is also a diagonal of $E$.

\noindent\textbf{Case 1.} Assume that $\boldsymbol d_+, \boldsymbol d_- \in \ell^1$. By \eqref{p3v3}, \eqref{p4v3}, and \eqref{p0v3} we have
\[
\sigma_+=\sigma_->0.
\]
If $\sigma_+<\infty$, then Theorem \ref{step5} yields the conclusion. If $\sigma_+=\infty$, then we apply Theorem \ref{step7a}(iv) to reach the conclusion.

\noindent\textbf{Case 2.} Assume that $\boldsymbol d_+ \in \ell^1$ and $\boldsymbol d_- \not \in \ell^1$. The majorization \eqref{p2v3} implies that $\boldsymbol \lambda_- \not \in \ell^1$, whereas \eqref{p3v3} implies that
$
\sigma_- \ge \sigma_+ .$

Assume first that $\sigma_+>0$. If $\boldsymbol \lambda_+ \in \ell^1$, then by Proposition \ref{LR} we have
\[
\sigma_+=\liminf_{k\to\infty} \sum_{i=1}^k (\lambda_i^{+\downarrow} - d_i^{+\downarrow})= \sum_{i=1}^\infty \lambda_i^{+\downarrow} - \sum_{i=1}^\infty d_i^{+\downarrow} <\infty.
\]
Hence, Theorem \ref{step5} and Theorem \ref{step6} yield the required conclusion when $\sigma_-=\sigma_+$ and $\sigma_->\sigma_+$, respectively. 
If $\boldsymbol \lambda_+ \not\in \ell^1$, then by Proposition \ref{LR} we have
\[
\sigma_+=\liminf_{k\to\infty} \sum_{i=1}^k (\lambda_i^{+\downarrow} - d_i^{+\downarrow}) =\infty.
\]
Thus, $\sigma_-=\infty$ and we apply Theorem \ref{step7a}(iii) to reach the conclusion.

Assume next that $\sigma_+=0$. By Proposition \ref{LR} we have
\begin{equation}\label{trc}
\sum_{\lambda_i>0} \lambda_i = \sum_{d_i>0} d_i <\infty.
\end{equation}
By \eqref{p0v3} we necessarily have $\sigma_->0$. By the cardinality condition \eqref{p5v3} we have
\begin{equation}\label{trc1}
\#|\{i:\lambda_{i}=0\}|\geq \#|\{i:d_{i}=0\}|
\quad\text{and}\quad
\#|\{i:\lambda_{i} \ge 0\}|\geq \#|\{i:d_{i} \ge0\}|.
\end{equation}
Let $I_1=\{i: \lambda_i>0\}$ and $J_1=\{i:d_{i} >0\}$. Observe that \eqref{p1v3} and \eqref{trc} imply that $\#|I_1| \le \#|J_1|$. In particular, if $I_1$ is infinite, then so is $J_1$. Hence, by Proposition \ref{kwd2d}, if $E_1$ is a self-adjoint operator with diagonal $\boldsymbol \lambda|_{I_1}$, then $\boldsymbol d|_{J_1}$ is also a diagonal of $E_1$.
Note that $\lambda_i \le 0$ for all $i\in I_2:=\N \setminus I_1$ and $d_i \le 0$ for all $i \in J_2:=\N \setminus J_1$. Since $\sigma_->0$ we can apply Proposition \ref{kwnonneg} to negatives of sequences $\boldsymbol \lambda|_{I_2}$ and $\boldsymbol d|_{J_2}$ to deduce that whenever $E_2$ is a self-adjoint operator with diagonal $\boldsymbol \lambda|_{I_2}$, then $\boldsymbol d|_{J_2}$ is also a diagonal of $E_2$. By Lemma \ref{subdiag}, $\boldsymbol d$ is a diagonal of $E$.

On the other hand, if $I_1$ is finite, then by \eqref{trc1} we can find an index set $I_1'$ such that
\[
I_1 \subset I'_1 \subset \{i: \lambda_i \ge 0\}
\quad\text{and}\quad
\#|I'_1| = \#|J_1|.
\]
In addition, $I_1'$ can be chosen so that 
\[
\#|\{i \in I_2 : \lambda_{i}=0\}|\geq \#|\{i\in J_2 : d_{i} = 0\}|,
\qquad\text{where } I_2=\N \setminus I'_1, J_2=\N \setminus J_1.
\]
In fact, the above holds automatically unless $J_1$ is infinite, where extra care needs to be taken.
The Schur-Horn theorem (if $J_1$ is finite) or the finite rank Schur-Horn theorem (if $J_1$ is infinite) yields that whenever $E_1$ is a self-adjoint operator with diagonal $\boldsymbol \lambda|_{I'_1}$, then $\boldsymbol d|_{J_1}$ is also a diagonal of $E_1$. Note that $\lambda_i \le 0$ for all $i\in I_2$ and $d_i \le 0$ for all $i \in J_2$. Since $\sigma_->0$, Proposition \ref{kwnonneg} yields that whenever $E_2$ is a self-adjoint operator with diagonal $\boldsymbol \lambda|_{I_2}$, then $\boldsymbol d|_{J_2}$ is also a diagonal of $E_2$. By Lemma \ref{subdiag}, $\boldsymbol d$ is a diagonal of $E$.

\noindent\textbf{Case 3.} Assume that $\boldsymbol d_- \in \ell^1$ and $\boldsymbol d_+ \not \in \ell^1$. This is deduced from Case 2 by the symmetric argument.

\noindent\textbf{Case 4.} Assume that $\boldsymbol d_+, \boldsymbol d_- \not\in \ell^1$. The majorization conditions \eqref{p1v3} and \eqref{p2v3} imply that $\boldsymbol \lambda_+, \boldsymbol \lambda_- \not \in \ell^1$. Hence, Theorem \ref{step7a}(i) yields the required conclusion.
\end{proof}

Theorem \ref{intropv2} is now a straightforward consequence of Theorem \ref{intropv3}

\begin{proof}[Proof of Theorem \ref{intropv2}]
The equivalence of (i) and (ii) is a known consequence of the Weyl-von Neumann-Berg theorem, see \cite[Proposition 39.10]{conway}. Suppose that (ii) holds. Then, (iii) is a consequence of the necessity part of Theorem \ref{intropv3}. 

Now suppose that (iii) holds. If $\sigma_++\sigma_->0$, then the sufficiency part of Theorem \ref{intropv3} implies (ii). Indeed, adding extra zero terms to the eigenvalue sequence $\boldsymbol \lambda$ guarantees that \eqref{p5v3} and \eqref{p6v3} also hold. Hence, there exists a self-adjoint operator $T'$ with the same nonzero eigenvalues (with multiplicities) as $T$, but possibly with higher-dimensional kernel than $T$, and with diagonal $\boldsymbol d$. 

Suppose next that $\sigma_+=\sigma_-=0$. We split the sequence $\boldsymbol \lambda$ into two subsequences $\boldsymbol \lambda_0$ and $\boldsymbol \lambda_1$ consisting of negative and positive terms of $\boldsymbol \lambda$, respectively. In this splitting we disregard zero terms of $\boldsymbol \lambda$. Likewise, we split the sequence $\boldsymbol d$ into two subsequences $\boldsymbol d_0$ and $\boldsymbol d_1$ consisting of negative and positive terms of $\boldsymbol d$, respectively. Let $r$ be the number of zero terms of $\boldsymbol d$, which were disregarded in this splitting. 
By the majorization \eqref{p1v2} and $\sigma_+=0$, we see that $\boldsymbol \lambda_1$ strongly majorizes $\boldsymbol d_1$, see Definition \ref{rmajdef}. Hence, the length of $\boldsymbol \lambda_1$ is at most that of $\boldsymbol d_1$. We can also guarantee that $\boldsymbol d_1$ and $\boldsymbol \lambda_1$ have the same length by adding extra zero terms to $\boldsymbol \lambda_1$, if necessary. Therefore, there exists a compact operator $T_1$ with eigenvalue list $\boldsymbol \lambda_1$ and diagonal $\boldsymbol d_1$. This is a consequence of one of the folowing:
\begin{itemize}
\item the Schur-Horn theorem when $\boldsymbol d_1$ is finite, 
\item the finite rank Schur-Horn theorem when $\boldsymbol \lambda_1$ has only finitely many nonzero terms and $\boldsymbol d_1$ is infinite, or 
\item the Kaftal-Weiss theorem \cite[Corollary 6.1]{kw} if $\boldsymbol \lambda_1$ has  infinitely many nonzero terms and $\boldsymbol d_1$ are infinite. 
\end{itemize}
Indeed, in the final case, all of the terms of $\boldsymbol \lambda_1$ and $\boldsymbol d_1$ are positive, and hence Proposition \ref{kwd2d} applies. Likewise, there exists a compact operator $T_0$ with diagonal $\boldsymbol d_0$ and with eigenvalue list $\boldsymbol \lambda_0$, possibly extended by some zero terms. Define the operator $T'=T_0 \oplus \mathbf 0_r \oplus T_1$. Then, $T'$ has diagonal $\boldsymbol d$. Moreover, $T'$ has the same nonzero eigenvalues (with multiplicities) as $T$, but possibly with a different kernel than $T$. This proves (ii).
\end{proof}

\section{Algorithm for determining diagonals of compact operators}

In this section we present an algorithm for determining whether a numerical sequence is a diagonal of a given compact self-adjoint operator, or not. In some extreme situations the algorithm is inconclusive due to the kernel problem for positive compact operators. The best known is result about diagonals of compact positive operators is due to Loreaux and Weiss \cite{lw}. The following theorem is a convenient reformulation of their two results. The sufficiency part of Theorem \ref{kp} is \cite[Theorem 2.4]{lw}, whereas the necessity is \cite[Theorem 3.4]{lw}.

\begin{thm}\label{kp}
Let $E$ be a compact operator on $\mathcal H$ with the eigenvalue list $\boldsymbol\lambda \in c^+_0$. 
Let 
$\boldsymbol d \in c^+_0$. If  $\#|\{i:d_{i}=0\}|<\infty$, then we set
$z=   \#|\{ i : \lambda_i =0 \}| - \#|\{i:d_{i}=0\}|$; otherwise set $z=0$.

\noindent {\rm (Necessity)} If $\boldsymbol d$ is a diagonal of $E$, then
\begin{align}\label{kp1}
 \#|\{i:\lambda_{i}=0\}| &\geq \#|\{i:d_{i}=0\}|,
\\
\label{kp2}
\sum_{i=1}^n \lambda_i^{+ \downarrow} &\ge 
\sum_{i=1}^n d_i^{+ \downarrow}
\qquad\text{for all }n\in \N,
\\
\label{kp3}
\sum_{i=1}^\infty \lambda_i &= 
\sum_{i=1}^\infty d_i,
\end{align}
and for any $p\in \N$, $p\le z$, and for every $\epsilon>0$, there exists $N=N_{p,\epsilon}>0$, such that
\begin{equation}\label{kp4a}
\sum_{i=1}^n \lambda_i^{+ \downarrow} +  \epsilon \lambda_{n+1}^{+ \downarrow}
\ge 
\sum_{i=1}^{n+p} d_i^{+ \downarrow} 
\qquad\text{for all }n\ge N.
\end{equation}

\noindent {\rm (Sufficiency)} Conversely, if \eqref{kp1}--\eqref{kp3} hold and
for any $p\in \N$, $p\le z$, there exists $N=N_p$, such that
\begin{equation}\label{kp4b}
\sum_{i=1}^n \lambda_i^{+ \downarrow} 
\ge 
\sum_{i=1}^{n+p} d_i^{+ \downarrow} 
\qquad\text{for all }n\ge N,
\end{equation}
then $\boldsymbol d$ is a diagonal of $E$. 
\end{thm}

In the  case $z=0$, conditions \eqref{kp4a} and \eqref{kp4b} are vacuous, and Theorem \ref{kp} recovers a characterization of diagonals of positive compact operators with trivial kernel due to Kaftal and Weiss \cite{kw}. In the case $z=\infty$, conditions \eqref{kp4a} and \eqref{kp4b} coincide, which yields a characterization of diagonals of positive compact operators with infinite-dimensional kernel \cite[Corollary 3.5]{lw}. In the case $0<z<\infty$, the necessary condition \eqref{kp4a} is strictly weaker than the sufficient condition \eqref{kp4b}. This gap between the known necessary and sufficient conditions is referred to as the {\it kernel problem} for compact positive operators (with nontrivial finite-dimensional kernel). 

To describe our algorithm, it is convenient to introduce the concepts of decoupling and splitting of an operator. The difference between these two concepts is quite subtle, but we need to distinguish between them. Note that the concept of decoupling is a decomposition of an operator with respect to a diagonal given by a specific orthonormal basis. In contrast, splitting is just a particular direct sum decomposition of an operator.

\begin{defn}\label{decouple}
Let $E$ be a self-adjoint operator on a Hilbert space $\mathcal H$. Let $(d_{i})$ be a diagonal of $E$ with respect to an orthonormal basis $(e_{i})$ of $\mathcal H$.
We say that the operator $E$ {\it decouples at $\alpha\in \R$} with respect to $(d_{i})$ if 
\[
\mathcal H_0=\overline{\lspan }\{ e_i: d_i < \alpha \}\qquad\text{and}\qquad \mathcal H_1=\overline{\lspan} \{ e_i: d_i \ge \alpha \},
\]
are invariant subspaces of $E$ and
\[
\sigma(E|_{\mathcal H_0}) \subset (-\infty,\alpha]\qquad\text{and}\qquad\sigma(E|_{\mathcal H_1}) \subset [\alpha,\infty).
\]
\end{defn}

\begin{defn}\label{splitting} Let $E$ be a self-adjoint operator. Let $\alpha \in \R$. We say that a pair of operators $E_1$ and $E_2$ is a {\it splitting} of $E$  at $\alpha$ if there exist $z_0,z_1,z_2 \in \N \cup \{0,\infty\}$ such that:
\begin{enumerate}
\item $E_1$ is a self-adjoint operator such that $\sigma(E_1) \subset (-\infty,\alpha]$ and $z_1=\dim \ker (E_1-\alpha \bf I)$,
\item $E_2$ is a self-adjoint operator such that $\sigma(E_2) \subset [\alpha,\infty)$ and $z_2=\dim \ker (E_2-\alpha \bf I)$,
 and
\item $E$ is unitarily equivalent to $\alpha \mathbf I_{z_0} \oplus E_1 \oplus E_2$, and hence $z_0+z_1+z_2= \dim \ker (E-\alpha \bf I)$.
\end{enumerate}
\end{defn}

The following elementary fact bridges previous two concepts.

\begin{thm}\label{isplit}
Let $E$ be a self-adjoint operator.
Let $\boldsymbol d$ be a bounded sequence. 
Suppose that $\boldsymbol d$ is a diagonal of $E$ and the operator $E$ decouples at $\alpha \in \R$. Then, there exists a splitting $E_1$ and $E_2$ of $E$ at $\alpha$ such that:
\begin{enumerate}[(i)]
\item the sequence $(d_{i})_{d_i<\alpha}$ is a diagonal of $E_1$,
\item the sequence $(d_{i})_{d_i>\alpha}$ is a diagonal of $E_2$, and
\item the number of zeros in $\boldsymbol d$ satisfies
\begin{equation}\label{z102}
\dim \ker (\alpha \mathbf I - E) =  \dim \ker (\alpha \mathbf I -E_1) + \dim \ker (\alpha \mathbf I - E_2) + \#|\{i: d_i=\alpha\}|.
\end{equation}
\end{enumerate}
Conversely, if there exists a splitting of $E$ such that (i)--(iii) hold, then $\boldsymbol d$ is a diagonal of $E$ and the operator $E$ decouples at $\alpha$.
\end{thm}

\begin{proof}
Suppose that $\boldsymbol d$ is a diagonal of $E$ and the operator $E$ decouples at $\alpha \in \R$. 
Take $i\in I$ such that $d_i=\alpha$. Since $E_2 \ge \alpha \mathbf I$, the vector $e_i$ is an eigenvector of $E_2$ with eigenvalue $\alpha$. Hence, 
\[
 \overline{\lspan} \{ e_i: d_i = \alpha \} \subset \ker (\alpha \mathbf I - E_2) \subset  \ker (\alpha \mathbf I - E).
 \]
Define $\mathcal H_1'=\overline{\lspan} \{ e_i: d_i > \alpha \}$. Since $\mathcal H_1$ is an invariant subspace of $E$, so is $\mathcal H_1'$. It follows that the pair of operators $E_1= E|_{\mathcal H_0}$ and $E_2=E|_{\mathcal H_1'}$ is a splitting of $E$, and (i)---(iii) hold. Indeed, \eqref{z102} follows from the fact that $E$ is unitarily equivalent with $\alpha \mathbf I_{z_0} \oplus E_1 \oplus E_2$, where $z_0=\#|\{i: d_i =\alpha\}|$. The converse direction follows immediately from the definitions of splitting and decoupling.
\end{proof}

The following theorem describes the scenario when we may encounter the kernel problem for positive compact operators.

\begin{thm}\label{intropv4}
Let $\boldsymbol\lambda, \boldsymbol d \in c_0$. Let
\[\sigma_{+} = \liminf_{n\to\infty} \sum_{i=1}^n (\lambda_i^{+ \downarrow} - d_i^{+ \downarrow})\quad\text{and}\quad \sigma_{-} = \liminf_{n\to\infty} \sum_{i=1}^n (\lambda_i^{- \downarrow} - d_i^{- \downarrow})\]
and assume that
\begin{equation}\label{p0v4}\sigma_{+}=0 \quad\text{or}\quad \sigma_{-}=0.\end{equation}
Let $E$ be a compact self-adjoint operator with eigenvalue list $\boldsymbol\lambda$.
Then $\boldsymbol d$ is a diagonal of $E$ if and only if the following conditions hold:
\begin{enumerate}[(i)]
\item there exists a self-adjoint operator $E_1$ with eigenvalue list $(\lambda_{i})_{\lambda_i<0} \oplus 0_{z_1}$ and diagonal $(d_{i})_{d_i<0}$, where \(0_{z_{1}}\) is the sequence consisting of \(z_{1}\) terms equal to zero,
\item there exists a self-adjoint operator $E_2$ with eigenvalue list $(\lambda_{i})_{\lambda_i>0} \oplus 0_{z_2}$ and diagonal $(d_{i})_{d_i>0}$,
\item there exists $z_1,z_2 \in \N \cup \{0,\infty\}$ such that 
\begin{equation}\label{z12}
\#|\{i:\lambda_i=0\}| = z_1+z_2+ \#|\{i: d_i=0\}|.
\end{equation}

\end{enumerate}
\end{thm}

In other words, $\boldsymbol d$ is a diagonal of $E$ if and only if there exists a splitting of $E$ at $0$ into  $E_1$ and $E_2$ with $z_0= \#|\{i: d_i=0\}|$ such that $E_1$ has diagonal $(d_{i})_{d_i<0}$ and $E_2$ has diagonal $(d_{i})_{d_i>0}$. The requirement that $z_0=\#|\{i: d_i=0\}|$ is due to the fact that any basis vector corresponding to zero on diagonal $d_i=0$ belongs to the kernel of $E$.

\begin{proof}
Suppose that $E$ is a compact self-adjoint operator with eigenvalue list $\boldsymbol\lambda$ and diagonal $\boldsymbol d$. By Theorem \ref{intropv3} we have \eqref{p1v3} and \eqref{p2v3}. Hence, by Proposition \ref{LR} and \eqref{p0v4} we can apply Proposition \ref{p211} to deduce that $E$ decouples at $0$. Therefore, Theorem \ref{isplit} yields a splitting of $E$ at $0$ into  $E_1$ and $E_2$ with $z_0= \#|\{i: d_i=0\}|$ such that $E_1$ has diagonal $(d_{i})_{d_i<0}$ and $E_2$ has diagonal $(d_{i})_{d_i>0}$. That shows that (i)--(iii) hold with \eqref{z12} being a consequence of \eqref{z102}. Coversely, suppose that (i)--(iii) hold. This implies that the pair $E_1$ and $E_2$ is a splitting of $E$ at $0$, and conditions (i)--(iii) in Theorem \ref{isplit} hold. Hence, $\boldsymbol d$ is a diagonal of $E$.
\end{proof}

\begin{figure}
\resizebox{5in}{!}{
\begin{tikzpicture}[node distance=2cm]
\node (input) [io] {Input: self-adjoint operator $E$ such that $\#|\sigma_{ess}(E)|=1$ and a sequence $(d_{i})$};

\node (pro2) [process, below of=input] {Normalize $E$ to be compact};

\node (pro2a) [decision, below of=pro2, yshift=-1.5cm] {(Necessity) Theorem \ref{intropv3}};

\node (pro2b) [no, right of=pro2a, xshift=4cm] {$(d_{i})$ is not a \\ diagonal of $E$};

\node (dec2) [decision, below of=pro2a, yshift=-2.5cm] {\qquad $\sigma_+ + \sigma_->0$\qquad };

\node (pro3) [process, below of=dec2, yshift=-1.5cm] {Apply Theorem \ref{intropv4} to split $E$ into two compact operators};

\node (pro4) [stop, right of=dec2, xshift=4cm] {$(d_{i})$ is a \\ diagonal of $E$};

\node (pro5) [kernel, right of=pro3, xshift=4cm, yshift=0.5cm] {Compact kernel problem};

\node (pro6) [kernel, right of=pro3, xshift=4cm, yshift=-0.5cm] {Compact kernel problem};

\draw [arrow] (input) -- (pro2);

\draw [arrow] (pro2) -- (pro2a);
\draw [arrow] (pro2a) --  node [xshift=0.45cm] {yes} (dec2);
\draw [arrow] (pro2a) -- node [yshift=0.25cm] {no} (pro2b);

\draw [arrow] (dec2) -- node [yshift=0.25cm] {yes} (pro4);
\draw [arrow] (dec2) -- node [xshift=0.45cm] {no} (pro3);

\draw [arrow] (pro3) -- (pro5);
\draw [arrow] (pro3) -- (pro6);

\end{tikzpicture}
}
\caption{Algorithm for operators with \(1\)-point essential spectrum}
\label{diagram1}
\end{figure}
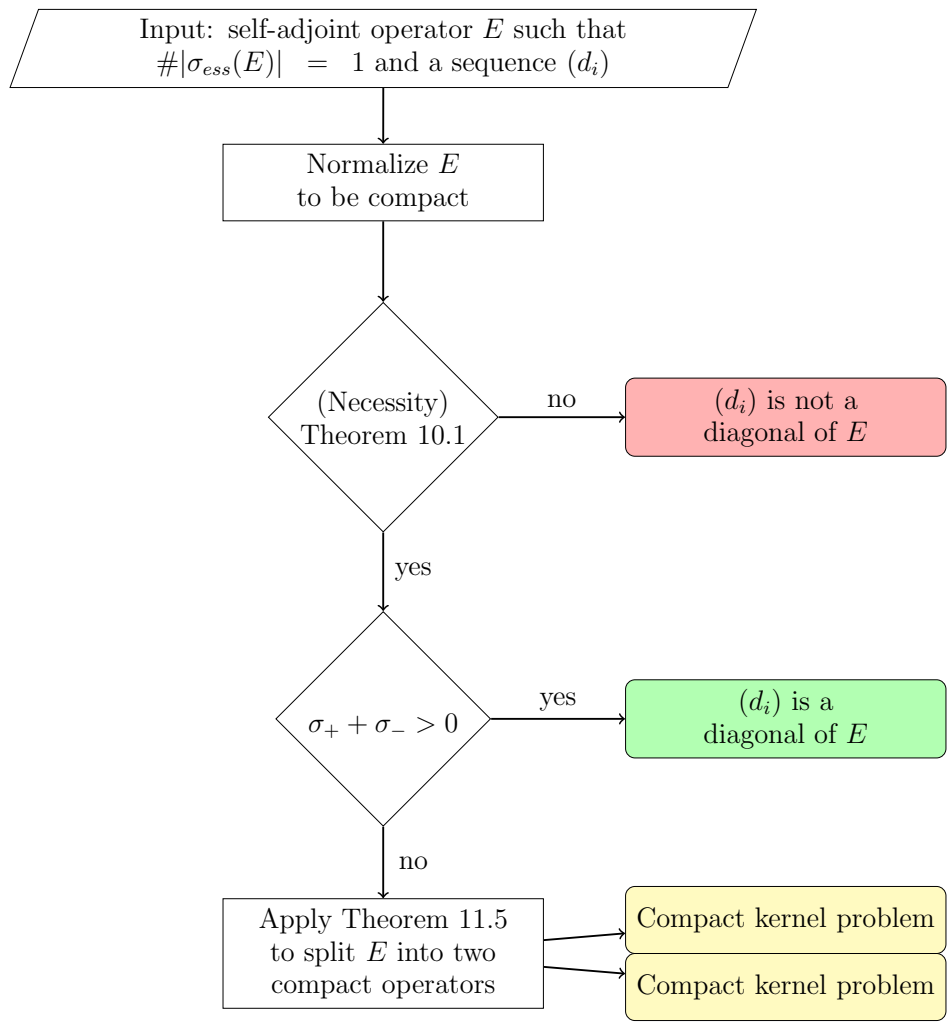

The algorithm for determining whether a sequence $\boldsymbol d$ is a diagonal of a compact operator $E$ is represented by Figure \ref{diagram1}. This procedure actually works for any self-adjoint operator with one point essential spectrum. First we check whether all necessary conditions \eqref{p5v3}--\eqref{p4v3} in Theorem \ref{intropv3} are satisfied. If not, then $\boldsymbol d$ is not a diagonal of $E$. Otherwise, $\boldsymbol d$ is a diagonal of $E$ provided $\sigma_+ + \sigma_- >0$. In the case $\sigma_+=\sigma_-=0$, Theorem \ref{intropv4} applies. In particular, we look for all possible splittings of $E$ into $E_1$ and $E_2$ at $0$ with $z_0=\#|\{i: d_i=0\}|$ and apply Theorem \ref{kp} to test whether both   
$(d_{i})_{d_i<0}$ is a diagonal of $E_1$ and  $(d_{i})_{d_i>0}$ is a diagonal of $E_2$. 
If the necessary conditions in Theorem \ref{kp} fail for all possible splittings, then $\boldsymbol d$ is not a diagonal of $E$. On the other hand, if the sufficient conditions in Theorem \ref{kp} hold for some splitting, then $\boldsymbol d$ is a diagonal of $E$. This leaves out the possibility that some splittings satisfy the necessary conditions, but all fail sufficient conditions. This is exactly the kernel problem, where the algorithm is inconclusive. 

Note that if $\dim \ker E=\infty$, then we will not encounter the kernel problem since it suffices to consider only: (a) the splitting with $z_1=z_2=0$ when $\#|\{i: d_i=0\}|=\infty$, or (b) two splittings with $z_1=0$, $z_2=\infty$, or vice versa, when $\#|\{i: d_i=0\}|<\infty$. In both scenarios Theorem \ref{kp} is conclusive. So we may encounter the kernel problem only when $\dim \ker E<\infty$, which requires testing only a finite number of possible splittings corresponding to: 
\[
z_0=\#|\{i: d_i=0\}|, \qquad z_1=0,\ldots, n, \qquad z_2=n-z_1 \qquad\text{where } n=\dim \ker E - z_0. 
\]
Hence, the algorithm requires analyzing only a finite number of splittings of $E$, if any.

\end{document}